\documentclass[a4paper,reqno]{amsart}
\pdfoutput=1

\usepackage{enumerate} 
\usepackage{enumitem}

\usepackage{amsmath,amsfonts,amssymb,amsthm,stmaryrd,bbm,graphicx,mathtools,enumerate}
\usepackage{mathrsfs}
\usepackage[utf8]{inputenc} 
\usepackage[english]{babel}
\usepackage[tracking,spacing,kerning,babel]{microtype}
\usepackage{wasysym} 
\usepackage{color}
\usepackage{xcolor}
    	\usepackage{framed} 

\usepackage{mdframed} 

\usepackage{wrapfig} 

\usepackage[final]{hyperref}   
\hypersetup{
    linktoc=page,
    linkcolor=red,          
    citecolor=blue,        
    filecolor=blue,      
    urlcolor=cyan,
   colorlinks=true           
}

\usepackage{minitoc}
\usepackage{multicol}

\newtheoremstyle{mine}
{\baselineskip}
{\baselineskip}
{\itshape}
{
}
{\bfseries}
{.}
{.5em}
{#1 #2\ifx#3\relax\else~(#3)\fi}

\theoremstyle{mine}

\newtheorem{theorem}{Theorem}

\newenvironment{ftheo}
  {\begin{mdframed}\begin{theorem}}
  {\end{theorem}\bigskip\end{mdframed}}

\numberwithin{theorem}{section}
\newtheorem{corollary}[theorem]{Corollary}
\newtheorem{proposition}[theorem]{Proposition}
\newtheorem{lemma}[theorem]{Lemma}

\numberwithin{equation}{section}

\theoremstyle{remark}
\newtheorem{remark}{Remark}

\newcommand{\ind}[1]{\mathbf{1}_{\left\{#1\right\}}}

\def\Z{\mathbb{Z}}
\def\R{\mathbb{R}}
\def\S{\mathbb{S}}
\def\n{\mathbb{N}}
\def\T{\mathbb{T}}
\def\t{\mathbf{t}}
\def\P{\mathbf{P}}
\def\Q{\mathbf{Q}}
\def\E{\mathbf{E}}
\def\L{\mathcal{L}}
\def\rt{\mathcal{R}}
\def\W{\mathbf{W}} 

\def\J{\mathcal{J}}
\def\cap{\mathbf{Cap}}
\def\bcap{\mathbf{BCap}}
\def\rad{\mathbf{rad}}
\def\Es{\mathbf{Es}}
\def\c{\mathbf{c}}
\def\C{\mathbf{C}}
\def\B{\mathbf{B}}
\def\A{\mathcal{A}}
\def\N{\mathcal{N}}

\def\supp{\mathrm{supp}}
\def\d{\mathrm{d}}

\title
[Yaglom theorem for Critical BRW on $\Z^d$]
{
Yaglom theorem for critical branching random walk on $\Z^d$
}

\author{
    Xinxin Chen\textsuperscript{1,*}
    \and 
    Shen Lin\textsuperscript{2,\textsuperscript{$\dagger$}}
}
\date{\today}

\begin{document}

\maketitle

\noindent
\textsuperscript{1}Beijing Normal University, School of Mathematical Sciences, Beijing, China, \\\textsuperscript{*}E-mail: xinxin.chen@bnu.edu.cn;\\
\textsuperscript{2}Sorbonne Université, Université Paris Cité, CNRS, Laboratoire de Probabilités, Statistique et Modélisation, LPSM, F-75005 Paris, France. \\ \textsuperscript{$\dagger$}Email: shen.lin@sorbonne-universite.fr

\vspace{0.5cm}

\noindent
*
Corresponding author: XINXIN CHEN. E-mail: xinxin.chen@bnu.edu.cn;\\
All authors contributed equally to this work.

%
%
%

\begin{abstract}
We study the critical branching random walk on $\Z^d$ started from a distant point $x$ and conditioned to hit some compact set $K$ in $\Z^d$. We are interested in the occupation time in $K$ and present its asymptotic behaviors in different dimensions. It is shown in this work that the occupation time is of order $\|x\|^{4-d}$ in dimensions $d\le3$, of order $\log\|x\|$ in dimension $d=4$, and of order $1$ in dimensions $d\ge 5$. The corresponding weak convergences are also established. These results answer a question raised by Le Gall and Merle (\emph{Elect.~Comm.~in Probab.} 11 (2006), 252–265). 

\end{abstract}

\maketitle


\section{Introduction}

The study of critical branching random walks (CBRW for short) has long been a fertile ground for probabilistic research, offering rich connections to various aspects of random spatial processes. In contrast to supercritical systems that exhibit exponential growth, or subcritical ones that face rapid extinction, the critical case, where each particle produces exactly one offspring on average, displays a delicate balance that often aligns its behavior with that of ordinary random walks.

This connection becomes particularly evident when examining range asymptotics. While the range of ordinary centered random walks on $\mathbb{Z}^d$ has been extensively studied since the pioneering work of Dvoretzky and Erd\H{o}s \cite{DE1951}, recent developments have revealed analogous behavior for CBRW. In particular, Le Gall and Lin \cite{LeGall-Lin,LeGall-Lin16} established that the size of the range of CBRW follows asymptotic behaviors similar to those of ordinary random walks, with a shift in the critical dimension from 2 to 4. This dimensional shift reflects the fundamental difference in recurrence-transience dichotomy between the two processes: while dimension 2 marks the critical threshold for ordinary random walks, dimension 4 plays the corresponding role for CBRW, as shown in Benjamini and Curien \cite{BC12} and Legrand, Sabot and Schapira \cite{LSS24}.

A similar pattern has also been observed for capacity asymptotics. For the capacity of CBRW ranges, the critical dimension shifts further to $d=6$, as demonstrated in Bai and Wan \cite{BW22} and Bai and Hu \cite{BH22,BH23}. This contrasts with the capacity of ordinary random walk ranges, where dimension 4 is critical, as established in Jain and Orey \cite{JO68}, Asselah, Schapira and Sousi \cite{ASS2018,ASS2019}, and Chang \cite{Chang}.

The development of a discrete capacity theory for critical branching systems has seen significant advances through Zhu's series of works \cite{Zhu+, Zhu2017, Zhu}, which introduced the concept of branching capacity. This framework was further used in \cite{Zhu19+} to construct a model of branching interlacements on $\mathbb{Z}^d$ for $d\geq 5$, which turns out to be the local limit of CBRW on a discrete torus, conditioned on the size proportional to the volume of the torus.

Recent developments continue to refine our understanding of these phenomena. Schapira \cite{Schapira23} obtained a law of large numbers for the branching capacity of random walk ranges in $\mathbb{Z}^d$ ($d\geq 6$), with a logarithmic correction in the critical dimension 6. For the five-dimensional case, Bai, Delmas and Hu \cite{BDH24,BDH24+} proved convergence of the renormalized branching capacity to the Brownian snake capacity of the range of a Brownian motion. Most recently, Baran \cite{Baran} has extended this line of inquiry to the branching capacity of CBRW ranges, where the critical dimension is expected to be 8, establishing a weak law of large numbers with logarithmic corrections in $\mathbb{Z}^8$. In the special case of a stationary CBRW with geometric offspring, Bai and Hu \cite{BH25} have shown a central limit theorem with Gaussian limiting
distribution for the size of the range in sufficiently high dimensions $d>16$.

\subsection{Context} $ $

In the opening paragraph of \cite{LGM06}, Le Gall and Merle raised the following question concerning CBRW, which we quote verbatim:

\medskip

\noindent
\textit{Consider a population of branching particles in $\Z^d$, such that individuals move independently in discrete time according to a random walk with zero mean and finite second moments, and at each integer time individuals die and give rise independently to a random number of offspring according to a critical offspring distribution. Suppose that the population starts with a single individual sitting at a point $x\in\Z^d$ located far away from the origin, and condition on the event that the population will eventually hit the origin. Then what will be the typical number of individuals that visit the origin, and is there a limiting distribution for this number?}
\medskip

The present work is devoted to answering this question in all dimensions.

First, we introduce a probability distribution $\{p_k\}_{k\ge0}$ on $\Z_+$ as the offspring law, satisfying $p_1<1$, and
\begin{equation}\label{offspring}
\sum_{k\ge0}kp_k=1,\quad \sigma^2:=\sum_{k\ge 0}(k-1)^2 p_k\in(0,\infty).
\end{equation}
Secondly, we take a probability distribution $\mu$ on $\Z^d$ to be the jump law. 
Throughout this work, we always assume that $\mu$ has zero mean and is not supported on a strict subgroup of $\Z^d$. We also assume that the covariance matrix $\Gamma$ of $\mu$ exists, and is positive definite. 

Now we are ready to construct the CBRW $\{(u, S_u)\}_{u\in\T}$ on $\Z^d$. We start with one ancestor $\rho$ located at $S_\rho=x\in \Z^d$. At time $1$, $\rho$ dies and gives birth to a random number of children according to the offspring law $\{p_k\}_{k\ge0}$, each child $u$ makes a random jump from $x$ according to $\mu$ independently and the resulting position of $u$ is denoted by $S_u$. The system goes on in this way that each individual $v$ born at time $n$ dies at time $n+1$ and produces a random number of children according to $\{p_k\}_{k\ge0}$, each of its children immediately makes an independent jump from the position of $v$ according to $\mu$. 

Naturally, we obtain a genealogical tree $\T$ rooted at $\rho$, which is distributed as a critical Bienaym\'e--Galton--Watson tree since the offspring law has mean $1$. We denote by $\#\T$ the number of vertices in $\T$. For every individual $u\in \T$, let $S_u$ denote its spatial position in $\Z^d$ and let $|u|$ denote its generation. For any $u,v\in\T$, we write $v\le u$ if $u$ is a descendant of $v$, and write $v< u$ if $v\le u$ and $v\neq u$. The law of the CBRW $\{(u, S_u), u\in\T\}$ is given by $\P$. We denote $\P_x$ for $\P(\cdot\vert S_\rho=x)$ for any $x\in\Z^d$, and the corresponding expectation is denoted by $\E_x$.

For any $n\in\n$, we define the occupation measure $Z_n(\cdot)$ at time $n$ by
\[
Z_n(B):=\sum_{|u|=n}\ind{S_u\in B}, \forall B\subset\Z^d,
\]
with the convention that $\sum_\varnothing=0$. 
The total occupation measure of the whole system is then defined by
\[
Z_\T(B):=\sum_{n\ge0}Z_n(B)=\sum_{u\in\T}\ind{S_u\in B}, \forall B\subset\Z^d.
\]
We sometimes call $Z_\T(B)$ the occupation time of the CBRW at $B$. For convenience, we will write $Z_\T(y)$ for $Z_\T(\{y\})$, and write $Z_n$ for $Z_n(\Z^d)$. 
As $\T$ is a critical BGW tree, the system becomes extinct almost surely.
A well-known theorem of Kolmogorov says that the probability $\P(Z_n>0)$ that the critical branching system survives to generation $n$ is asymptotically equivalent to $\frac{2}{n\sigma^2}$, as $n\to \infty$.

When the ancestor $\rho$ is located at a point $x\in \Z^d$ far away from the origin, we know that with high probability the branching system dies out before some individual arrives in a non-empty finite set $K\subset \Z^d$. This means that
\[
\P_x(Z_\T(K)\ge1)\to 0,\textrm{ as } \|x\|\to\infty,
\]
where $\|\cdot\|$ denotes the Euclidean norm of vectors in $\R^d$. 

As mentioned above, the primary aim of this work is to study the asymptotic behavior of the occupation time $Z_\T(K)$ conditioned on the event $Z_\T(K)\ge1$. Le~Gall and Merle have investigated in~\cite{LGM06} the analogous question for super-Brownian motion in $\R^d$, the continuous counterpart of the discrete-time CBRW, with the compact set $K$ replaced by the unit ball. Our main results provide discrete analogues of Theorem~1 in \cite{LGM06}. In contrast to their techniques based on the method of moments, we present a different approach, relying on spinal decomposition (for $d\geq 5$), the structure of the reduced tree (for $d=4$) and the relation between CBRW and its scaling limit (for $d\leq 3$). 

To state our results more precisely, let us first recall some basic facts about ordinary random walk on $\Z^d$. Under $\P$, let $(S_n)_{n\ge0}$ be a random walk on $\Z^d$ with i.i.d.~increments $X_n=S_{n}-S_{n-1}$, $n\ge1$, which are all distributed as $\mu$. Again, we write $\P_x((S_n)\in\cdot)=\P((S_n)\in\cdot\vert S_0=x)$. For $x\in \R^d$, we set the norm $\J(x):=\sqrt{ x\cdot \Gamma^{-1}x/d}$ with $\Gamma$ being the covariance matrix of $\mu$.

When the dimension $d\ge3$, the random walk is transient. We define its Green function by
\[
g(x,y):=\sum_{n\ge0} \P_x(S_n=y), \quad \forall x,y\in \Z^d.
\]
For any finite subset $K\subset \Z^d$, let $g(x,K)=\sum_{y\in K}g(x,y)$. Clearly, $g(x,y)=g(0,y-x)$. When the jump law $\mu$ is symmetric, $g(x,y)=g(y,x)$. Moreover, if the jump law is assumed to have finite support, we know (see e.g.~Theorem 4.3.1 of \cite{Lawler-Limic}) that as $\|x\|\to\infty$,
\begin{equation}\label{green}
g(0,x)= \frac{\c_d}{\J(x)^{d-2}}+ O\left(\frac{1}{\|x\|^d}\right),
\end{equation}
where $\c_d=\frac{\pmb{\Gamma}(d/2)}{d^{d/2-1}(d-2)\pi^{d/2}\sqrt{\det \Gamma}}$.
Under the weaker assumption that the jump law has a finite $(d-1)$-th moment, we have (see e.g.~Theorem 3.5 of \cite{LeGall}) 
\begin{equation}\label{green-not-finite-support}
g(0,x)= \frac{\c_d}{\J(x)^{d-2}}(1+o_x(1))
\end{equation}
as $\|x\|\to\infty$.

Next, it is well-known (see for example Section 6.5 of \cite{Lawler-Limic}) that the capacity of a finite subset $K\subset \Z^d$ has the following two representations:
\[
\cap(K):=\sum_{y\in K}\underbrace{\P_y(T^+_K=\infty)}_{\Es_K(y)}=\lim_{\|x\|\to \infty}\frac{\P_x(T_K<\infty)}{g(x,0)},
\]
in terms of the hitting time $T_K:=\inf\{n\ge0 \colon S_n\in K\}$ and the first return time $T^+_K:=\inf\{n\ge1 \colon S_n\in K\}$. Here, $\Es_K(y)$ stands for the escape probability that the random walk starting from $y\in K$ never returns to $K$.

For any nonempty finite subset $K$ of $\Z^d$, asymptotics of the hitting probability $\P_x(Z_\T(K)\ge1)$ for CBRW have been studied at first by Le Gall and Lin in \cite{LeGall-Lin} and \cite{LeGall-Lin16} for each dimension $d\geq 1$. The following precise results for dimensions $d\ge5$ and $d=4$ are established in subsequent works of Zhu \cite{Zhu+, Zhu19, Zhu}. 

\begin{itemize}
\item When $d\ge5$, under the \emph{weak $L^d$} assumption that there exists a constant $C>0$, such that for any $r\geq 1$,
\[
	\mu(\{x\in\Z^d\colon \|x\|>r\})< C\cdot r^{-d},
\]
Theorem 1.1 of \cite{Zhu+} shows that 
\begin{equation}\label{5duK}
\lim_{\|x\|\to\infty}\J(x)^{d-2}\P_x(Z_\T(K)\ge 1)=\c_d \cdot\bcap(K)\in(0,\infty),
\end{equation}
where $\bcap(K)$ is the branching capacity defined in \cite{Zhu+} as
\[
	\bcap(K):=\sum_{y\in K}\widetilde\Es_K(y),
\]
where, by analogy, $\widetilde\Es_K(y)$ is the escape probability that some infinite version of CBRW starting at $y\in K$ never returns to $K$ (see \cite[Definition 3.1]{Zhu+} for a precise definition).
Note that \eqref{5duK} can be rewritten as 
\[
	\bcap(K)=\lim_{\|x\|\to \infty}\frac{\P_x(Z_\T(K)\ge 1)}{g(x,0)}.
\]
See also Asselah, Schapira and Sousi \cite{ASS2023} for an approximate variational characterization of the branching capacity. For other connections between capacity and branching capacity, see for example Asselah, Okada, Schapira and Sousi \cite{AOSS} and Bai, Chen and Peres \cite{BCP25}.
Moreover, Theorem 1.2 of \cite{Zhu+} says that conditionally on $\{Z_\T(K)\ge1\}$, the intersection between $K$ and the range $\{S_u, u\in\T\}$ converges in distribution, as $\|x\|\to \infty$.
\item When $d=4$, under the assumption that for some $\lambda>0$,
\[
	\sum_{x\in \Z^4} \mu(x)\cdot \exp(\lambda\|x\|)< \infty,
\]
Theorem 1.1 of \cite{Zhu} proves that
\[
\lim_{\|x\|\to\infty}\big(\J(x)^{2}\log\J(x)\big)\P_x(Z_\T(K)\ge 1)=\frac{1}{2\sigma^2}.
\]
\item When $d\leq 3$, under the assumption that
\begin{equation}\label{eq:assumption-moment4}
	\lim_{r\to\infty}r^4\mu(\{x\in\Z^d\colon \|x\|>r\})=0,
\end{equation}
Theorem 7 of \cite{LeGall-Lin} states that
\[
\lim_{\|x\|\to\infty} \J(x)^2\P_x(Z_\T(0)\ge1)=\frac{2(4-d)}{\sigma^2 d}.
\]
In fact, the original arguments in \cite{LeGall-Lin} can be easily adapted to show the same result for the hitting probability of any nonempty finite set $K$. Let us mention that the proof was initially built on the strong convergence of discrete snakes to the Brownian snake, established in Janson and Marckert \cite{JM} under the finite exponential moment condition on the offspring distribution $\{p_k\}_{k\geq 0}$. This extra assumption has been shown recently by Marzouk \cite{Marzouk} to be unnecessary, and may be weakened to a finite second moment assumption. However, the condition \eqref{eq:assumption-moment4} regarding the spatial displacement is necessary for the strong convergence of discrete snakes to the Brownian snake (see Theorem 2 in \cite{JM}).
\end{itemize}

It reveals that dimension 4 is the critical dimension for the behavior of the hitting probability. 
In fact, when $d<4$, we only need the genealogical tree $\T$ to be large enough (survival up to generation of order $\J(x)^2$) to get outside of the ball $\B(x,r):=\{y\in\Z^d \colon \J(y-x)\le r\}$ with $r=\J(x)$, then with positive probability, the finite set $K$ is attained by some individual. This strategy no longer works when $d\ge4$. In this case, even if we force the genealogical tree $\T$ to be large and the system reaches the boundary of $\B(x,r)$, there is only a small proportion of the sites in $\B(x,r)$ being visited. 

If the CBRW starts with $n$ ancestor particles at time 0, according to Watanabe \cite{Watanabe}, under suitable hypotheses on the initial distribution of ancestor particles, as $n\to \infty$, the measure-valued process naturally associated with the CBRW converges, after being properly rescaled, to the super-Brownian motion $\mathbf{X}_t$ with diffusion parameter~$\sigma^2$. In dimensions 2 and higher, the random measure $\mathbf{X}_t$ is, for each $t > 0$, almost surely singular with respect to the Lebesgue measure on $\R^d$. However, Sugitani \cite{Sugitani} showed that in dimensions 2 and 3, the occupation measure 
\[
\mathbf{L}_t :=\int_0^t \mathbf{X}_s \mathrm{d}s
\] 
up to time $t$ of the super-Brownian motion is absolutely continuous, and there is a jointly continuous version $L_t(x)$ of the density process. (For dimensions 4 and higher, this is no longer true.) Within this setting, Theorem 1 in Lalley and Zheng \cite{LZ10} says that in dimensions 2 and 3, the rescaled local time density process of CBRW converges to the local time density process $L_t(x)$ of the super-Brownian motion. 

Back to our model of CBRW starting with a single ancestor, a similar invariance principle has been established in Theorem 4 of \cite{LeGall-Lin}, which says that in dimension $d\leq 3$, the process 
$$\Big(n^{\frac{d}{4} -1}\,Z_\T(\lfloor n^{1/4}x \rfloor)\Big)_{x\in \R^d\backslash\{0\}}$$
under $\P_x(\cdot\vert \#\T= n)$, converges as $n\to \infty$, to some limiting density process $(\ell^x)$ (up to scaling constants and a linear transformation of the variable $x$), at least in the sense of weak convergence of finite-dimensional marginals. Moreover, the limit $(\ell^x, x\in \R^d)$ is the continuous density process of a random probability measure on $\R^d$, the so-called \emph{Integrated Super-Brownian Excursion} (ISE for short). In dimension $d=1$, this invariance principle (in its functional form) has been obtained earlier by Bousquet-M\'elou and Janson \cite[Theorem 3.6]{BMJanson} and Devroye and Janson \cite[Theorem 1.2]{DevroyeJanson}. 

\subsection{Main results} $ $

We consider all the individuals of CBRW that hit a nonempty finite subset $K\subset \Z^d$ for the first time. 
Let $\L_K$ be the collection of such ``pioneers'' (name adopted from \cite{AS2022} and \cite{BHJ23}), that is
\[
\L_K:=\{u\in\T \colon S_u\in K; \forall \rho\le v < u, S_v\notin K \},
\]
where $\rho\le v< u$ means that $v$ (different from $u$) lies on the unique path in $\T$ connecting $\rho$ and $u$. 
In particular, if $S_\rho\in K$, we set $\L_K=\{\rho\}$. This is a stopping line of the branching random walk. Then, observe that
\begin{equation}
\label{eq:pioneer-decomp}
Z_\T(K)=\sum_{u\in\L_K}\sum_{v: u\le v}\ind{S_v\in K}.
\end{equation}
On the other hand, by conditioning on the number of offspring at each generation, we immediately have 
\[
\E_x[Z_\T(K)]=g(x,K).
\]
Let $L_K:=|\L_K|$ be the cardinality of $\L_K$, i.e.~the number of pioneers in
$K$.

\bigskip

To state our main theorems, let us start with the 4-dimensional case. 

\bigskip

\begin{ftheo}\label{4d}
When $d=4$, we assume \eqref{offspring} and that $\mu$ is symmetric and has finite support. If $K\subset\Z^d$ is compact, then under $\P_x(\cdot\vert Z_\T(K)\ge1)$, the following joint convergence in law holds:
\[
\frac{1}{2\sigma^2 \c_4 \log \J(x)}\left(\frac{L_K}{\cap(K)}, \frac{Z_\T(K)}{|K|}\right)\xrightarrow[\|x\|\to\infty]{\mathrm{(d)}} (1,1)\times Y,
\]
where $Y$ has exponential distribution with parameter $1$, and the constant $\c_4=\frac{1}{8\pi^2\sqrt{\det \Gamma}}$.
\end{ftheo}

\medskip

This is reminiscent of the classical theorem of Yaglom, according to which the number $Z_n$ of particles at generation $n$ divided by $\sigma^2 n/2$, conditionally on the event that the branching process survives to generation $n$, converges in distribution to an exponential random variable with parameter $1$.

Here we choose the jump law to be symmetric with finite support for the sake of technical simplicity. The same result is expected to hold for more general jump laws.

The classical Yaglom's exponential limit law in the case of the critical Bienaym\'e--Galton--Watson tree has several proofs. Yaglom \cite{Yaglom} proved it (under a third moment assumption) via the Laplace transform of the process by analyzing generating functions. Lyons, Pemantle and Peres \cite{LPP} gave a probabilistic proof using the celebrated size-biased BGW tree (also known as the spinal decomposition). Later, Geiger \cite{Geiger} presented another proof of Yaglom’s theorem based on a distributional equation which explains the appearance of the exponential law in the limit. More recently, Ren, Song and Sun \cite{RSS} developed another proof using a two-spine decomposition technique. See also Cardona-Tob\'on and Palau \cite{CP} about Yaglom’s limit for critical Galton–Watson processes in varying environment. The rate of convergence of the classical Yaglom's limit with respect to the Wasserstein metric can be found in Pek\"oz and R\"ollin \cite{PR} and Cardona-Tob\'on, Jaramillo and Palau \cite{CJP}.

In the past few years, Yaglom-type limits have also been proven in various extended settings, including critical non-local branching Markov processes \cite{HHKW}, branching Brownian motion with absorption \cite{MS}, branching diffusions in bounded domains \cite{Pow}, and critical branching
processes in a random Markovian environment with finite state space \cite{GLLP}.

We should compare Theorem \ref{4d} with Theorem 5 in Lalley and Zheng \cite{LZ11}, which says that, for all $\Z^d, d\geq 3$, conditionally on survival to generation $n$, the number of occupied sites in generation $n$, together with the number of sites occupied by $j$ generation-$n$ particles, $j\geq1$, jointly converges in distribution as $n$ goes to infinity, to a deterministic multiple of  a single exponential random variable coming from the classical Yaglom theorem recalled above. In this context, a recent work of Pek\"oz, R\"ollin and Ross \cite{PRR20} has given a rate of convergence in the Wasserstein metric for all $d\geq 3$, and even a second order fluctuation for $d\geq 7$.

By use of Theorem~\ref{4d} and an $L^2$ estimate \eqref{sameinK}, we get the following corollary, suggesting that conditionally on $K$ being hit, all sites inside $K$ are occupied in a fairly uniform manner.

\begin{corollary}\label{cor-4dim}
When $d=4$, we assume \eqref{offspring} and that $\mu$ is symmetric and has finite support. If $K\subset\Z^d$ is compact, then under $\P_x(\cdot\vert Z_\T(K)\ge1)$, the following joint convergence in law holds: 
\[
\left( \frac{Z_\T(y)}{2\sigma^2\c_4\log \J(x)}, y\in K\right)\xrightarrow[\|x\|\to\infty]{\mathrm{(d)}} \underbrace{(1,\cdots,1)}_{\in \R^{|K|}}\times Y,
\]
where $Y$ is as in Theorem \ref{4d}.
\end{corollary}

Now we turn to the high dimensional case $d\geq 5$.
\medskip

\begin{ftheo}\label{highd}
When $d\geq 5$, we assume \eqref{offspring} and that there exists a constant $C>0$, such that for any $r\geq 1$,
\begin{equation}
  \mu(\{x\in\Z^d\colon \|x\|>r\})< C\cdot r^{-d}. 
  \label{eq:weakLd}
\end{equation}
If $K\subset\Z^d$ is compact, then the following joint convergence in law holds: under $\P_x(\cdot\vert Z_\T(K)\ge1)$,
\[
\left(L_K, Z_\T(K)\right)\xrightarrow[\|x\|\to\infty]{\mathrm{(d)}} (Y_{d,K},Z_{d,K}),
\]
where the limiting random vector $(Y_{d,K},Z_{d,K})$ is non-degenerate and its law depends on $d$ and $K$. 
\end{ftheo}

\smallskip
Note that \eqref{eq:weakLd} implies a finite $(d-1)$-th moment of $\mu$. Moreover, \eqref{eq:weakLd} holds if $\mu$ has a finite $d$-th moment. In the limit, the joint law of $(Y_{d,K}, Z_{d,K})$ also depends on the offspring distribution and the jump law. A full description can be found in Remark \ref{joint-limit}. In particular, the marginal law of $Y_{d,K}$ will be given in \eqref{limitlaw5d}.

\begin{remark}
From the proof of the previous theorem, one can also deduce that under $\P_x(\cdot\vert Z_\T(K)\ge1)$, as $\|x\|\to\infty$,
\[
(Z_\T(y),y\in K) 
\]
converges in law to a random vector $(Z_{d,K}(y), y\in K)$ as $\|x\|\to\infty$. Contrary to the 4-dimensional case, $Z_{d,K}(y), y\in K$ are no longer identical for $d\geq 5$.
\end{remark}

\medskip

For the low dimensional case $d\leq 3$, we consider the Brownian snake under the (infinite) excursion measure $\mathbf{N}_0$. If the Brownian snake is conditioned to have duration 1, its total occupation measure under the conditional probability measure $\mathbf{N}_0^{(1)}$ is the ISE mentioned above. According to Proposition 1 in \cite{LeGall-Lin}, both $\mathbf{N}_0$-a.e.~and $\mathbf{N}_0^{(1)}$-a.s., the total occupation measure has a continuous local time density process on $\R^d$, which will be denoted by $(\ell^x, x\in \R^d)$. See \cite[Section 2.3]{LeGall-Lin} and the references therein for more details about the Brownian snake.
By spherical symmetry of the Brownian snake, for any $\vartheta$ on the sphere $\mathbb{S}^{d-1}=\{x\in \R^d: \|x\|=1\}$, $\ell^\vartheta$ has the same distribution under the excursion measure $\mathbf{N}_0$. Let $\ell_+$ be a positive random variable such that 
\begin{equation}\label{localtime+}
\P(\ell_+\in\cdot)=\mathbf{N}_0(\ell^\vartheta\in\cdot \,\vert\, \ell^\vartheta>0)
\end{equation}
for an arbitrary point $\vartheta\in\mathbb{S}^{d-1}$.

\medskip
\begin{ftheo}\label{lowd}
When $d\leq 3$, we assume \eqref{offspring} and  \eqref{eq:assumption-moment4}.
If $K\subset\Z^d$ is compact, then under $\P_x(\cdot\vert Z_\T(K) \geq 1)$, the following convergence in law holds:
\[
\frac{Z_\T(K)}{|K|\J(x)^{4-d}}\xrightarrow[\|x\|\to\infty]{\mathrm{(d)}} \frac{\sigma^2}{4\sqrt{\det\Gamma}}d^{\frac{4-d}{2}}\ell_+ ,
\]
where $\ell_+$ is defined in \eqref{localtime+}.
\end{ftheo}

\medskip

Applying Proposition 2 in \cite{LeGall-Lin}, we can verify that, up to a multiplicative constant, the limiting distribution in the previous theorem is the same one that appeared in Theorem 1, part (i) of \cite{LGM06} for $d\leq 3$, in the context of super-Brownian motion.

\begin{remark}
It follows from the proof of Theorem \ref{lowd} that under $\P_x(\cdot\vert Z_\T(K)\ge1)$, as $\|x\|\to\infty$, we also have
\[
\left( \frac{Z_\T(y)}{\J(x)^{4-d}}, y\in K\right)\xrightarrow[\|x\|\to\infty]{\mathrm{(d)}} \underbrace{(1,\cdots,1)}_{\in \R^{|K|}}\times \frac{\sigma^2}{4\sqrt{\det\Gamma}}d^{\frac{4-d}{2}}\ell_+.
\]
This is the analog of Corollary \ref{cor-4dim} for $d\leq 3$.
\end{remark}

\subsection{Further discussions} $ $

For convenience, we restrict to $K=\{0\}$ for illustration, and we take $\mu$ to be the simple random walk jump law on $\Z^d$. It is proved by Angel, Hutchcroft and J\'arai \cite{AHJ21} that, assuming the existence of some finite exponential moment for $\{p_k\}_{k\geq 0}$, we have
\[
\P_0(Z_\T(0)\ge n) =
\begin{cases}
 \Theta\bigl(n^{-2/(4-d)}\bigr) & \qquad \mathrm{if} \quad d\leq 3;\\
\exp\left(-\Theta (\sqrt{n})\right)  & \qquad \mathrm{if} \quad d=4;\\
\exp\left( -\Theta(n)\right) & \qquad \mathrm{if} \quad d \geq 5.
\end{cases}
\]
Here the notation $f(n)=\Theta(g(n))$ for every $n\geq 1$ means that there exist positive constants $c$ and $C$ depending only on the offspring distribution $\{p_k\}_{k\geq 0}$ and the dimension $d$ such that $cg(n) \leq f(n) \leq Cg(n)$ for every $n\geq 1$. More general results for the tail of the time spent in a ball (or in a finite collection of balls) by CBRW can be found in Asselah and Schapira \cite{AS2022} and Asselah, Schapira and Sousi \cite{ASS2023}.


Let \(Z_x(0)\) be a random variable distributed as the occupation time at \(0\), \(Z_\T(0)\), of a CBRW \(\{(u, S_u)\}_{u\in\T}\) started at \(S_\rho=x\) in \(\Z^d\). Similarly, let \(L_0^x\) be a random variable distributed as \(L_0\), the number of pioneers in \(\{0\}\), for a CBRW started at \(x\). By \eqref{eq:pioneer-decomp}, we have the following distributional identity:
\[
Z_x(0)\overset{(\mathrm{d})}{=}\sum_{j=1}^{L^x_0}Z^{(j)}_0(0),
\]
where, given \(L^x_0\), the variables \(\{Z^{(j)}_0(0)\}_{1\le j\le L^x_0}\) are i.i.d. copies of \(Z_0(0)\). When \(d\ge3\), since \(\E[Z_0(0)]=\E_0[Z_\T(0)]=g(0,0)<\infty\), the law of large numbers implies that \(L_0^x\) and \(Z_x(0)\) are of the same order of magnitude. However, for \(d\le 2\), the order of magnitude of \(L_0^x\) is strictly smaller than that of \(Z_x(0)\). In fact, if $d=1$, thanks to the fact that $\P_0(Z_\T(0)\ge n)=\Theta(n^{-2/3})$, $L_0^x$ should be of order $\|x\|^2$ so that $Z_x(0)$ is of order $\|x\|^3$. 
When $d=2$, using the Laplace transform and a Tauberian theorem, we expect that $L_0$ is of order $\frac{\|x\|^2}{\log \|x\|}$ under $\P_x(\cdot\vert Z_\T(0)\ge1)$.

Loosely speaking, we have the following table \ref{tab:asym} that summarizes the previous discussions. Here $c_d', \ell_d, d\geq 1$ are positive constants depending on the dimension $d$.
\begin{table}[hbtp]
\centering
\caption{Asymptotics of $(L_0,Z_\T(0))$}
\label{tab:asym}
\begin{tabular}{|c|c|c|c|c|c|}
\hline
dimension & CBRW & CBRW & Order of $(L_0,Z_\T(0))$ & Tail dist.\\
$d$&$\P_x(Z_\T(0)\ge1)$&$\E_x[L_0\vert Z_\T(0)\ge1]$& under $\P_x(\cdot\vert Z_\T(0)\ge1)$ &$\P_0(Z_\T(0)\ge n)$\\
\hline
1&$c_1'\|x\|^{-2}$&$\ell_1\|x\|^2$& $(\Theta(\|x\|^2), \Theta(\|x\|^3))$ & $\Theta(n^{-2/3})$\\
\hline
2&$c_2'\|x\|^{-2}$&$\ell_2\|x\|^2$& $(\Theta(\frac{\|x\|^2}{\log \|x\|} ), \Theta( \|x\|^2 ) )$ & $\Theta(n^{-1})$\\
\hline
3&$c_3'\|x\|^{-2}$&$\ell_3\|x\|$& $(\Theta( \|x\| ), \Theta( \|x\| ))$ & $\Theta(n^{-2})$\\
\hline
4&$\frac{c_4'}{\|x\|^2\log \|x\|}$& $\ell_4\log\|x\|$ & $(\Theta( \log \|x\| ), \Theta( \log\|x\| ))$ & $e^{-\Theta(\sqrt{n})}$\\
\hline
$\ge5$&$c_d'\|x\|^{2-d}$&$\ell_d$&  $(\Theta(1),\Theta(1))$  &  $e^{-\Theta(n)}$\\
\hline
\end{tabular}
\end{table}

To finish, let us mention a recent work of Berestycki, Hutchcroft and Jego \cite{BHJ23} exploring the (pioneer) thick points for \emph{4-dimensional critical branching Brownian motion}. They have identified explicitly the exponent for the large deviation probability that a ball is hit by an exceptionally large number of pioneers, conditionally on hitting the ball. Quite remarkably, in each dimension $d$ an infinite-order asymptotic expansion is given for the probability that critical branching Brownian motion hits a distant unit ball. Using a strong coupling between tree-indexed Brownian motion and tree-indexed random walk, they have also derived some analogous results for the maximal local time and local time thick points of 4-dimensional CBRW.

\subsection{Outline of the paper} $ $

The rest of this article is organized as follows. In Section \ref{lems}, we first recall the spinal decomposition and  collect some technical lemmas about random walk on $\Z^d$. 

In Section \ref{hd}, we will study $L_K$ and $Z_\T(K)$ for $d\ge5$ and establish Theorem \ref{highd} via the spinal decomposition. The limiting distribution will be identified at the end of the proof.

In Section \ref{cd}, we will deal with the critical dimension $d=4$ and prove Theorem~\ref{4d}. 
Here we will develop a spatial version of the reduced tree structure and follow the approach of Geiger \cite{Geiger} which is built on a distributional identity that characterizes the exponential distributions. We also obtain in Proposition \ref{moments} the precise asymptotic for all moments of $L_K$ and $Z_\T(K)$, which may be of independent interest.

Finally, we will prove Theorem \ref{lowd} for $d\le 3$ in Section \ref{ld}. Our proof is based on the arguments developed in \cite{LeGall-Lin}. After conditioning on the size of the genealogical tree, we can carry out a careful analysis by applying the known relations between CBRW and ISE via the Brownian snake. 

We will use $c$ and $C$ to denote constants which may be different from line to line. Moreover, a constant denoted by $c_h$ or $C_h$ means that it depends on $h$, $f(x)=o_x(1)$ means that $f(x)\to 0$ as $\|x\|\to\infty$, and $f(x)\sim g(x)$ means that $f(x)/g(x)\to 1$ as $\|x\|\to\infty$.

\section{Preliminaries}\label{lems}

In this section, we state some facts and lemmas that will be used later.

Let us first introduce some notation. Let 
\[
\mathbb{U}:=\{\varnothing\}\bigcup\Big(\bigcup_{n\ge1}(\n^*)^n\Big)
\]
where $\n^*=\n\setminus\{0\}$. If $u=i_1\cdots i_n\in \mathbb{U}$, it can be viewed as the $i_n$-th child of the $i_{n-1}$-th child of $\cdots$ of the $i_1$-th child of the initial ancestor. This initial ancestor $\rho$ is denoted by $\varnothing$ as an element of $\mathbb{U}$. We call a subset $\t$ of $\mathbb{U}$ a tree if it satisfies
\begin{enumerate}
\item $\varnothing\in\t$;
\item if $uj\in\t$ for some $j\in\n^*$, then $u\in \t$;
\item if $u\in\t$, then there exists some integer $N_u(\t)\in\n= \{0,1,2,\ldots\}$ so that $uj\in\t$ if and only if $1\le j\le N_u(\t)$.
\end{enumerate}

%

For \(u=i_1\cdots i_n\in \mathbb{U}\), its generation is given by \(|u|=n\). For \(0\le k\le n\), we write \(u_k=i_1\cdots i_k\) for the ancestor of \(u\) in the \(k\)-th generation, with the convention that \(u_0=\varnothing\). The ancestral line of \(u\) is thus the finite sequence \((u_0,\dots, u_n)\). For \(u,v\in\mathbb{U}\), we say that \(v\) is a descendant of \(u\) if \(u\) is an ancestor of \(v\). If \(u\) belongs to a tree \(\t\), let \(\t_u\) be the subtree rooted at \(u\), defined by
\[
\t_u:=\{v\in\t\colon v=uj_1\cdots j_m, \textrm{ for some } j_1\cdots j_m \in \mathbb{U}\}.
\]

The lexicographic order on \(\mathbb{U}\) is defined as follows. For \(u=i_1\cdots i_n\) and \(v=j_1\cdots j_m\) in \(\mathbb{U}\), we say that \(u\) is on the left of \(v\) (equivalently, \(u\) is strictly smaller than \(v\) in the lexicographic order) if
\begin{itemize}
\item either there exists \(k\le n-1\) such that \(i_r=j_r\) for all \(1\le r\le k\) and \(i_{k+1}<j_{k+1}\);
\item or \(n<m\) and \(i_r=j_r\) for every \(1\le r\le n\).
\end{itemize}
For two distinct vertices \(u,v\in\mathbb{U}\), we say that \(u\) is on the right of \(v\) if \(u\) is not on the left of \(v\). In particular, \(u\) lies on the left of all its strict descendants.

The genealogical tree $\T$ of the CBRW is viewed as a random tree taking values in $\mathbb{U}$.
For any $u\in\T$, the corresponding subtree $\T_u$ (containing $u$) is defined as above.

\subsection{Change of measure and Spinal decomposition} $ $

As the offspring law has mean $1$, $(Z_n=Z_n(\Z^d))_{n\ge0}$ is a critical Bienaym\'e--Galton--Watson process. It is immediate to see that $(Z_n)_{n\ge0}$ is a martingale with respect to the natural filtration $(\mathcal{F}_n)_{n\ge0}$ where $\mathcal{F}_n:=\sigma((u, S_u),|u|\le n)$. Note that the lexicographic information up to the $n$-th generation is also contained in $\mathcal{F}_n$. Therefore, for all $x\in \Z^d$ we can define a new probability measure $\Q_x$ as follows
\[
\frac{\mathrm{d}\Q_x}{\mathrm{d}\P_x}\Big\vert_{\mathcal{F}_n}=Z_n, \quad \forall n\ge0.
\]

In what follows, we construct a probability measure $\Q^*_x$ on the branching random walk $\{(u,S_u), u\in\T\}$ with a marked ray $(w_n)_{n\ge0}$, so that the marginal law of $\{(u,S_u), u\in\T\}$ under $\Q_x^*$ is $\Q_x$. 
\begin{itemize}
\item We start with an initial ancestor $\rho$ located at $S_\rho=x$, and we take $w_0=\rho$. 
\item At time $1$, the root $\rho$ dies and produces a random number of children according to the size-biased offspring law:
\[
\hat{p}_k=kp_k, \quad \forall k\ge0.
\] 
All these children form the first generation. Each child makes an independent jump from $S_\rho$ according to $\mu$. Among the children of $w_0$, we choose uniformly one to be $w_1$. Let $\L(w_1)$ be the set of children of $w_0$ which are on the left of $w_1$ and let $\rt(w_1)$ be the set of children of $w_0$ which are on the right of $w_1$.
\item For any $n\ge1$, assume that the system is well constructed up to the $n$-th generation. At time $n+1$, $w_n$ dies and produces a random number of children according to $\{\hat{p}_k\}_{k\ge0}$, while any other individual of the $n$-th generation dies and produces independently a random number of children according to $\{p_k\}_{k\ge0}$. Each new born individual makes an independent jump from the position of its parent according to $\mu$. All children of the individuals of the $n$-th generation form the $(n+1)$-th generation. Among the children of $w_n$, we choose uniformly one to be $w_{n+1}$ and let $\L(w_{n+1})$ (or $\rt(w_{n+1})$) be the set of children of $w_n$ which are on the left (or right respectively) of $w_{n+1}$.
\item By abuse of notation, we still write $\T$ for the genealogical tree formed by the individuals of this system, rooted at $w_0=\rho$. The position of each individual $u\in\T$ is still denoted by $S_u$. The marked ray $(w_n)_{n\ge0}$ is called the spine of $\T$.
\end{itemize}
For any $n\ge1$, let $L(w_n)$ and $R(w_n)$ be the cardinality of $\L(w_n)$ and of $\rt(w_n)$ respectively. 
According to the construction, we have
\[
\E_{\Q^*_x}[L(w_n)+R(w_n)]=\sum_{k\ge1} k\hat{p}_k-1=\sigma^2.
\]
This can also be recovered by the fact that 
\[
\Q^*_x(L(w_n)=i, R(w_n)=j)=p_{i+j+1}, \quad \forall i,j\ge0.
\]

As a consequence of the change of measure stated above, we have the following proposition. 

\begin{proposition}\label{spine}
Under the probability $\Q^*_x$ with any $x\in\Z^d$,
\begin{enumerate}
\item for any $n\ge0$, 
\[
\Q_x^*(w_n=u\vert \mathcal{F}_n)=\frac{1}{Z_n}, \forall u\in\T \textrm{ such that } |u|=n;
\]
\item the random variables $ (L(w_n), R(w_n))_{n\ge1}$ are i.i.d.;
\item along the spine $(w_n)_{n\ge0}$, $(S_{w_n})_{n\ge0}$ is a random walk with i.i.d.~increments distributed as $\mu$, started from $S_{w_0}=x$, which is independent of $ (L(w_n), R(w_n))_{n\ge1}$.
\item conditioned on $(w_n, S_{w_n}; \L(w_n)\cup \rt(w_n), (S_u)_{u\in\L(w_n)\cup \rt(w_n)})_{n\ge1}$, for all $u\in\bigcup_{n\ge1}(\L(w_n)\cup \rt(w_n))$, its descendants $(v, S_v)_{v\in\T_u}$ form an independent branching random walk of law $\P_{S_u}$.
\end{enumerate}

\end{proposition}

Recall that $(u_0,\cdots, u_{|u|})$ is the ancestral line of $u$ and that $(S_n)$ is a random walk with i.i.d.~increments distributed as $\mu$ under $\P$. The following Many-to-One lemma is derived from the previous proposition.

\begin{lemma}\label{Manyto1}
For any integer $n\ge0$ and $x\in\Z^d$, for any non-negative measurable function $F:\R^{n+1}\to \R_+$, we have
\[
\E_x\left[\sum_{|u|=n}F(S_{u_k}, 0\le k\le n)\right]= \E_x\big[F(S_k, 0\le k\le n)\big].
\]
\end{lemma}

\subsection{Results for random walk on $\Z^d$}\label{SRW} $ $

In this subsection, we state some results on the $d$-dimensional random walk $(S_n)_{n\ge0}$, whose proofs will be postponed to Appendix \ref{appendixA}. 

For a finite subset $K\subset\Z^d$, define the hitting time of $K$ by
\[
T_K:=\inf\{n\ge0\colon S_n\in K\}.
\]
For convenience, we write $q_K(x):=\P_x(T_K<\infty)$ for any $x\in \Z^d$. We also define $T^+_K:=\inf\{n\ge1 \colon S_n\in K\}$.

Recall that when $d\ge3$, for a finite set $K\subset \Z^d$, as $\|x\|\to\infty$,
\begin{equation}\label{hittingprobabS}
\P_x(T_K<\infty)= g(x,0)\cap(K)(1+o_x(1)).
\end{equation}
Assume that $\mu$ has a finite $(d-1)$-th moment. Then together with \eqref{green-not-finite-support}, we see that 
\begin{equation}\label{hittingprobabS-2}
\P_x(T_K<\infty)= \frac{\c_d\cap(K)}{\J(x)^{d-2}}(1+o_x(1)).
\end{equation}
For any $r>0$, we define the ball $\B_r:=\{z\in\Z^d \colon \J(z)\le r\}$ and the exit time
\[
\tau_r:=\inf\{n\ge0 \colon S_n\notin \B_r\}=\inf\{n\ge0 \colon \J(S_n)>r\}.
\]
When $\mu$ is finitely supported, we will consider the support of $S_{\tau_r}$, denoted by $\partial_\mu \B_r$. In other words,
\[
\partial_\mu \B_r=\{z\in \B_r^c \colon \exists\, x\in \B_r\textrm{ such that } \mu(x-z)>0\},
\]
where $\B_r^c$ is the complementary of $\B_r$ in $\Z^d$.

Note that if the random walk starts from $x$ and ends at $y$ with $\|x\|\wedge \|y\|\gg1$, then it hardly enters a finite set $K\subset\Z^d$ when $d\ge3$. This is what the following lemma states. 

\begin{lemma}\label{fargreen}
We assume that the jump law $\mu$ on $\Z^d$ has zero mean and a finite $(d-1)$-th moment. 
When $d\ge3$, for any finite subset $K\subset\Z^d$, 
there exists some constant $C>0$ (depending on $K$ and $d$) such that for all $x,y \in \Z^d$,
\[
g(x,y)-\sum_{n=0}^\infty\P_x(S_n=y, T_K>n) \le \frac{C}{\J(x)^{d-2}\J(y)^{d-2}}.
\]
\end{lemma}

The next lemma will be used in the proof of Theorem \ref{highd} when $d\ge 5$, and also in the proof of Theorem \ref{lowd} when $d\le 3$.

\begin{lemma}\label{bigjump}
Given any $x\in \Z^d, d\geq 1$, we have for $R>r>\J(x)$ that 
\begin{equation}\label{eq:bigjump}
  \P_x(\J(S_{\tau_r})>R)\leq \E_x[\tau_r]\cdot \P(\J(X_1)>R-r).
\end{equation}
If we assume that the jump law $\mu$ on $\Z^d$ has zero mean and a finite $(d-1)$-th moment (with $d\geq 3$), then
\begin{equation}\label{eq:bigjump-asymp}
  \P_x(\J(S_{\tau_r})>2r)=o(r^{3-d}) \quad \mbox{ as } r\to \infty.
\end{equation}
\end{lemma}

The following lemma applies differently for $d=4$ or $d\ge5$. 

\begin{lemma}\label{bdhgg}
We assume that the jump law $\mu$ on $\Z^d$ has zero mean and a finite $(d-1)$-th moment. 
If $h \colon \Z^d\to \R_+$ is a function such that
\[
h(x)\le C_0 \frac{(\log (2+\|x\|))^k}{1+\|x\|^{d-2}},
\]
with some $k\in\mathbb{N}$ and $C_0\in(0,\infty)$, then there exists some constant $C_1>0$ (depending on $d$) such that for any $a,b\in\Z^d$,
\begin{enumerate}
\item when $d=4$, 
\begin{equation}\label{hgg4d}
\sum_{w\in\Z^4}h(w) g(a,w)g(w,b)\le C_1C_0 \frac{(\log (2+M))^k[1\vee \log\frac{M+1}{m+1}]}{(1+M)^2};
\end{equation}
\item when $d\ge5$,
\begin{equation}\label{hgg5d}
\sum_{w\in\Z^d}h(w) g(a,w)g(w,b)\le C_1C_0 \frac{(\log (2+M))^k(1+m)^{4-d}}{(1+M)^{d-2}},
\end{equation}
\end{enumerate}
where $M=\max\{\|a\|, \|b\|, \|a-b\|\}$ and $m=\min\{\|a\|,\|b\|, \|a-b\|\}$.
\end{lemma}

The last lemma holds when $d=4$.

\begin{lemma}\label{sumh}
We assume that the jump law $\mu$ on $\Z^4$ is symmetric with finite support.
For a non-negative measurable function $h\colon \Z^4\to\R_+$ such that
\[
h(x)\sim c_h \frac{(\log \J(x))^k }{\J(x)^2},\textrm{  as } \|x\|\to\infty,
\]
with some integer $k\in\mathbb{N}$ and some constant $c_h>0$, we have the following estimates:
\begin{enumerate}
\item For any $x\in\Z^4$, there exists a constant $C_h\in(0,\infty)$ such that
\begin{equation}\label{bdsumh}
\E_x\Bigg[\sum_{j=0}^{T_K}h(S_j); T_K<\infty\Bigg]\le C_h \frac{\log(2+\J(x))^{k+1}}{(1+\J(x))^2}.
\end{equation}
\item As $\|x\|\to\infty$, 
\begin{equation}\label{cvgsumh}
\E_x\Bigg[\sum_{j=0}^{T_K-1}h(S_j)\,\bigg\vert\, T_K<\infty\Bigg]\sim \frac{4c_h}{k+1}(\log\J(x))^{k+1}.
\end{equation} 
\end{enumerate}
\end{lemma}

\section{Occupation times when $d\ge5$: proof of Theorem \ref{highd}}\label{hd}

In this section, we assume that $\mu$ has zero mean and satisfies \eqref{eq:weakLd}, 
with $d\ge5$. We aim to study the asymptotics of $(L_K, Z_\T(K))$ for a finite subset $K\subset\Z^d$ in high dimensions. In fact, it suffices to study 
\[
\E_x[f(L_K)\ind{L_K\ge 1}]
\]
for any continuous and bounded function $f\colon\R\to\R$. The joint convergence in law of $(L_K, Z_\T(K))$ follows in the same way by considering 
\[
\E_x[f(L_K, Z_\T(K))\ind{L_K\ge 1}]
\]
for any continuous and bounded function $f\colon\R^2\to\R$.

First, by the many-to-one Lemma \ref{Manyto1}, one sees that
\begin{equation}\label{meanL}
\E_x[L_K\ind{L_K\ge 1}]=\E_x[L_K]=\P_x(T_K<\infty)=q_K(x).
\end{equation}
Observe that for $x\notin K$,
\[
\E_x[f(L_K)\ind{L_K\ge 1}]=\E_x\left[\sum_{n=1}^\infty \sum_{|u|=n}\ind{u\in \L_K}\frac{f(L_K)\ind{L_K\ge 1}}{L_K}\right]
\]
By change of measure and Proposition \ref{spine}, we get that
\begin{align*}
&\E_x[f(L_K)\ind{L_K\ge 1}]=\sum_{n=1}^\infty\E_{\Q_x}\left[\frac{\sum_{|u|=n}\ind{u\in\L_K}}{Z_n}\E_x\left[\frac{f(L_K)\ind{L_K\ge1 }}{L_K}\Big\vert \mathcal{F}_n\right]\right]\\
=&\sum_{n=1}^\infty\E_{\Q_x^*}\left[\sum_{|u|=n}\ind{w_n=u}\ind{u\in\L_K}\E_{x}\left[\frac{f(L_K)\ind{L_K\ge1 }}{L_K}\Big\vert \mathcal{F}_n\right]\right]\\
=&\sum_{n=1}^\infty\E_{\Q_x^*}\left[\sum_{|u|=n}\ind{w_n=u}\ind{w_n\in\L_K}\E_{\Q^*_x}\left[\frac{f(L_K)\ind{L_K\ge1 }}{L_K} \Big\vert \widetilde{\mathcal{F}}_n \right]\right]\\
=&\sum_{n=1}^\infty\E_{\Q_x^*}\left[\ind{w_n\in\L_K}\E_{\Q^*_x}\left[\frac{f(L_K)\ind{L_K\ge1 }}{L_K}\Big\vert \widetilde{\mathcal{F}}_n\right]\right]\\
=&\sum_{n=1}^\infty\E_{\Q_x^*}\left[\ind{w_n\in\L_K}\frac{f(L_K)\ind{L_K\ge1 }}{L_K}\right].
\end{align*}
Here we define \(\widetilde{\mathcal F}_n:=\sigma(\mathcal{F}_n, (w_k, S_{w_k})_{k\le n})\), which augments \(\mathcal{F}_n\) with the spine information up to generation \(n\). And we have used the fact that, on the event \(\{w_n=u\}\),
\[
\mathbf{1}_{\{u\in\L_K\}}\E_x\!\left[\frac{f(L_K)\mathbf{1}_{\{L_K\ge 1\}}}{L_K}\,\middle|\, \mathcal{F}_n\right]
=
\mathbf{1}_{\{w_n\in\L_K\}}\E_{\Q^*_x}\!\left[\frac{f(L_K)\mathbf{1}_{\{L_K\ge 1\}}}{L_K}\,\middle|\, \widetilde{\mathcal{F}}_n\right],
\]
as on both sides, $\L_K$ does not contain any descendant of $u=w_n$. 

Note that $w_n\in\L_K$ means that
\[
\forall 0\le k<n, S_{w_k}\notin K, \; \textrm{ and } S_{w_n}\in K.
\]
We set $T_K(w):=\inf\{n\ge0 \colon S_{w_n}\in K\}$. For any $v\in\T$, we say that $v\le \L_K$ if either $v\in\L_K$, or for all $\rho\le w\le v$, $S_w\notin K$. We define
\[
L_K(v):=\sum_{w\in\T_v}\ind{w\in \L_K}, \textrm{ if } v\le \L_K.
\]
We thus get that
\begin{align*}
&\E_x[f(L_K)\ind{L_K\ge 1}]=\E_{\Q^*_x}\left[\ind{T_K(w)<\infty}\frac{f(1+\sum_{k=1}^{T_K(w)}\sum_{u\in\L(w_k)\cup \rt(w_k)}L_K(u))}{1+\sum_{k=1}^{T_K(w)}\sum_{u\in\L(w_k)\cup \rt(w_k)}L_K(u)}\right]\\
=&\sum_{y\in K}\sum_{n=0}^\infty\E_{\Q^*_x}\left[\ind{T_K(w)=n, S_{w_n}=y}\frac{f(1+\sum_{k=1}^{n}\sum_{u\in\L(w_k)\cup \rt(w_k)}L_K(u))}{1+\sum_{k=1}^{n}\sum_{u\in\L(w_k)\cup \rt(w_k)}L_K(u)}\right].
\end{align*}
Let us denote by $\bar\mu$ the probability distribution on $\Z^d$ such that for all $x\in \Z^d$, $\bar\mu(\{x\})=\mu(\{-x\})$.
By time-reversal, it follows that 
\begin{align*}
\E_x[f(L_K)\ind{L_K\ge 1}]=&\sum_{y\in K}\sum_{n=0}^\infty \bar\E_y\left[\ind{S_n=x, T^+_K>n}\frac{f(1+\sum_{k=1}^n\sum_{j=1}^{B_k}L_K^{k,j}(S_k+X_{k,j}))}{1+\sum_{k=1}^n\sum_{j=1}^{B_k}L_K^{k,j}(S_k+X_{k,j})}\right],
\end{align*}
where, under $\bar \P_y$, 
\begin{itemize}
  \item $(S_n)_{n\ge0}$ is a random walk started from $y$ with i.i.d.~increments distributed as $\bar\mu$,
  \item $(B_k)_{k\ge1}$ is a sequence of i.i.d.~random variables distributed as $L(w_1)+R(w_1)$ under $\Q^*_x$,
  \item $\{X_{k,j}; k,j\ge1\}$ are i.i.d.~random variables distributed as $\mu$,
  \item for any $z\in\Z^d$, $\{L_K^{k,j}(z); k,j \geq 1\}$ are i.i.d.~random variables distributed as $L_K$ under $\P_z$,
  \item $(B_k)_{k\ge1}$, $(S_n)_{n\ge0}$, $\{X_{k,j}; k,j\ge1\}$ and $\{L_K^{k,j}; k,j \geq 1\}$ are independent.
\end{itemize}   
It follows that, given $(B_k)_{k\ge1}$, $(S_n)_{n\ge0}$ and $\{X_{k,j}; k,j\ge1\}$, the family of random variables $L_K^{k,j}(S_k+X_{k,j})$, $k,j\ge1$ are independent and distributed as $L_K$ under $\P_{S_k+X_{k,j}}$.

To simplify notation, we set 
\[\gamma:=\gamma_x:=\frac{\J(x)}{(\log \J(x))^{\frac14}}.\] 
Recall that $\tau_\gamma=\inf\{n\ge0\colon S_n\notin \B_\gamma\}$. When $\J(x)$ is sufficiently large, since $x\notin \B_\gamma$, we have
\begin{align*}
&\E_x[f(L_K)\ind{L_K\ge 1}]\\
=&\sum_{y\in K}\sum_{n=0}^\infty\sum_{\ell=0}^n \bar \E_y\left[\ind{S_n=x, T^+_K>n, \tau_\gamma=\ell}\frac{f(1+\sum_{k=1}^n\sum_{j=1}^{B_k}L_K^{k,j}(S_k+X_{k,j}))}{1+\sum_{k=1}^n\sum_{j=1}^{B_k}L_K^{k,j}(S_k+X_{k,j})}\right]\\
=&\sum_{y\in K}\sum_{m=0}^\infty\sum_{\ell=0}^\infty \bar \E_y\left[\ind{S_{\ell+m}=x, T^+_K>m+\ell, \tau_{\gamma}=\ell}\frac{f(1+\sum_{k=1}^{\ell+m}\sum_{j=1}^{B_k}L_K^{k,j}(S_k+X_{k,j}))}{1+\sum_{k=1}^{\ell+m}\sum_{j=1}^{B_k}L_K^{k,j}(S_k+X_{k,j})}\right].
\end{align*}
For any integers $1\le m_1\le m_2\le \ell+m$, let us write $\Sigma_{[m_1, m_2]}$ for the sum $\sum_{k=m_1}^{m_2}\sum_{j=1}^{B_k}L_K^{k,j}(S_k+X_{k,j})$. In the previous expectation, we claim that
\[
\Sigma_{[1,\ell+m]}:=\sum_{k=1}^{\ell+m}\sum_{j=1}^{B_k}L_K^{k,j}(S_k+X_{k,j})
\]
can be approximated by $\Sigma_{[1,\ell-1]}$.
In fact, we will show that
\begin{equation}\label{smallpart5d}
\Sigma_\eqref{smallpart5d}:=\sum_{y\in K}\sum_{m=0}^\infty\sum_{\ell=0}^\infty \bar \E_y\left[\ind{S_{\ell+m}=x, T^+_K>m+\ell, \tau_{\gamma}=\ell}\ind{\Sigma_{[\ell,\ell+m]}\ge 1}\right]=o_x(1) \|x\|^{2-d}.
\end{equation}

Let us postpone the proof of \eqref{smallpart5d} until the end of this section, and continue the study of $\E_x[f(L_K)\ind{L_K\ge 1}]$. Recall that the sum $\Sigma_{[\ell,\ell+m]}$ takes its values in $\n$.
On the event $\{\Sigma_{[\ell,\ell+m]}\ge 1\}$, $\frac{|f(1+t)|}{1+t}\le \sup_t|f(t)|=:\|f\|_\infty<\infty$. On the contrary, if $\Sigma_{[\ell,\ell+m]}$ is equal to 0, $\Sigma_{[1,\ell+m]}$ coincides with $\Sigma_{[1,\ell-1]}$. Consequently,
\begin{align*}
&\E_x[f(L_K)\ind{L_K\ge 1}]\\
=\,&o_x(1) \|x\|^{2-d}+\sum_{y\in K}\sum_{m=0}^\infty\sum_{\ell=0}^\infty\bar \E_y\left[\ind{S_{\ell+m}=x, T^+_K>m+\ell, \tau_{\gamma}=\ell}\frac{f(1+\Sigma_{[1,\ell-1]})}{1+\Sigma_{[1,\ell-1]}} \ind{ \Sigma_{[\ell,\ell+m]} = 0 } \right]\\
=\,&o_x(1) \|x\|^{2-d}+\sum_{y\in K}\sum_{m=0}^\infty\sum_{\ell=0}^\infty \bar \E_y\left[\ind{S_{\ell+m}=x, T^+_K>m+\ell, \tau_{\gamma}=\ell}\frac{f(1+\Sigma_{[1,\ell-1]})}{1+\Sigma_{[1,\ell-1]}}\right].
\end{align*}
By the strong Markov property of the random walk $(S_n)_{n\geq 0}$ at $\tau_\gamma$, we have
\begin{align}
&\E_x[f(L_K)\ind{L_K\ge 1}] = o_x(1) \|x\|^{2-d} \nonumber\\
&\quad +\sum_{y\in K}\sum_{z\in \B_\gamma^c}\left(\bar \E_y\left[\frac{f(1+\Sigma_{[1,\tau_\gamma-1]})}{1+\Sigma_{[1,\tau_\gamma-1]}}\ind{S_{\tau_\gamma}=z, \tau_\gamma<T^+_K}\right]\sum_{m=0}^\infty\bar \P_z(S_m=x, T_K>m)\right). \label{eq:sum-d5}
\end{align}

Note that for any $z\in  \B_\gamma^c$ such that $\gamma<\J(z)\leq 2\gamma$, we deduce from Lemma \ref{fargreen} that
\[
\sum_{m=0}^\infty\bar \P_z(S_m=x, T_K>m)=g(x,z)+o_x(1)\frac{1}{\J(x)^{d-2}}.
\]
Together with \eqref{green-not-finite-support}, we see that
\[
\sum_{m=0}^\infty\bar \P_z(S_m=x, T_K>m)=(1+o_x(1))\frac{\c_d}{\J(x)^{d-2}}.
\]
Moreover, this asymptotic is uniform for all such $z\in \B_{2\gamma}\backslash \B_{\gamma}$.
Thus, 
\begin{align}
&\sum_{z\in \B_{2\gamma}\backslash \B_{\gamma}}\left(\bar \E_y\left[\frac{f(1+\Sigma_{[1,\tau_\gamma-1]})}{1+\Sigma_{[1,\tau_\gamma-1]}}\ind{S_{\tau_\gamma}=z, \tau_\gamma<T^+_K}\right]\sum_{m=0}^\infty\bar \P_z(S_m=x, T_K>m)\right) \nonumber \\ 
=\;&(1+o_x(1))\frac{\c_d}{\J(x)^{d-2}}\sum_{z\in \B_{2\gamma}\backslash \B_{\gamma}}\bar \E_y\left[\frac{f(1+\Sigma_{[1,\tau_\gamma-1]})}{1+\Sigma_{[1,\tau_\gamma-1]}}\ind{S_{\tau_\gamma}=z, \tau_\gamma<T^+_K}\right].\label{eq:sum-annu}
\end{align}

On the other hand, we have
\begin{align}
&\left|\sum_{z\in \B_{2\gamma}^c}\left(\bar \E_y\left[\frac{f(1+\Sigma_{[1,\tau_\gamma-1]})}{1+\Sigma_{[1,\tau_\gamma-1]}}\ind{S_{\tau_\gamma}=z, \tau_\gamma<T^+_K}\right]\sum_{m=0}^\infty\bar \P_z(S_m=x, T_K>m)\right)\right|\nonumber \\ 
\leq \;& \sum_{z\in \B_{2\gamma}^c}\bar \E_y\left[\frac{|f(1+\Sigma_{[1,\tau_\gamma-1]})|}{1+\Sigma_{[1,\tau_\gamma-1]}}\ind{S_{\tau_\gamma}=z, \tau_\gamma<T^+_K}\right]\sum_{m=0}^\infty\bar \P_z(S_m=x)\nonumber\\
\leq \;& \|f\|_\infty\sum_{z\in \B_{2\gamma}^c}\bar\P_y(S_{\tau_\gamma}=z, \tau_\gamma<T^+_K) g(x,z)\nonumber \\ 
= \; & \|f\|_\infty\bar \E_y\left[g(x,S_{\tau_\gamma})\ind{\J(S_{\tau_\gamma})>2\gamma, \tau_\gamma<T^+_K}\right].\label{eq:sum-outside-2gamma}
\end{align}
Let us fix some $\varepsilon\in (0,1)$. When $\J(x)$ is sufficiently large so that $\varepsilon\J(x)>2\gamma$, 
\begin{align}
   & \bar \E_y\left[g(x,S_{\tau_\gamma})\ind{\J(S_{\tau_\gamma})>2\gamma, \tau_\gamma<T^+_K}\right]\nonumber\\ 
   \leq\; &\bar \E_y\left[g(x,S_{\tau_\gamma})\ind{\J(S_{\tau_\gamma})>\varepsilon\J(x)}\right]+\bar \E_y\left[g(x,S_{\tau_\gamma})\ind{2\gamma<\J(S_{\tau_\gamma})\leq \varepsilon\J(x)}\right].\label{eq:sum-outside-2gamma+}
\end{align}
Since the Green function $g$ is bounded, we set $\|g\|_\infty:=\sup_{x,z\in\Z^d}g(x,z)<\infty$. Then applying \eqref{eq:bigjump}, we get 
\begin{align*}
  \bar \E_y\left[g(x,S_{\tau_\gamma})\ind{\J(S_{\tau_\gamma})>\varepsilon\J(x)}\right]\leq & \|g\|_\infty \cdot\bar\P_y(\J(S_{\tau_\gamma})>\varepsilon\J(x))\\ 
  \leq & \|g\|_\infty \cdot \bar\E_y[\tau_\gamma] \cdot \P(\J(X_1)>\varepsilon\J(x)-\gamma)\\ 
  \leq & \|g\|_\infty \cdot \bar\E_y[\tau_\gamma] \cdot \P(\J(X_1)>\frac{\varepsilon}{2}\J(x))
\end{align*}
Under the assumption \eqref{eq:weakLd}, $\P(\J(X_1)>\frac{\varepsilon}{2}\J(x))$ is bounded by $(\frac{\varepsilon}{2}\J(x))^{-d}$ up to a multiplicative constant. 
As the subset $K$ is finite, there exists some constant $C>0$ such that $\sup_{y\in K} \bar\E_y[\tau_\gamma]\leq C\gamma^2$ (see Proposition 2.4.5 and Exercise 2.7 in \cite{Lawler-Limic}) . Putting them together, we see that 
\[
  \bar \E_y\left[g(x,S_{\tau_\gamma})\ind{\J(S_{\tau_\gamma})>\varepsilon\J(x)}\right]=o_x(1)\|x\|^{2-d}.
\]
For the second expectation on the right hand side of \eqref{eq:sum-outside-2gamma+}, we use the bound $g(0,x)\leq C\|x\|^{2-d}$ to deduce that
\[
  \bar \E_y\left[g(x,S_{\tau_\gamma})\ind{2\gamma<\J(S_{\tau_\gamma})\leq \varepsilon\J(x)}\right]\leq C_\varepsilon \|x\|^{2-d} \bar\P_y(2\gamma<\J(S_{\tau_\gamma})\leq \varepsilon\J(x)),
\]
where the finite constant $C_\varepsilon$ depends on $\varepsilon$. By \eqref{eq:bigjump-asymp}, it follows immediately that 
\[
  \bar \E_y\left[g(x,S_{\tau_\gamma})\ind{2\gamma<\J(S_{\tau_\gamma})\leq \varepsilon\J(x)}\right]=o_x(1)\|x\|^{2-d}.
\]
As a result, the last term in \eqref{eq:sum-outside-2gamma} is $o_x(1)\|x\|^{2-d}$. Putting this and \eqref{eq:sum-annu} into \eqref{eq:sum-d5}, we deduce that
\begin{align}
&\E_x[f(L_K)\ind{L_K\ge 1}]= o_x(1) \|x\|^{2-d} \nonumber\\
&+ (1+o_x(1))\frac{\c_d}{\J(x)^{d-2}}\sum_{y\in K} \sum_{z\in \B_{2\gamma}\backslash \B_{\gamma}}\bar \E_y\left[\frac{f(1+\Sigma_{[1,\tau_\gamma-1]})}{1+\Sigma_{[1,\tau_\gamma-1]}}\ind{S_{\tau_\gamma}=z, \tau_\gamma<T^+_K}\right].\label{eq:E-annu}
\end{align}
Notice that 
\[
  \left|\sum_{y\in K}\sum_{z\in \B_{2\gamma}^c}\bar \E_y\left[\frac{f(1+\Sigma_{[1,\tau_\gamma-1]})}{1+\Sigma_{[1,\tau_\gamma-1]}}\ind{S_{\tau_\gamma}=z, \tau_\gamma<T^+_K}\right]\right|\leq \|f\|_\infty \sum_{y\in K}\P_y(\J(S_{\tau_\gamma})>2\gamma).
\]
Since $|K|<\infty$, we can apply \eqref{eq:bigjump-asymp} to see that 
\[
  \sum_{y\in K}\sum_{z\in \B_{2\gamma}^c}\bar \E_y\left[\frac{f(1+\Sigma_{[1,\tau_\gamma-1]})}{1+\Sigma_{[1,\tau_\gamma-1]}}\ind{S_{\tau_\gamma}=z, \tau_\gamma<T^+_K}\right]=o_x(1).
\]
Consequently, we derive from \eqref{eq:E-annu} that
\begin{multline}\label{cvghighd}
\E_x[f(L_K)\ind{L_K\ge 1}] \\
=o_x(1) \|x\|^{2-d}+(1+o_x(1))\frac{\c_d}{\J(x)^{d-2}}\sum_{y\in K}\bar \E_y\left[\frac{f(1+\Sigma_{[1,\tau_\gamma-1]})}{1+\Sigma_{[1,\tau_\gamma-1]}}\ind{ \tau_\gamma<T^+_K}\right].
\end{multline}

Then observe that for any $y\in K$, as $\gamma\to\infty$, $\bar\P_y$-a.s.~$\tau_\gamma\to\infty$ and $\Sigma_{[1,\tau_\gamma-1]}$ converges towards 
\begin{equation}\label{backwardspine}
\Sigma_\infty:=\sum_{k=1}^\infty\sum_{j=1}^{B_k}L_K^{k,j}(S_k+X_{k,j}),
\end{equation}
which is finite $\bar \P_y$-a.s. In fact, recalling $\bar\E_y[B_k]=\sigma^2$, by \eqref{meanL} we have
\begin{align*}
\bar\E_y[\Sigma_\infty]=&\sum_{k=1}^\infty \bar\E_y\left[\sum_{j=1}^{B_k}L_K^{k,j}(S_k+X_{k,j})\right]\\
=&\sigma^2\sum_{k=1}^\infty \bar\E_y\left[\sum_{e\in \Z^d} q_K(S_k+e)\mu(e)\right]=\sigma^2\sum_{w\in\Z^d} h(w)g(w,y),
\end{align*}
where $h(w):=\sum_{e\in \Z^d} q_K(w+e)\mu(e)$. According to \eqref{eq:weakLd} and \eqref{hittingprobabS-2}, one can find a constant $C>0$, such that 
\begin{equation}\label{eq:h-bound}
h(x) \le \sum_{\|e\|>\frac{\|x\|}{2}} \mu(e) + \sum_{\|e\|\leq \frac{\|x\|}{2}} q_K(x+e)\mu(e) \le\frac{C}{1+\|x\|^{d-2}}.
\end{equation}
Combined with \eqref{green-not-finite-support}, one sees that as $K$ is finite and $d\ge5$, 
\[
\sup_{y\in K}\sum_{w\in\Z^d} h(w)g(w,y)<\infty.
\]
This implies that $\Sigma_\infty$ is finite $\bar\P_y$-a.s. Therefore, by dominated convergence theorem, for any $y\in K$.
\[
\bar\E_y\left[\frac{f(1+\Sigma_{[1,\tau_\gamma-1]})}{1+\Sigma_{[1,\tau_\gamma-1]}}\ind{ \tau_\gamma<T^+_K}\right]\xrightarrow{\|x\|\to\infty} \bar \E_y\left[\frac{f(1+\Sigma_\infty)}{1+\Sigma_\infty}\ind{T_K^+=\infty}\right].
\]
Going back to \eqref{cvghighd}, we obtain that
\begin{align*}
&\E_x[f(L_K)\ind{L_K\ge 1}]\\
=&o_x(1) \|x\|^{2-d}+(1+o_x(1))\frac{\c_d}{\J(x)^{d-2}}\sum_{y\in K}\left(\bar\E_y\left[\frac{f(1+\Sigma_\infty)}{1+\Sigma_\infty}\ind{T_K^+=\infty}\right]+o_x(1)\right).
\end{align*}
In particular, when the function $f$ is constantly equal to 1, we get that
\[
\P_x(L_K\ge 1)=o_x(1) \|x\|^{2-d}+(1+o_x(1))\frac{\c_d}{\J(x)^{d-2}}\sum_{y\in K}\bigg(\bar\E_y\left[\frac{\ind{T_K^+=\infty}}{1+\Sigma_\infty}\right]+o_x(1)\bigg).
\]
In summary, when $\|x\|\to \infty$, 
\begin{align*}
\P_x(L_K\ge1)\sim&\frac{\c_d}{\J(x)^{d-2}}\bigg(\sum_{y\in K}\bar \E_y\left[\frac{\ind{T_K^+=\infty}}{1+\Sigma_\infty}\right]\bigg),\\
\E_x[f(L_K)\ind{L_K\ge 1}]\sim&\frac{\c_d}{\J(x)^{d-2}}\sum_{y\in K}\bigg(\bar\E_y\left[\frac{f(1+\Sigma_\infty)}{1+\Sigma_\infty}\ind{T_K^+=\infty}\right]\bigg).
\end{align*}
As a result, we have established that 
\[
\lim_{\|x\|\to\infty} \E_x[f(L_K)\vert L_K\ge 1] = \frac{\sum_{y\in K}\bar\E_y\left[\frac{f(1+\Sigma_\infty)}{1+\Sigma_\infty}\ind{T_K^+=\infty}\right]}{\sum_{y\in K}\bar\E_y\Big[\frac{\ind{T_K^+=\infty}}{1+\Sigma_\infty}\Big]},
\]
where the limit on the right-hand side defines a probability measure $\nu_{d,K}$ on $\n^*$ as
\begin{equation}\label{limitlaw5d}
 \nu_{d,K}(k)=\frac{\sum_{y\in K}\bar\E_y\left[\frac{\ind{1+\Sigma_\infty=k}}{1+\Sigma_\infty}\ind{T_K^+=\infty}\right]}{\sum_{y\in K}\bar\E_y\Big[\frac{\ind{T_K^+=\infty}}{1+\Sigma_\infty}\Big]}, \quad \forall k\ge 1.
\end{equation}
So we conclude that conditionally on $L_K\ge 1$ (i.e.~$Z_\T(K)\ge 1$), $L_K$ converges in law towards $\nu_{d,K}$. 

\begin{remark}
Our previous arguments recover Theorem 1.1 of \cite{Zhu+}, stated in \eqref{5duK}, under the same assumption on $\mu$. We see further that
\[
\bcap(K)=\sum_{y\in K}\bar\E_y\left[\frac{\ind{T_K^+=\infty}}{1+\Sigma_\infty}\right].
\]
\end{remark}

\begin{remark}\label{joint-limit}
  For the joint convergence in law of $(L_K, Z_\T(K))$, when $\|x\|\to\infty$, by similar arguments we get the limit of $\E_x[f(L_K, Z_\T(K))\vert L_K\ge 1]$ as 
  \begin{equation*} \label{limitlaw5d-withZ}
    \Bigg(\sum_{y\in K}\bar\E_y\left[\frac{\ind{T_K^+=\infty}}{1+\Sigma_\infty}\right]\Bigg)^{-1}\times \sum_{y\in K}\bar\E_y\left[\frac{f(1+\Sigma_\infty, Z_K(y)+\widetilde{\Sigma}_\infty)}{1+\Sigma_\infty}\ind{T_K^+=\infty}\right],
  \end{equation*}
  where 
  \[
    \widetilde{\Sigma}_\infty:=\sum_{k=1}^\infty\sum_{j=1}^{B_k}Z_K^{k,j}(S_k+X_{k,j}).
  \]
  Here, under $\bar\P_y$, 
  \begin{itemize}
  \item $Z_K(x)$ is distributed as $Z_\T(K)$ under $\P_y$;
  \item $(B_k)_{k\ge1}$, $(S_n)_{n\ge0}$ and $\{X_{k,j}; k,j\ge1\}$ are the same as above;
  \item for any $z\in\Z^d$, $\{(L_K^{k,j}(z), Z_K^{k,j}(z)); k,j \geq 1\}$ are i.i.d.~random variables distributed as $(L_K, Z_\T(K))$ under $\P_z$;
  \item $Z_K(x)$, $(B_k)_{k\ge1}$, $(S_n)_{n\ge0}$, $\{X_{k,j}; k,j\ge1\}$ and $\{(L_K^{k,j}, Z_K^{k,j}); k,j \geq 1\}$ are independent.
\end{itemize} 
\end{remark}

To finish the proof of Theorem \ref{highd}, it remains to check \eqref{smallpart5d}.

\begin{proof}[Proof of \eqref{smallpart5d}]
First of all, by Markov's inequality,
\begin{align*}
&\bar\E_y\left[\ind{S_{\ell+m}=x, T^+_K>m+\ell, \tau_{\gamma}=\ell}\ind{\sum_{k=\ell}^{\ell+m}\sum_{j=1}^{B_k}L_K^{k,j}(S_k+X_{k,j})\ge 1}\right]\nonumber\\
\le \;& \bar\E_y\Bigg[\ind{S_{\ell+m}=x, T^+_K>m+\ell, \tau_{\gamma}=\ell}\sum_{k=\ell}^{\ell+m}\sum_{j=1}^{B_k}L_K^{k,j}(S_k+X_{k,j})\Bigg]\nonumber\\
=\;& \bar\E_y\Bigg[\ind{S_{\ell+m}=x, T^+_K>m+\ell, \tau_{\gamma}=\ell}\sum_{k=\ell}^{\ell+m}\sum_{j=1}^{B_k}\E_{S_k+X_{k,j}}[L_K]\Bigg].
\end{align*}
Recalling \eqref{meanL}, we have $\E_{S_k+X_{k,j}}[L_K]=q_K(S_k+X_{k,j})$. Recall also that $\bar\E_y[B_k]=\sigma^2$. It follows that
\begin{align*}
\Sigma_\eqref{smallpart5d}
&\le \sum_{y\in K}\sum_{m=0}^\infty\sum_{\ell=0}^\infty\bar\E_y\Bigg[\ind{S_{\ell+m}=x, T^+_K>m+\ell, \tau_{\gamma}=\ell}\sum_{k=\ell}^{\ell+m}\sigma^2 \sum_{e\in \Z^d} q_K(S_k+e)\mu(e)\Bigg].
\end{align*}
By the strong Markov property at time $\tau_\gamma$ and then by time-reversal, one sees that
\begin{align*}
\Sigma_\eqref{smallpart5d}\le& \,\sigma^2 \sum_{y\in K}\sum_{m=0}^\infty\sum_{z\in \B_\gamma^c}\bar \P_y(\tau_\gamma<T^+_K, S_{\tau_\gamma}=z)\bar\E_z\Bigg[\sum_{k=0}^m\sum_{e\in \Z^d} q_K(S_k+e)\mu(e)\ind{S_m=x}\Bigg]\\
=& \,\sigma^2 \sum_{y\in K}\sum_{z\in \B_\gamma^c}\Bigg(\sum_{m=0}^\infty\E_x\Bigg[\sum_{k=0}^m \sum_{e\in \Z^d} q_K(S_k+e)\mu(e)\ind{S_m=z}\Bigg]\Bigg)\bar\P_y(\tau_\gamma<T^+_K, S_{\tau_\gamma}=z).
\end{align*}
Here we note that 
\[
\sum_{m=0}^\infty\E_x\Bigg[\sum_{k=0}^m \sum_{e\in \Z^d} q_K(S_k+e)\mu(e)\ind{S_m=z}\Bigg]
\le \sum_{w\in \Z^d} h(w) g(x,w) g(w,z),
\]
with $h(w)=\sum_e q_K(w+e)\mu(e)$. 
We have already seen in \eqref{eq:h-bound} that $h$ satisfies the assumption of Lemma \ref{bdhgg} with $k=0$, which implies the existence of a constant $C>1$ such that 
\[
\sum_{w\in \Z^d} h(w) g(x,w) g(w,z)\le C\frac{(1+m)^{4-d}}{(1+M)^{d-2}},
\]
where $M:=\max\{\|x\|, \|z\|, \|x-z\|\}\ge \|x\|$ and $m:=\min\{\|x\|,\|z\|, \|x-z\|\}\le \|x\|$.
\begin{itemize}
  \item If $z\in \B_{2\gamma}\backslash\B_\gamma$, we have $m \ge \gamma$
  for $\|x\|$ sufficiently large. As the dimension $d\ge 5$, it follows that
  \[
  \sum_{m=0}^\infty\E_x\Bigg[\sum_{k=0}^m \sum_{e\in \Z^d} q_K(S_k+e)\mu(e)\ind{S_m=z}\Bigg]=o_x(1) \|x\|^{2-d}.
  \]
  \item Otherwise, if $z\in \B_{2\gamma}^c$, we still have the upper bound
  \[
    \sum_{m=0}^\infty\E_x\Bigg[\sum_{k=0}^m \sum_{e\in \Z^d} q_K(S_k+e)\mu(e)\ind{S_m=z}\Bigg]\leq C\|x\|^{2-d}.
  \]
\end{itemize}
Therefore, we finally obtain that
\begin{eqnarray*}
\Sigma_\eqref{smallpart5d} &\leq& o_x(1) \|x\|^{2-d} \sum_{y\in K}\sum_{z\in \B_{2\gamma}\backslash\B_\gamma}\bar \P_y(\tau_\gamma<T^+_K, S_{\tau_\gamma}=z) \\ 
& & \quad\quad  + C\|x\|^{2-d}\sum_{y\in K}\sum_{z\in \B_{2\gamma}^c}\bar \P_y(\tau_\gamma<T^+_K, S_{\tau_\gamma}=z) \\ 
 &\leq &o_x(1) \|x\|^{2-d} \sum_{y\in K}\bar \P_y(\tau_\gamma<T^+_K)+C\|x\|^{2-d} \sum_{y\in K} \bar\P_y(\J(S_{\tau_\gamma})>2\gamma).
\end{eqnarray*}
Both terms in the last line are $o_x(1) \|x\|^{2-d}$, since $\sum_{y\in K}\bar \P_y(\tau_\gamma<T^+_K)\le |K|$ and $\bar\P_y(\J(S_{\tau_\gamma})>2\gamma)=o_x(1)$ according to \eqref{eq:bigjump-asymp}. The proof of \eqref{smallpart5d} is completed.
\end{proof}

\section{Occupation times in the critical dimension $d=4$}\label{cd}
This section is devoted to the critical dimension $d=4$. Here we always assume \eqref{offspring} and that the jump law $\mu$ is symmetric with its support $\supp(\mu)$ being finite. 
Recall that the hitting probability for the critical branching random walk is given by \cite{Zhu}, which says that for any finite subset $K\subset \Z^4$, as $\|x\|\to\infty$,
\begin{equation}\label{hittingprobab4d}
u_K(x):=\P_x(Z_\T(K)\ge 1)\sim \frac{1}{2\sigma^2\J(x)^{2} \log \J(x)}.
\end{equation}
The arguments in the previous section fail to obtain this asymptotic because $\Sigma_\infty$ in \eqref{backwardspine} becomes a.s.~infinite in $\Z^4$. Let us first review the idea of the proof in \cite{Zhu}. On the one hand, the many-to-one lemma (Lemma \ref{Manyto1}) shows that
\begin{equation}\label{meanZ}
m_K(x):=\E_x[Z_\T(K)]=g(x,K)\sim \frac{\c_4 |K|}{\J(x)^{2}}.
\end{equation}
On the other hand, let $u^*_K$ be the first individual in the lexicographic order that arrives in $K$, then
\begin{align*}
\P_x(Z_\T(K)\ge 1)=&\sum_{n=0}^\infty\E_x\Bigg[ \sum_{|u|=n}\ind{u=u^*_K}\Bigg],\\
\E_x[Z_\T(K)]=&\sum_{n=0}^\infty\E_x\Bigg[ \sum_{|u|=n} \E_x\Big[\ind{u=u^*_K}Z_\T(K)\vert \mathcal{F}_n\Big]\Bigg].
\end{align*}
Recall that for any $v\in\T$, the notation $v< \L_K$ means that for all $\rho\le w\le v$, $S_w\notin K$. By change of measure and Proposition \ref{spine}, one has
\begin{align*}
&\P_x(Z_\T(K)\ge 1)=\sum_{n=0}^\infty\E_{\Q_x}\Bigg[\sum_{|u|=n}\frac{1}{Z_n}\E_x\Big[\ind{u=u^*_K}\vert\mathcal{F}_n\Big]\Bigg]\\
=&\sum_{n=0}^\infty\E_{\Q_x}\Bigg[\sum_{|u|=n}\frac{1}{Z_n}\ind{u\in\L_K}\prod_{k=1}^n\prod_{v\in\L(u_k)}\prod_{w: |w|=n, v\le w}\ind{w<\L_K}\P_{S_w}(Z_\T(K)=0)\Bigg]\\
=&\sum_{n=0}^\infty \E_{\Q^*_x}\Bigg[\sum_{|u|=n}\Q^*_x(w_n=u\vert\mathcal{F}_n)\ind{u\in\L_K}\prod_{k=1}^n\prod_{v\in\L(u_k)}\prod_{w: |w|=n, v\le w}\ind{w<\L_K}(1-u_K(S_w))\Bigg].
\end{align*}
Consequently,
\begin{equation}\label{uK}
u_K(x)=\E_{\Q^*_x}\Bigg[\ind{T_K(w)<\infty}\prod_{k=1}^{T_K(w)}\prod_{v\in \L(w_k)}(1-u_K(S_v))\Bigg].
\end{equation}
In the same spirit, we have
\begin{align*}
&\E_x[Z_\T(K)]=\sum_{n=0}^\infty \E_{\Q^*_x}\left[\ind{w_n=u^*_K}Z_\T(K)\right]\\
=\;&\sum_{n=0}^\infty \E_{\Q^*_x}\Bigg[\ind{w_n=u^*_K} \Bigl(m_K(S_{w_n})+\sum_{k=1}^n \sum_{v\in \rt(w_k)} m_K(S_v)\Bigr)\Bigg]\\
=\;&\E_{\Q^*_x}\Bigg[\ind{T_K(w)<\infty}\prod_{k=1}^{T_K(w)}\prod_{v\in \L(w_k)}(1-u_K(S_v))\times \bigg( m_K(S_{T_K(w)})+\sum_{k=1}^{T_K(w)}\sum_{v\in \rt(w_k)} m_K(S_v) \bigg)\Bigg].
\end{align*}
In fact, it is shown in \cite{Zhu} that under $\Q^*_x$, the sum 
\[
	m_K(S_{T_K(w)})+ \sum_{k=1}^{T_K(w)} \sum_{v\in \rt(w_k)} m_K(S_v)
\]
can be approximated by $2\sigma^2\c_4|K|\log \J(x)$ as $\|x\|\to\infty$. This leads to 
\[
m_K(x)=\E_x[Z_\T(K)]\sim u_K(x) 2\sigma^2\c_4|K|\log \J(x).
\]
Then \eqref{hittingprobab4d} follows immediately from \eqref{meanZ}. 

In order to study the asymptotic law of $Z_\T(K)$ (or $L_K$) under $\P_x(\cdot\vert Z_\T(K)\ge 1)$, one possible way is to calculate the moments
$\E_x[Z_\T(K)^j\vert Z_\T(K)\ge 1]$ of all integer orders $j\geq 1$. This can be done in a recursive way under the extra assumption that $\sup_{j\ge1}\sum_{k}k^j p_k<\infty$. Such results will be stated in Proposition \ref{moments}, whose proof is postponed to Appendix \ref{AppendixB}.

Our main idea of the proof is however inspired by \cite{Geiger}. For all individuals in $\L_K$ (or located in $K$), we consider their most recent common ancestor $\A_x(K)$. For convenience, we write its spatial position $H_x:=S_{\A_x(K)}\in \Z^d$.
Let $\N_x(K)$ be the set of children of $\A_x(K)$ which have descendants arriving in $K$. By the definition of $\A_x(K)$ as the most recent common ancestor, $\#\N_x(K)\ge 2$ conditioned on $L_K\ge 2$. In particular, if $L_K=1$, then $\A_x(K)\in\L_K$ and $\N_x(K)$ is the singleton $\{\A_x(K)\}$. 

\begin{figure}[!htbp]
 \begin{center}
 \includegraphics[width=0.4\textwidth]{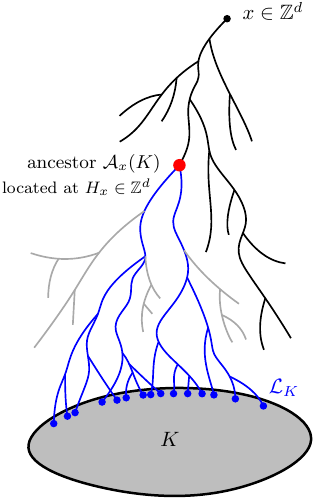}
 \caption{\textsf{The most recent common ancestor $\A_x(K)$ and its spatial location $H_x$.} \label{fig-reduced_tree}}
 \end{center}
\end{figure}

 For any $u\in \N_x(K)$, let $X_u=S_u-H_x$ represent its jump from $H_x$ and let $\T_u$ represent the subtree rooted at $u$. Then by regrouping the individuals in $\{u\in\T\colon S_u\in K\}$ according to their ancestors in $\N_x(K)$, we have
\[
Z_\T(K)=\sum_{u\in \N_x(K)} Z_{\T_u}(K).
\]
Under the conditional law $\P_x(\cdot\vert Z_\T(K)\ge 1)$, given $\{S_u=H_x+X_u, u\in \N_x(K)\}$, $Z_{\T_u}(K), u\in \N_x(K)$ are independent and distributed as $\P_{S_u}(Z_\T(K)\in \cdot\vert Z_\T(K)\ge 1)$. Then after suitable renormalisation, we get
\begin{equation}\label{eq:eq-dis-approxi}
\frac{Z_\T(K)}{\log \J(x)}=\frac{\log \J(H_x)}{\log \J(x)}\sum_{u\in \N_x(K)}\frac{Z_{\T_u}(K)}{\log \J(H_x)}.
\end{equation}
Later in Proposition \ref{MRCA}, we will show that under $\P_x(\cdot\vert Z_\T(K)\ge 1)$, 
 \[
 \#\N_x(K)\overset{\mathrm{(d)}}{\longrightarrow} 2 \quad \mbox {and} \quad \frac{\log \J(H_x)}{\log \J(x)}\overset{\mathrm{(d)}}{\longrightarrow} U
 \]
 in distribution as $\|x\|\to \infty$, where $U$ has uniform distribution on $[0,1]$. So, we will be able to compare the previous identity \eqref{eq:eq-dis-approxi} with the following equation in distribution:
 \[
 Y\overset{\mathrm{(d)}}{=}U(Y_1+Y_2)
 \]
 with $Y_1$ and $Y_2$ being i.i.d.~copies of $Y$ and independent of $U$. The unique non-negative solution of this equation is given by the family of exponential distributions (see \cite{KS}). 
 More precisely, we will check that the $L_2$-Wasserstein distance between $\frac{Z_\T(K)}{\log \J(x)}$ and $Y$ goes to zero where $Y$ has exponential distribution with a suitable parameter. This concludes our Yaglom-type theorem \ref{4d}. We emphasize that in order to use the $L_2$-Wasserstein distance, we need a finite second moment of $Z_\T(K)$. This is indeed the case under our assumption \eqref{offspring}.
 
In the next subsection, we first compute the moments of $L_K$ and of $Z_\T(K)$ in a recursive way.
 
\subsection{Moments of $L_K$ and $Z_\T(K)$} $ $

In Proposition 6.3 of \cite{AHJ21}, using a diagrammatic expansion of moments, Angel, Hutchcroft and J\'arai have proved (for simple random walk jump law) that if there exist some positive constants $c$ and $C$ such that $p_k\leq Ce^{-ck}$ for every $k\geq 1$, then for every $j\ge1$ and for every $x\in\Z^4$, 
\[
(\C_-)^j j!\frac{(j+\log \J(x))^{j-1}}{(1+\J(x))^{2}}\le \E_x[Z_\T(0)^j]\le (\C_+)^j j!\frac{(j+\log \J(x))^{j-1}}{(1+\J(x))^{2}},
\]
with some constants $0<\C_-\le\C_+<\infty$. We manage to obtain the following asymptotics for all moments of $L_K$ and $Z_\T(K)$. Our arguments are mainly inspired by the work of Fleischman \cite{Fle78}. The proof is postponed to Appendix \ref{AppendixB}.

\begin{proposition}\label{moments}
Assume that the jump law $\mu$ is symmetric and has finite support. For any integer $j\ge1$, if $\sum_{k=0}^\infty k^j p_k<\infty$, then as $\|x\|\to\infty$, we have
\begin{align}
\E_x[L_K^j]&\sim u_K(x) \left(2  \c_4 \sigma^2 \cap(K) \log \J(x)\right)^j j!,\label{momL}\\
\E_x[Z_\T(K)^j]&\sim u_K(x) \left(2\c_4\sigma^2 |K|\log \J(x)\right)^j j!.\label{momZ}
\end{align}
\end{proposition}

One could immediately deduce from Proposition \ref{moments} and Theorem 3.3.12 of \cite{Durrett} that if the offspring distribution has finite moments of all orders, then the convergence in law for $\frac{L_K}{\log \J(x)}$ (or for $\frac{Z_\T(K)}{\log\J(x)}$) under $\P_x(\cdot \vert Z_\T(K)\ge1)$ holds. However, we will see in the next subsection that the finite second moment for the offspring distribution suffices to establish the Yaglom theorem \ref{4d}.

Let us restrict to the second moment here, as stated in the following lemma. 
\begin{lemma}\label{2mom}
Assume \eqref{offspring} and that the jump law $\mu$ is symmetric and has finite support. As $\|x\|\to\infty$, we have 
\begin{align}
\E_x[L_K^2]&\sim u_K(x) 8\left(  \c_4 \sigma^2 \cap(K) \log \J(x)\right)^2, \label{2momL}\\
\E_x[Z_\T(K)^2]&\sim u_K(x) 8\left(\c_4\sigma^2 |K|\log \J(x)\right)^2.\label{2momZ}
\end{align}
Moreover, for any $y_1,y_2\in K$, we have
\begin{equation}\label{cov}
\E_x[Z_\T(y_1)Z_\T(y_2)]\sim u_K(x) 8\left(\c_4\sigma^2 |K|\log \J(x)\right)^2.
\end{equation}
\end{lemma}

It follows from \eqref{cov} that for any $y_1, y_2\in K$, 
\begin{equation}\label{sameinK}
\E_x\left[(Z_\T(y_1)-Z_\T(y_2))^2\vert Z_\T(K)\ge 1\right]=o_x(1)(\log \J(x))^2.
\end{equation}

\begin{proof}
To prove \eqref{2momL}, let us write $\ell_j(x,K):=\E_x[L_K^j]$ and $m_j(x,K):=\E_x[Z_\T(K)^j]$ for any $j\ge1$. First, by the many-to-one lemma, it is known that
\[
\ell_1(x,K)=\P_x(T_K<\infty)\textrm{ and }m_1(x,K)=g(x,K).
\]
By \eqref{hittingprobabS} and \eqref{green}, as $\|x\|\to\infty$,
\begin{equation}\label{meanL4d}
\ell_1(x,K)\sim \frac{\c_4\cap(K)}{\J(x)^2}\textrm{ and } m_1(x,K)\sim \frac{\c_4 |K|}{\J(x)^2}.
\end{equation}
Next, observe that
\begin{align*}
\ell_2(x,K)=&\E_x\Bigg[\bigg(\sum_{u\in\T}\ind{u\in\L_K}\bigg)^2\Bigg]=\E_x\Bigg[\sum_{u,v\in\T}\ind{u,v\in \L_K}\Bigg]\\
=&\ell_1(x,K)+\E_x\Bigg[\sum_{u\neq v}\ind{u,v\in\L_K}\Bigg].
\end{align*}
By considering the most recent common ancestor $w$ of the couple $(u,v)$ with $u\neq v$, we get that
\begin{equation*}
\E_x\Bigg[\sum_{u\neq v}\ind{u,v\in\L_K}\Bigg]=\E_x\Bigg[\sum_{w\in\T}\ind{w<\L_K, N_w\ge 2}\sum_{1\le i_1\neq i_2\le N_w}\bigg(\sum_{u\in \L_K}\ind{wi_1\le u}\bigg)\bigg(\sum_{v\in\L_K} \ind{wi_2\le v}\bigg)\Bigg],
\end{equation*}
where $N_w$ denotes the number of children of $w$ and $wi_1, wi_2$ are two distinct children of $w$. It then follows from the branching property that
\begin{align*}
\E_x\Bigg[\sum_{u\neq v}\ind{u,v\in\L_K}\Bigg]=&\sum_{n=0}^\infty \E_x\Bigg[\sum_{|w|=n}\ind{w<\L_K}\sum_{m=2}^\infty p_m m(m-1)\left(\sum_e\ell_1(S_w+e,K)\mu(e)\right)^2\Bigg]\\
=&\sum_{n=0}^\infty \E_x\Bigg[\sum_{|w|=n}\ind{w<\L_K}\sigma^2\left(\sum_e\ell_1(S_w+e,K)\mu(e)\right)^2\Bigg].
\end{align*}
Hence, applying the many-to-one lemma to the last expectation, one gets that
\begin{align*}
\ell_2(x,K)=&\ell_1(x,K)+\sum_{n=0}^\infty\sigma^2 \E_x\Bigg[\ind{n<T_K}\left(\sum_e\ell_1(S_n+e,K)\mu(e)\right)^2\Bigg]\\
=&\ell_1(x,K)+\sigma^2 \sum_{n=0}^\infty\E_x\left[\ind{n<T_K}h_2(S_n)\P_{S_n}(T_K<\infty)\right],
\end{align*}
where the function
\[
h_2(z): =\frac{(\sum_e \ell_1(z+e,K)\mu(e))^2}{\P_z(T_K<\infty)}=\frac{(\sum_e \ell_1(z+e,K)\mu(e))^2}{\ell_1(z,K)}.
\]
Observe that by the Markov property at time $n$,
\[
\E_x[\ind{n<T_K<\infty}h_2(S_n)]=\E_x\left[\ind{n<T_K}h_2(S_n)\P_{S_n}(T_K<\infty)\right].
\]
As a consequence, 
\begin{align*}
\ell_2(x,K)=&\ell_1(x,K)+\sigma^2\E_x\left[\sum_{n=0}^{T_K-1} h_2(S_n); T_K<\infty\right]\\
=&\P_x(T_K<\infty)+\sigma^2\E_x\left[\sum_{n=0}^{T_K-1} h_2(S_n)\Big\vert T_K<\infty\right]\P_x(T_K<\infty).
\end{align*}
Since $\mu$ has finite support, we derive from \eqref{meanL4d} that as $\|z\|\to\infty$,
\[
	\sum_e \ell_1(z+e,K)\mu(e) \sim \frac{\c_4\cap(K)}{\J(z)^2},
\]
and therefore
\[
h_2(z)\sim \frac{\c_4\cap(K)}{\J(z)^2}.
\]
The estimate \eqref{cvgsumh} of Lemma \ref{sumh} can be applied here with $k=0$. 
We thus conclude that
\[
\ell_2(x,K)\sim \frac{\c_4\cap(K)}{\J(x)^2}\sigma^2 4\c_4\cap(K)\log \J(x),
\]
which combined with \eqref{hittingprobab4d} yields \eqref{2momL}.

It remains to prove \eqref{2momZ} and \eqref{cov}. Observe that 
\[
Z_\T(K)=\sum_{u\in \T}\ind{u\in \L_K}Z_{\T_u}(K), \textrm{ and } Z_\T(y_i)=\sum_{u\in\T}\ind{u\in\L_K}Z_{\T_u}(y_i), \textrm{ for } i=1,2.
\]
Similarly as the arguments presented above, one can obtain \eqref{2momZ} and \eqref{cov} by replacing $\ind{u\in\L_K}$ by $\ind{u\in \L_K}Z_{\T_u}(K)$ and $\ind{u\in\L_K}Z_{\T_u}(y_i)$ respectively. We leave the details to the reader.  
\end{proof}

Note that if we use \eqref{bdsumh} instead of \eqref{cvgsumh} in the previous proof, we obtain the existence of some constant $C<\infty$ such that 
\begin{equation}\label{sup2momZ}
\E_x \left[Z_\T(K)^2\right]\le C \frac{\log(2+\J(x))}{(1+\J(x))^2}, \quad \forall x\in\Z^4.
\end{equation}
This upper bound will be useful later. 

\subsection{Yaglom-type limit: Theorem \ref{4d}}

\subsubsection{Most recent common ancestor $\A_x(K)$} $ $

In this part, we focus on the most recent common ancestor $\A_x(K)$ of all the individuals in $\L_K$. We will study its spatial position $H_x$ and the number of its children having descendants in $K$ which is denoted by $N_x(K):=\#\N_x(K)$. The main results are summarized in the following proposition.

\begin{proposition}\label{MRCA}
Assume \eqref{offspring} and that the jump law $\mu$ is symmetric and has finite support. Under $\P_x(\cdot\vert Z_\T(K)\ge 1)$, as $\|x\|\to\infty$, 
\begin{align}
\frac{\log \J(H_x)}{\log \J(x)}&\xrightarrow{\mathrm{(d)}} U,\label{cvgHx}\\
N_x(K)&\xrightarrow{\textrm{ \emph{in probability} }} 2,\label{cvgNx}
\end{align}
where $U$ is uniformly distributed on $(0,1)$. Moreover, $(\frac{\log_+\J(H_x)}{\log \J(x)})^2$ is uniformly integrable under $\P_x(\cdot\vert Z_\T(K)\ge 1)$.
\end{proposition}

As an immediate consequence of Proposition \ref{MRCA}, as $\|x\|\to\infty$, we have
\begin{equation}\label{2momHx}
\E_x\left[\left(\frac{\log_+\J(H_x)}{\log \J(x)}\right)^2\Big\vert Z_\T(K)\ge 1\right]\to \E\big[U^2\big]=\frac13.
\end{equation}

\begin{figure}[!htbp]
 \begin{center}
 \includegraphics[width=0.5\textwidth]{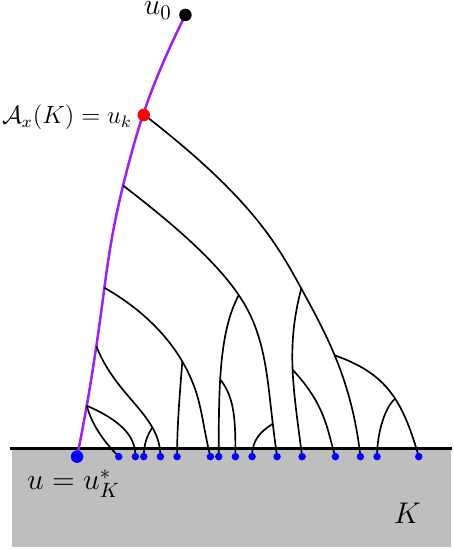}
 \caption{\textsf{The left most path entering $K$.} \label{fig-leftmost_path}}
 \end{center}
\end{figure}

To prove Proposition \ref{MRCA}, we will treat $\J(H_x)$ and $N_x(K)$ separately.  First observe that if $u=u^*_K$ is the first individual hitting $K$ in the lexicographic order, then $\A_x(K)$ lies on the ancestral line $\{u_0, u_1,\cdots, u_{|u|}\}$ of $u$. So by change of measure and Proposition \ref{spine}, similarly as in \eqref{uK}, we obtain 
\begin{multline}\label{positionAx}
\P_x(\J(H_x)\le r, Z_\T(K)\ge 1)\\
=\E_{\Q^*_x}\Bigg[\bigg(\ind{T_K(w)<\infty}\prod_{k=1}^{T_K(w)} \prod_{u\in \L(w_k)}\ind{Z_{\T_u}(K)=0}\bigg)\ind{H_x\in \B_r}\Bigg].
\end{multline}
Let us introduce for all $r>0$
\[
\sigma_r(w):=\max\{0\le k<T_K(w) \colon S_{w_k}\notin \B_r\}, 
\]
which is the last time that the random walk along the spine stays outside the ball $\B_r$ before hitting $K$. Then we observe immediately that under $\Q^*_x$, for $r>\max_{y\in K}\J(y)$, on the event $u^*_K\in \{w_n, n\ge1\}$,
\begin{equation}\label{obsHx}
\bigg\{\sum_{k=1}^{\sigma_r(w)}\sum_{v\in\rt(w_k)}Z_{\T_v}(K)=0\bigg\}=\{w_{\sigma_r(w)}<\A_x(K)\}\subset\{ H_x\in \B_r\}.
\end{equation}
In the right-hand side of \eqref{positionAx}, we will approximate $\ind{H_x\in \B_r}$ by 
\[
   \ind{\sum_{k=1}^{\sigma_r(w)}\sum_{v\in\rt(w_k)}Z_{\T_v}(K)=0}
\] 
with $r=\J(x)^a\gg 1$ with $a\in(0,1)$, thanks to the following lemma. 

Recall that $\gamma=\gamma_x=\J(x)(\log\J(x))^{-\frac14}$. When $\|x\|\to\infty$, $\gamma\gg r=\J(x)^a$.
\begin{lemma}\label{mainHx}
For $r=\J(x)^a$ with $a\in(0,1)$ and for any $\delta\in(0,1-a)$, as $\|x\|\to\infty$, 
\begin{align}
&\E_{\Q^*_x}\left[\ind{T_K(w)<\infty}\prod_{k=1}^{T_K(w)}\prod_{u\in\L(w_k)}\ind{Z_{\T_u}(K)=0}\times \ind{\min\limits_{\sigma_\gamma(w)<k\le \sigma_{r^{1+\delta/a}}(w)} \J(S_{w_k})\le r}\right]\nonumber\\
\;=\;&o_x(1) u_K(x),\label{smallpartcvgH1}\\
&\E_{\Q^*_x}\Bigg[\ind{T_K(w)<\infty}\prod_{k=1}^{T_K(w)}\prod_{u\in\L(w_k)}\ind{Z_{\T_u}(K)=0}\times\ind{\sum_{k=1}^{\sigma_\gamma(w)}\sum_{u\in\rt(w_k)} Z_{\T_u}(K)\ge1}\Bigg]\nonumber\\
\;=\;&o_x(1)u_K(x).\label{smallpartcvgH2}
\end{align}
Moreover, we have
\begin{multline}\label{cvgpositionH}
\E_{\Q^*_x}\Bigg[\ind{T_K(w)<\infty}\prod_{k=1}^{T_K(w)}\prod_{u\in\L(w_k)}\ind{Z_{\T_u}(K)=0}\times\ind{\sum_{k=1+\sigma_\gamma(w)}^{\sigma_r(w)}\sum_{u\in\rt(w_k)} Z_{\T_u}(K)=0}\Bigg]\\
=(a+o_x(1))u_K(x).
\end{multline}
\end{lemma}

The proof of this technical lemma is postponed to Appendix \ref{AppendixB}. Now we are ready to prove \eqref{cvgHx} in Proposition \ref{MRCA}. 

\begin{proof}[Proof of \eqref{cvgHx}]
On the one hand, by \eqref{positionAx} and \eqref{obsHx}, for $r=\J(x)^a$ with $a\in(0,1)$, one has
\begin{align*}
&\P_x(\J(H_x)\le r, Z_\T(K)\ge 1)\\
\;\ge\; &\E_{\Q^*_x}\left[\ind{T_K(w)<\infty}\prod_{k=1}^{T_K(w)} \prod_{u\in \L(w_k)}\ind{Z_{\T_u}(K)=0}\ind{\sum_{k=1}^{\sigma_r(w)}\sum_{v\in\rt(w_k)}Z_{\T_v}(K)=0}\right]\\
\;\ge\; & \E_{\Q^*_x}\left[\ind{T_K(w)<\infty}\prod_{k=1}^{T_K(w)}\prod_{u\in\L(w_k)}\ind{Z_{\T_u}(K)=0}\times\ind{\sum_{k=1+\sigma_\gamma(w)}^{\sigma_r(w)}\sum_{u\in\rt(w_k)} Z_{\T_u}(K)=0}\right]\\
&\quad -\E_{\Q^*_x}\left[\ind{T_K(w)<\infty}\prod_{k=1}^{T_K(w)}\prod_{u\in\L(w_k)}\ind{Z_{\T_u}(K)=0}\times\ind{\sum_{k=1}^{\sigma_\gamma(w)}\sum_{u\in\rt(w_k)} Z_{\T_u}(K)\ge1}\right].
\end{align*}
Then we deduce from \eqref{smallpartcvgH2} and \eqref{cvgpositionH} that 
\[
\liminf_{\|x\|\to\infty}\P_x(\J(H_x)\le \J(x)^a\vert Z_\T(K)\ge 1)\ge a.
\]

\begin{figure}[!htbp]
 \begin{center}
 \includegraphics[width=0.6\textwidth]{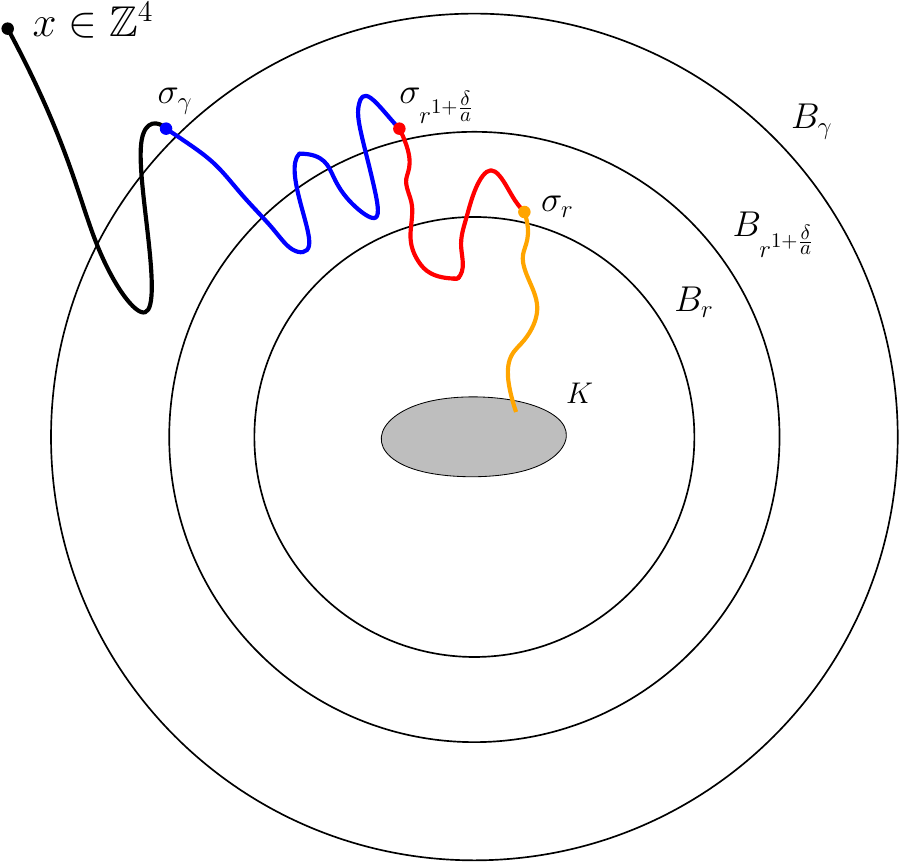}
 \caption{\textsf{Spine behavior in the space $\Z^4$.} \label{fig-spine_walk}}
 \end{center}
\end{figure}

On the other hand, for the upper bound one sees that for any small $\delta\in(0,1-a)$,
\begin{align}
&\P_x(\J(H_x)\le r, Z_\T(K)\ge 1)\nonumber\\
\;\le\; & \E_{\Q^*_x}\left[\ind{T_K(w)<\infty}\prod_{k=1}^{T_K(w)}\prod_{u\in\L(w_k)}\ind{Z_{\T_u}(K)=0}\ind{\min\limits_{\sigma_\gamma(w)<k\le \sigma_{r^{1+\delta/a}}(w)} \J(S_{w_k})\le r}\right]\nonumber\\
&+\E_{\Q^*_x}\left[\ind{T_K(w)<\infty}\prod_{k=1}^{T_K(w)}\prod_{u\in\L(w_k)}\ind{Z_{\T_u}(K)=0} \ind{\A_x(K)\le w_{\sigma_\gamma(w)} }\right] \nonumber\\
&+\E_{\Q^*_x}\left[\ind{T_K(w)<\infty}\prod_{k=1}^{T_K(w)}\prod_{u\in\L(w_k)}\ind{Z_{\T_u}(K)=0} \ind{ \A_x(K)>w_{\sigma_{r^{1+\delta/a}}(w)} }\right]. \label{upbdcvgHx}
\end{align}
Here, we observe that
\begin{align*}
&\E_{\Q^*_x}\left[\ind{T_K(w)<\infty}\prod_{k=1}^{T_K(w)}\prod_{u\in\L(w_k)}\ind{Z_{\T_u}(K)=0}\ind{ \A_x(K)\le w_{\sigma_\gamma(w)}}\right]\\
\;\le\; & \E_{\Q^*_x}\left[\ind{T_K(w)<\infty}\prod_{k=1}^{T_K(w)}\prod_{u\in\L(w_k)}\ind{Z_{\T_u}(K)=0}\times\ind{\sum_{k=1}^{\sigma_\gamma(w)}\sum_{u\in\rt(w_k)} Z_{\T_u}(K)\ge1}\right],
\end{align*}
which is $o_x(1)u_K(x)$ by \eqref{smallpartcvgH2}. Applying it and \eqref{smallpartcvgH1} to \eqref{upbdcvgHx} yields that 
\begin{align*}
&\P_x(\J(H_x)\le r, Z_\T(K)\ge 1)\le o_x(1) u_K(x)\nonumber\\
 & +\E_{\Q^*_x}\left[\ind{T_K(w)<\infty}\prod_{k=1}^{T_K(w)}\prod_{u\in\L(w_k)}\ind{Z_{\T_u}(K)=0}\ind{\sum_{k=1+\sigma_\gamma(w)}^{\sigma_{r^{1+\delta/a}}(w)}\sum_{v\in\rt(w_k)}Z_{\T_v}(K)=0}\right].
\end{align*}
Applying \eqref{cvgpositionH} for $r^{1+\delta/a}=\J(x)^{a+\delta}$ implies that
\begin{multline*}
\E_{\Q^*_x}\left[\ind{T_K(w)<\infty}\prod_{k=1}^{T_K(w)}\prod_{u\in\L(w_k)}\ind{Z_{\T_u}(K)=0}\ind{\sum_{k=1+\sigma_\gamma(w)}^{\sigma_{r^{1+\delta/a}}(w)}\sum_{v\in\rt(w_k)}Z_{\T_v}(K)=0}\right] \\
= (a+\delta +o_x(1))u_K(x).
\end{multline*}
It then follows that for any small $\delta\in(0,1-a)$, 
\[
\limsup_{\|x\|\to\infty}\P_x(\J(H_x)\le \J(x)^a\vert Z_\T(K)\ge 1)\le a+\delta,
\]
which suffices to conclude \eqref{cvgHx}. 
\end{proof}

Next, let us check the uniform integrability of $(\frac{\log_+\J(H_x)}{\log \J(x)})^2$ under $\P_x(\cdot\vert Z_\T(K)\ge 1)$, for $\|x\|\gg1$. We define the radius of $K$ as $\rad(K):=\sup_{y\in K}\|y\|$.

\begin{proof}[Uniform integrability]
For $R:=\J(x)^a>\rad(K)$ with $a>1$, observe that 
\begin{align*}
&\P_x(\J(H_x)> \J(x)^a, \Z_\T(K)\ge 1)\le  \E_x\left[\sum_{u\in\T}\ind{u=\A_x(K), S_u\notin \B_R}\right]\\
\;\le\; & \E_x\Bigg[\sum_{u\in\T}\ind{u<\L_K, S_u\notin\B_R}\ind{\sum_{j=1}^{N_u}\ind{Z_{\T_{uj}}(K)\ge 1}\ge2}\Bigg]\\
\;\le\; &\sum_{n=0}^\infty \E_x\Bigg[\sum_{|u|=n}\ind{u<\L_K, S_u\notin\B_R}\sum_{m=2}^\infty p_m \binom{m}{2} \bigg(\sum_eu_K(S_u+e)\mu(e)\bigg)^2 \Bigg].
\end{align*}
Recall that the sum $\sum_{m=2}^\infty p_m \binom{m}{2}$ is given by $\sigma^2/2$. Set $\textbf{h}_2(z):=(\sum_e u_K(z+e)\mu(e))^2$. By the many-to-one lemma, we get that 
\begin{align*}
\P_x(\J(H_x)> \J(x)^a, \Z_\T(K)\ge 1)
\le & \frac{\sigma^2}{2}\E_x\left[\sum_{n=0}^\infty \ind{n<T_K, S_n\notin\B_R}\textbf{h}_2(S_n)\right]\\
\le &\frac{\sigma^2}{2} \sum_{z\notin\B_R} \textbf{h}_2(z) g(x,z).
\end{align*}
Note that by \eqref{hittingprobab4d}, there exists some constant $c<\infty$ such that for all $z\in \Z^4$ satisfying $\J(z)>\J(x)^a$ with $a>1$ and $\|x\|\gg1$,
\[
\textbf{h}_2(z)\le \frac{c}{(\log \J(z))^2\J(z)^4}.
\]
Hereafter the constant $c$ may change its value from line to line.
Using the last estimate and \eqref{green}, we obtain that
\[
\P_x(\J(H_x)> \J(x)^a, \Z_\T(K)\ge 1)\le \frac{c}{a^2(\log\J(x))^2 \J(x)^{2a}}.
\]
Thus, the conditional probability satisfies 
\[
\P_x(\J(H_x)> \J(x)^a\vert \Z_\T(K)\ge 1)\le  \frac{c}{a^2(\log\J(x)) \J(x)^{2a-2}}.
\]
This implies that for any $a>1$ and for all $\|x\|\gg1$,
\[
\E_x\left[\left(\frac{\log_+ \J(H_x)}{\log \J(x)}\right)^2\ind{\frac{\log_+ \J(H_x)}{\log \J(x)}>a}\Big\vert \Z_\T(K)\ge 1\right] \le \frac{c}{\J(x)^{2(a-1)}},
\]
which goes to zero as $a\to\infty$. This concludes the uniform integrability of $(\frac{\log_+\J(H_x)}{\log \J(x)})^2$ under $\P_x(\cdot\vert Z_\T(K)\ge 1)$.
\end{proof}

Finally, it remains to study $N_x(K)$.

\begin{proof}[Proof of \eqref{cvgNx}.]
Obviously, $N_x(K)\ge 2$ as long as $Z_\T(K)\ge 2$. If $Z_\T(K)=1$, one has $H_x\in K$. Hence,
\[
  \P_x(N_x(K)=1\vert Z_\T(K)\ge 1) \leq \P_x(H_x\in K \vert Z_\T(K)\ge 1).
\]
It then follows from \eqref{cvgHx} that
\begin{align*} 
&\limsup_{\|x\|\to\infty}\P_x(N_x(K)=1\vert Z_\T(K)\ge 1) \\
\le &\limsup_{a\to 0+}\lim_{\|x\|\to\infty}\P_x(\J(H_x)\le \J(x)^a \vert Z_\T(K)\ge 1)=0.
\end{align*}
We only need to show that $\lim_{\|x\|\to\infty}\P_x(N_x(K)\ge 3\vert Z_\T(K)\ge 1)=0$. First, observe that for any $\delta>0$,
\begin{align}
&\P_x(N_x(K)\ge 3\vert Z_\T(K)\ge 1)\nonumber\\
\le \;& \P_x\left(\frac{\log\J(H_x)}{\log \J(x)}\le \delta\vert Z_\T(K)\ge 1\right)\nonumber\\
&+ \frac{1}{\delta}\E_x\left[\frac{\log\J(H_x)}{\log \J(x)}(N_x(K)-2)_+\ind{\frac{\log\J(H_x)}{\log \J(x)}\ge \delta}\vert Z_\T(K)\ge 1\right].\nonumber
\end{align}
Again, thanks to \eqref{cvgHx}, it suffices to show that 
\begin{equation}\label{upbdNx}
\limsup_{\delta\to0+}\limsup_{\|x\|\to\infty}\frac{1}{\delta}\E_x\left[\frac{\log\J(H_x)}{\log \J(x)}(N_x(K)-2)_+\ind{\frac{\log\J(H_x)}{\log \J(x)}\ge \delta}\vert Z_\T(K)\ge 1\right]=0.
\end{equation}
Let us decompose $Z_\T(K)$ via the most recent common ancestor $\A_x(K)$. To this end, we let $(Z^+(x,K), H_x^+, N_x^+(K), (S_j^+)_{1\le j\le N_x^+(K)})$ be a random vector distributed as $(Z_\T(K), H_x, N_x(K), (S_{u})_{u\in\N_x(K)})$ under $\P_x(\cdot\vert Z_\T(K)\ge 1)$. 
Then,
\begin{equation}\label{decompZ}
Z^+(x,K)\overset{\mathrm{(d)}}{=}\sum_{j=1}^{N_x^+(K)} Z_j^+(S^+_j, K),
\end{equation}
where given $(S_j^+)_{1\le j\le N^+_x(K)}$, the $Z_j^+(S^+_j, K), 1\le j\le N^+_x(K)$ are independent random variables distributed as $Z^+(y, K)\vert_{y=S^+_j}$. Taking expectation on both sides of \eqref{decompZ} yields that
\begin{equation}\label{decompZex}
\frac{\E_x[Z_\T(K)]}{\P_x(Z_\T(K)\ge 1)}=\frac{m_1(x,K)}{u_K(x)}=\E_x\left[\sum_{j=1}^{N^+_x(K)}\frac{m_1(S_j^+,K)}{u_K(S^+_j)}\right].
\end{equation}
For convenience, let us write $m_1^+(x,K)$ for $\frac{m_1(x,K)}{u_K(x)}$. By \eqref{hittingprobab4d} and \eqref{meanL4d}, we already see that
\begin{equation}\label{meanZ+}
m_1^+(x,K)=(1+o_x(1))\c_4 2\sigma^2|K|\log\J(x).
\end{equation}
At the same time, we derive from \eqref{decompZex} that
\begin{align}\label{meandecompZ}
&\frac{m_1^+(x,K)}{\c_4 2\sigma^2|K|\log\J(x)}=\E_x\Bigg[\frac{1}{\log \J(x)}\sum_{u\in \N_x(K)}\frac{m_1^+(S_u, K)}{\c_4 2\sigma^2|K|}\Big\vert Z_\T(K)\ge 1\Bigg]\nonumber\\
\ge\; &\E_x\Bigg[\frac{\log \J(H_x)}{\log \J(x)}\ind{\frac{\log \J(H_x)}{\log \J(x)}\ge \delta}\sum_{u\in \N_x(K)}\frac{m_1^+(S_u, K)}{\c_4 2\sigma^2|K|\log\J(H_x)}\bigg\vert Z_\T(K)\ge 1\Bigg]
\end{align}
for any $\delta>0$. Note that as $\sup_{u\in\N_x(K)}\|S_u-H_x\|\le \rad(\supp (\mu))<\infty$ and $\J(H_x)\ge \J(x)^\delta\gg \rad(K)$, we have 
\[
\bigg\{\frac{\log \J(H_x)}{\log \J(x)}\ge \delta \bigg\}= \bigg\{\frac{\log \J(H_x)}{\log \J(x)}\ge \delta, N_x(K)\ge 2 \bigg\},
\]
and by \eqref{meanZ+}, 
\[
  \frac{m_1^+(S_u, K)}{\c_4 2\sigma^2|K|\log\J(H_x)}=1+o_x(1), \quad \forall u\in\N_x(K).
\]
Then \eqref{meandecompZ} becomes that
\begin{align*}
&\frac{m_1^+(x,K)}{\c_4 2\sigma^2|K|\log\J(x)}=1+o_x(1)\\
&\ge (1+o_x(1))\E_x\left[\frac{\log \J(H_x)}{\log \J(x)}\ind{\frac{\log \J(H_x)}{\log \J(x)}\ge \delta, N_x(K)\ge2}N_x(K)\Big\vert Z_\T(K)\ge 1\right]\\
&= 2(1+o_x(1))\E_x\left[\frac{\log \J(H_x)}{\log \J(x)}\ind{\frac{\log \J(H_x)}{\log \J(x)}\ge \delta}\Big\vert Z_\T(K)\ge 1\right]\\
&+(1+o_x(1))\E_x\left[\frac{\log \J(H_x)}{\log \J(x)}\ind{\frac{\log \J(H_x)}{\log \J(x)}\ge \delta}(N_x(K)-2)_+\Big\vert Z_\T(K)\ge 1\right].
\end{align*}{}
Because of \eqref{cvgHx} and the uniform integrability of $\frac{\log \J(H_x)}{\log \J(x)}$, we have
\[
2\E_x\left[\frac{\log \J(H_x)}{\log \J(x)}\ind{\frac{\log \J(H_x)}{\log \J(x)}\ge \delta}\Big\vert Z_\T(K)\ge 1\right]\xrightarrow{\|x\|\to\infty} \E\big[2U\ind{U\ge \delta}\big]=1-\delta^2.
\]
As a result, we obtain that
\begin{align*}
&(1+o_x(1))\E_x\left[\frac{\log \J(H_x)}{\log \J(x)}\ind{\frac{\log \J(H_x)}{\log \J(x)}\ge \delta}(N_x(K)-2)_+\Big\vert Z_\T(K)\ge 1\right]\\
&\le  1+o_x(1)-(1-\delta^2+o_x(1))=\delta^2+o_x(1),
\end{align*}
which implies \eqref{upbdNx}. We hence finish the proof of \eqref{cvgNx}.
\end{proof}

\subsubsection{$L_2$-Wasserstein distance on probability measures} $ $

In this subsection, we recall the definition of $L_2$-Wasserstein distance on probability measures and its basic properties, which will be used in the proof of Theorem \ref{4d}. Let $E$ be a separable Banach space with norm $|\cdot|$, and let $\mathcal{P}(E)$ denote the set of all probability measures on the Borel $\sigma$-algebra of $E$. In this work, we will take $E=\R$ or $E=\R^d$ with the Euclidean norm. Define
\[
\mathcal{P}_2(E):=\Big\{\mu\in\mathcal{P}(E)\colon \int_E |x|^2\mu(dx)<\infty\Big\}.
\]
For any $\mu,\nu\in\mathcal{P}_2(E)$, the $L_2$-Wasserstein distance (also known as the Mallows distance) between $\mu$ and $\nu$ is defined by
\begin{equation}\label{d2}
d_2(\mu,\nu):=\inf_{X\sim\mu, Y\sim\nu}\sqrt{\E[|X-Y|^2]}
\end{equation}
where the infimum is over all pairs of random variables $(X,Y)\in E^2$ with $X$ distributed as $\mu$ and $Y$ distributed as $\nu$ under $\P$. We recall some basic properties of the $L_2$-Wasserstein distance in the next lemma (see for example the standard reference of Villani~\cite{Villani}).
\begin{lemma}\label{Mallows}
For the $L_2$-Wasserstein distance, the following assertions hold.
\begin{enumerate}
\item The infimum in \eqref{d2} is attained by some $(X,Y)$, called the optimal coupling of $\mu$ and $\nu$.
\item $d_2$ is a metric on $\mathcal{P}_2(E)$.
\item For a sequence of probability measures $(\mu_n)_{n\geq 1}$ and $\mu$ in $\mathcal{P}_2(E)$, 
\begin{align*}
 d_2(\mu_n,\mu)\xrightarrow[n\to \infty]{} 0 & \textrm{ if and only if } \\
& \mu_n\xrightarrow[n\to \infty]{w} \mu \textrm{ and } \int_E |x|^2\mu_n(dx)\xrightarrow[n\to \infty]{} \int_E |x|^2\mu(dx).
\end{align*}
\end{enumerate}
\end{lemma}

Here $\xrightarrow{w}$ means the weak convergence of probability measures. Lemma \ref{Mallows} shows in particular that the convergence in $(\mathcal{P}_2(E), d_2)$ implies the convergence in law. On the other hand, Proposition \ref{MRCA}, together with \eqref{2momHx}, implies that the $L_2$-Wasserstein distance between the law of $\frac{\log\J(H_x)}{\log \J(x)}$ under $\P_x(\cdot\vert Z_\T(K)\ge 1)$ and the uniform distribution on $(0,1)$ converges to zero as $\|x\|\to\infty$.

For any random variable $\xi$, we use $\L(\xi)$ to denote its law. 

\begin{lemma}\label{attainabled2}
Suppose that the random variables $(\xi_i)_{i\ge0}$ all have finite second moment. Then we can find a sequence of random variables $(\eta_i)_{i\ge0}$ such that $\L(\eta_i)=\L(\xi_i)$ for all $i\ge0$, and that
\[
d_2(\L(\xi_i),\L(\xi_0))=\sqrt{\E[(\eta_i-\eta_0)^2]}, \quad \forall i\ge 1.
\]
\end{lemma}

\begin{proof}
  We first take the optimal coupling $(\eta_1, \eta_0)$ for $\L(\xi_1)$ and $\L(\xi_0)$. Then using the gluing lemma (see Lemma 7.6 in \cite{Villani}), for any integer $n\geq 2$, we can find by recurrence a finite sequence of random variables $(\eta_i)_{i\le n}$ such that $\L(\eta_i)=\L(\xi_i)$ for all $i\le n$, and that
  \[
  d_2(\L(\xi_i),\L(\xi_0))=\sqrt{\E[(\eta_i-\eta_0)^2]}, \quad \forall i\le n.
  \]
  Finally, the full result follows by the Kolmogorov extension theorem.
\end{proof}

\subsubsection{Proof of Theorem \ref{4d}} $ $

\vspace{0.2cm}

\noindent \textbf{Definition.} For all $x\in \Z^4$, let 
\[
  (L^+(x,K), Z^+(x,K), H_x^+, N_x^+(K), (S_j^+)_{1\le j\le N_x^+(K)})
\]
be a random vector defined under $\P$ and distributed as 
\[
  (L_K, Z_\T(K), H_x, N_x(K), (S_{u})_{u\in\N_x(K)})
\] 
under $\P_x(\cdot\vert Z_\T(K)\ge 1)$.  

Then Theorem \ref{4d} can be deduced from the following convergence
\begin{equation}\label{maincvgd2}
d_2\left(\L\left(\frac{L^+(x,K)}{2\sigma^2\c_4\cap(K)\log \J(x)}, \frac{Z^+(x,K)}{2\sigma^2\c_4|K|\log \J(x)}\right), \L(Y,Y)\right)\xrightarrow{\|x\|\to\infty} 0,
\end{equation}
where $Y$ is an exponential random variable with parameter $1$. 

For simplicity, we will only present the proof for the convergence
\begin{equation}\label{keycvgd2}
d_2\left(\L\left(\frac{Z^+(x,K)}{2\sigma^2\c_4|K|\log \J(x)}\right), \L(Y)\right)\xrightarrow{\|x\|\to\infty}0.
\end{equation}
The joint convergence \eqref{maincvgd2} can be derived in the same way, starting from the following identity similar to \eqref{decompZ},
\[
(L^+(x,K), Z^+(x,K))\overset{\mathrm{(d)}}{=}\sum_{j=1}^{N^+_x(K)}(L_j^+(S_j^+, K), Z^+_j(S_j^+,K)),
\]
where  given $(S_j^+)_{1\le j\le N^+_x(K)}$, the $(L_j^+(S_j^+,K), Z_j^+(S^+_j, K)), 1\le j\le N^+_x(K)$ are independent random variables distributed as $(L^+(y,K), Z^+(y, K))\vert_{y=S^+_j}$.

Let $\J_+(x)=2+\J(x)$ so that $\log \J_+(x) \ge \log 2>0$. Define
\begin{equation}\label{eq:defi-U_x^+}
  U_x^+:=\frac{\log \J_+(H^+_x)}{\log \J_+(x)}.
\end{equation}

\begin{proof}[Proof of \eqref{keycvgd2}] 
We write
\[
b_x:=d_2\left(\L\left(\frac{Z^+(x,K)}{2\sigma^2\c_4|K|\log \J_+(x)}\right), \L(Y)\right)^2, \quad \forall  x\in \Z^4.
\]
As $\frac{\log\J(x)}{\log\J_+(x)}=1+o_x(1)$, it suffices to show that $b_x$ goes to zero as $\|x\|\to\infty$.

First of all,  $b_x$ is well-defined since $\E[(\frac{Z^+(x,K)}{2\sigma^2\c_4|K|\log \J_+(x)})^2]<\infty$. In fact, it follows from \eqref{hittingprobab4d} and \eqref{sup2momZ} that
\begin{equation}\label{supZx}
\overline{\mathbf{m}}_2:=\sup_{x\in\Z^4}\E\left[\left(\frac{Z^+(x,K)}{2\sigma^2\c_4|K|\log \J_+(x)}\right)^2\right]<\infty.
\end{equation}
This implies that
\begin{equation}\label{supbx}
\overline{b}:=\sup_{x\in\Z^4}b_x<\infty.
\end{equation}

Next, we consider the decomposition \eqref{decompZ}, where $N_x^+(K)=2$ with high probability according to \eqref{cvgNx}. 
Let  $(Z^*(x,K), H_x^*, (S^*_1, S^*_2))$ be the random vector distributed as $(Z_\T(K), H_x, (S_u)_{u\in\N_x(K)})$ under $\P_x(\cdot\vert Z_\T(K)\ge 1, N_x(K)=2)$. Then 
\begin{equation}\label{2subtree}
\frac{Z^*(x,K)}{2\sigma^2\c_4|K|\log \J_+(x)}\overset{\mathrm{(d)}}{=}U^*_x\left(\frac{Z^+_1(S_1^*, K)}{2\sigma^2\c_4|K|\log \J_+(H^*_x)}+\frac{Z^+_2(S_2^*,K)}{2\sigma^2\c_4|K|\log \J_+(H^*_x)}\right).
\end{equation}
where $U^*_x:=\frac{\log \J_+(H_x^*)}{\log \J_+(x)}$ and, given $(U^*_x, S_1^*, S_2^*)$, the $Z_j^+(S_j^*,K)$, $j=1,2$ are independent random variables distributed as $Z^+(y,K)\vert_{y=S_j^*}$ respectively.

We are going to replace $Z^+(x,K)$ in $b_x$ by $Z^*(x,K)$ in order to use \eqref{2subtree}. Now we claim that 
\begin{align}
\sup_{x\in K^c}\E\left[\left(\frac{Z^*(x,K)}{2\sigma^2\c_4|K|\log \J_+(x)}\right)^2\right]+\sup_{x\in K^c}\E[(U^*_x)^2]<\;&\infty,\label{supUx}\\
d_2\left(\L\left(\frac{Z^+(x,K)}{2\sigma^2\c_4|K|\log \J_+(x)}\right),\L\left(\frac{Z^*(x,K)}{2\sigma^2\c_4|K|\log \J_+(x)}\right)\right)=\;&o_x(1),\label{2subtreeZ}\\
d_2(\L(U^*_x), \L(U))=\;&o_x(1),\label{2subtreeU}
\end{align}
where $U$ is uniformly distributed in $(0,1)$. The proof of \eqref{supUx}, \eqref{2subtreeZ} and \eqref{2subtreeU} will be given at the end of this section. 

By Lemma \ref{attainabled2}, we can choose some version of $(U^*_x)_{x\in \Z^4}$ and $U$ such that 
\begin{equation}\label{d2U}
d_2(\L(U^*_x), \L(U))^2=\E[(U^*_x-U)^2], \quad \forall x\in\Z^4.
\end{equation}
For the same reason, for each $j=1$ or 2, we can choose independently some version of $(Z^+_j(x,K))_{x,j}$ and $Y_j$ such that for any $x\in \Z^4$,
\[
b_x=d_2\left(\L\left(\frac{Z^+_j(x,K)}{2\sigma^2\c_4|K|\log \J_+(x)}\right),\L(Y_j)\right)^2=\E\Bigg[\left(\frac{Z^+_j(x,K)}{2\sigma^2\c_4|K|\log \J_+(x)}-Y_j\right)^2\Bigg].
\]
Here, $(Z^+_1(x,K))_{x}$ and $(Z^+_2(x,K))_{x}$ are i.i.d.~copies of $(Z^+(x,K))_x$, together with $Y_1$ and $Y_2$ being i.i.d.~copies of $Y$. There is independence between $((Z^+_1(x,K))_{x},Y_1)$, $((Z^+_2(x,K))_{x},Y_2)$ and $((U_x^*, S^*_1, S^*_2)_x, U)$.
Recall that $\L(Y)=\L(U(Y_1+Y_2))$. 

Now we are ready to bound $b_x$. In view of \eqref{2subtreeZ} and \eqref{2subtree}, one sees that
\begin{align}
&b_x =  d_2\left(\L\left(\frac{Z^*(x,K)}{2\sigma^2\c_4|K|\log \J_+(x)}\right), \L(Y)\right)^2+o_x(1)\nonumber\\
= \;&d_2\left(\L\left(U^*_x\left(\frac{Z^+_1(S_1^*, K)}{2\sigma^2\c_4|K|\log \J_+(H^*_x)}+\frac{Z^+_2(S_2^*,K)}{2\sigma^2\c_4|K|\log \J_+(H^*_x)}\right)\right), \L(U(Y_1+Y_2))\right)^2 +o_x(1)\nonumber\\
\le \;&\E\left[\left(U^*_x\left(\frac{Z^+_1(S_1^*, K)}{2\sigma^2\c_4|K|\log \J_+(H^*_x)}+\frac{Z^+_2(S_2^*,K)}{2\sigma^2\c_4|K|\log \J_+(H^*_x)}\right)-U(Y_1+Y_2)\right)^2\right]+o_x(1).\nonumber
\end{align}
It thus follows that
\begin{align}
b_x&\le \sum_{j=1}^2\E\left[(U^*_x\frac{Z^+_j(S_j^*, K)}{2\sigma^2\c_4|K|\log \J_+(H^*_x)}- UY_j)^2\right]\nonumber\\
&+2\E\left[\left(U^*_x\frac{Z^+_1(S_1^*, K)}{2\sigma^2\c_4|K|\log \J_+(H^*_x)}- UY_1\right)\left(U^*_x\frac{Z^+_2(S_2^*, K)}{2\sigma^2\c_4|K|\log \J_+(H^*_x)}- UY_2\right)\right]+o_x(1).\label{upbdbx}
\end{align}
Recall that given $(U_x^*, S_1^*, S_2^*, U)$, the random vectors $(Z^+_j(S_j^*,K), Y_j)$, $j=1,2$ are independent, and they satisfy that for all $j=1,2$,
\begin{align*}
\E\big[Z_j^+(S_j^*,K)\vert U_x^*, S_1^*, S_2^*, U\big]=& \,m^+_1(S_j^*, K),\\
\E\big[Y_j\vert U_x^*, S_1^*, S_2^*, U\big]=& \,\E\big[Y_j]=1.
\end{align*}
This tells us that
\begin{align*}
&\,\E\left[(U^*_x\frac{Z^+_1(S_1^*, K)}{2\sigma^2\c_4|K|\log \J_+(H^*_x)}- UY_1)(U^*_x\frac{Z^+_2(S_2^*, K)}{2\sigma^2\c_4|K|\log \J_+(H^*_x)}- UY_2)\right]\\
=&\,\E\left[(U^*_x\frac{m_1^+(S_1^*,K)}{2\sigma^2\c_4|K|\log\J_+(H^*_x)}-U)(U^*_x\frac{m_1^+(S_2^*,K)}{2\sigma^2\c_4|K|\log\J_+(H^*_x)}-U)\right]\\
=&\,\E\big[(U^*_x\delta_1+(U^*_x-U))(U^*_x\delta_2+(U^*_x-U))\big],
\end{align*}
where $\delta_j:=\frac{m_1^+(S_j^*,K)}{2\sigma^2\c_4|K|\log\J_+(H^*_x)}-1$, for $j=1,2$. Using the elementary fact that $(a_1+b)(a_2+b)\le a_1^2+a^2_2+2b^2$ for any $a_1,a_2,b\in\R$, we get that
\begin{align*}
&\E\left[(U^*_x\frac{Z^+_1(S_1^*, K)}{2\sigma^2\c_4|K|\log \J_+(H^*_x)}- UY_1)(U^*_x\frac{Z^+_2(S_2^*, K)}{2\sigma^2\c_4|K|\log \J_+(H^*_x)}- UY_2)\right]\\
\le \;& \E\left[(U^*_x\delta_1)^2\right]+ \E\left[(U^*_x\delta_2)^2\right]+2\E\left[(U^*_x-U)^2\right]\\
=\;&\E\left[(U^*_x\delta_1)^2\right]+ \E\left[(U^*_x\delta_2)^2\right]+2d_2(\L(U^*_x), \L(U))^2
\end{align*}
by \eqref{d2U}.
Note that $\sup_{j=1,2}\|S_j^*-H_x^*\|\le \rad(\supp (\mu))<\infty$. So, by \eqref{meanZ+}, on the event $\{\J_+(H_x^*)\ge \J_+(x)^\delta\}=\{U^*_x\ge \delta\}$ with some $\delta>0$ and $\|x\|\gg1$,
\[
\frac{m_1^+(S_j^*,K)}{2\sigma^2\c_4|K|\log\J_+(H^*_x)}=\frac{m_1^+(S_j^*,K)}{2\sigma^2\c_4|K|\log\J(S_j^*)}\frac{\log \J(S_j^*)}{\log \J_+(H_x^*)}=1+o_x(1), 
\]
for $j=1,2$. On the other hand, by \eqref{hittingprobab4d} and \eqref{meanZ}, one always has the uniform bound
\[
\sup_{j=1,2}\frac{m_1^+(S_j^*,K)}{2\sigma^2\c_4|K|\log\J_+(H^*_x)}\le \sup_{x\in\Z^4, e\in \supp(\mu)}\frac{m^+_1(x+e,K)}{2\sigma^2\c_4|K|\log\J_+(x)}=:C<\infty.
\]
As a consequence, 
\begin{align*}
&\E\left[(U^*_x\frac{Z^+_1(S_1^*, K)}{2\sigma^2\c_4|K|\log \J_+(H^*_x)}- UY_1)(U^*_x\frac{Z^+_2(S_2^*, K)}{2\sigma^2\c_4|K|\log \J_+(H^*_x)}- UY_2)\right]\\
\le\; & 2 \E[(U^*_x)^2\ind{U^*_x\ge \delta}]o_x(1)+ 2(C+1)^2 \E[(U^*_x)^2\ind{U^*_x<\delta}]+2d_2(\L(U^*_x),\L(U))^2,
\end{align*}
which vanishes as $\|x\|\to\infty$ and then $\delta\to0+$ because of \eqref{2subtreeU} and Lemma \ref{Mallows}. 
Going back to \eqref{upbdbx}, we obtain that
\begin{align}
b_x\le & \sum_{j=1}^2\E\Bigg[\bigg(U^*_x\frac{Z^+_j(S_j^*, K)}{2\sigma^2\c_4|K|\log \J_+(H^*_x)}- UY_j\bigg)^2\Bigg]+o_x(1)\nonumber\\
=&\sum_{j=1}^2\E\Bigg[\bigg(U^*_x\bigg(\frac{Z^+_j(S_j^*, K)}{2\sigma^2\c_4|K|\log \J_+(H^*_x)}-Y_j\bigg)+(U^*_x-U)Y_j\bigg)^2\Bigg]+o_x(1).\label{goodupbdbx}
\end{align}
For any $j\in\{1,2\}$, observe that 
\begin{align*}
&\E\Bigg[\bigg(U^*_x\bigg(\frac{Z^+_j(S_j^*, K)}{2\sigma^2\c_4|K|\log \J_+(H^*_x)}-Y_j\bigg)+(U^*_x-U)Y_j\bigg)^2\Bigg]\\
=&\,\E\Bigg[(U^*_x)^2\bigg(\frac{Z^+_j(S_j^*, K)}{2\sigma^2\c_4|K|\log \J_+(H^*_x)}-Y_j\bigg)^2\Bigg]+\E\big[(U^*_x-U)^2Y_j^2\big]\\
&\qquad+2\E\left[U^*_x\bigg(\frac{Z^+_j(S_j^*, K)}{2\sigma^2\c_4|K|\log \J_+(H^*_x)}-Y_j\bigg)(U^*_x-U)Y_j\right],
\end{align*}
where 
\[
\E\big[(U^*_x-U)^2Y_j^2\big]=\E[Y_j^2]\E\big[(U^*_x-U)^2\big]=2\E[(U_x^*-U)^2]=o_x(1),
\]
by \eqref{2subtreeU} and \eqref{d2U}. Moreover, the Cauchy-Schwarz inequality implies that
\begin{align*}
&\E\left[U^*_x\bigg(\frac{Z^+_j(S_j^*, K)}{2\sigma^2\c_4|K|\log \J_+(H^*_x)}-Y_j\bigg)(U^*_x-U)Y_j\right]\\
&\le\E\Bigg[(U^*_x)^2\bigg(\frac{Z^+_j(S_j^*, K)}{2\sigma^2\c_4|K|\log \J_+(H^*_x)}-Y_j\bigg)^2\Bigg]^{\frac12}\E\big[(U^*_x-U)^2Y_j^2\big]^{\frac12}=o_x(1),
\end{align*}
because
\begin{align}
&\sup_{x\in K^c}\E\Bigg[(U^*_x)^2\bigg(\frac{Z^+_j(S_j^*, K)}{2\sigma^2\c_4|K|\log \J_+(H^*_x)}-Y_j\bigg)^2\Bigg]\nonumber\\
&\le 2\sup_{x\in K^c} \E\Bigg[(U^*_x)^2\bigg(\bigg(\frac{Z^+_j(S_j^*, K)}{2\sigma^2\c_4|K|\log \J_+(H^*_x)}\bigg)^2+Y_j^2\bigg)\Bigg]\nonumber\\
&\le 2\sup_{x\in K^c}\E\big[(U^*_x)^2\big]\sup_{z\in\Z^4, e\in\supp(\mu)}\E\left[\left(\frac{Z^+(z+e,K)}{2\sigma^2\c_4|K|\log\J_+(z)}\right)^2+2\right]<\infty, \label{supZ+Y}
\end{align}
by \eqref{supUx} and \eqref{supZx}. Then \eqref{goodupbdbx} becomes 
\begin{align*}
b_x\le & \sum_{j=1}^2\E\Bigg[(U^*_x)^2 \bigg(\frac{Z^+_j(S_j^*, K)}{2\sigma^2\c_4|K|\log \J_+(H^*_x)}-Y_j\bigg)^2\Bigg]+o_x(1)\\
=&\sum_{j=1}^2\E\Bigg[(U^*_x)^2\bigg((1+\eta_j)\bigg(\frac{Z^+_j(S_j^*, K)}{2\sigma^2\c_4|K|\log \J_+(S^*_j)}-Y_j\bigg)+\eta_jY_j\bigg)^2\Bigg]+o_x(1),
\end{align*}
where $\eta_j=\frac{\log\J_+(S^*_j)}{\log \J_+(H^*_x)}-1$ for $j=1,2$. It is obvious that $\eta_j $ converges in probability to zero as $\|x\|\to\infty$. Moreover, we observe that
\begin{equation}\label{eq:eta-bar-defi}
  |\eta_j|\le \sup_{z\in\Z^4, e\in\supp(\mu)}\frac{\log\J_+(z+e)}{\log \J_+(z)}+1=:\overline{\eta}<\infty.
\end{equation}
Therefore, one can see from \eqref{supUx} that  $U_x^*\eta_j$ is $L^2$-bounded and $U_x^*\eta_j$ converges in probability to zero. By independence, one deduces that
\begin{equation}\label{smallsquareterm}
\E\big[(U_x^*)^2(\eta_j Y_j)^2\big]= \E\big[(U_x^* \eta_j )^2\big]\E\big[Y_j^2\big] \to 0
\end{equation}
as $\|x\|\to\infty$. Next, we see that by the Cauchy-Schwarz inequality,
\begin{align*}
&\E\left[(U^*_x)^2(1+\eta_j)\bigg(\frac{Z^+_j(S_j^*, K)}{2\sigma^2\c_4|K|\log \J_+(S^*_j)}-Y_j\bigg)\eta_j Y_j\right]\\
\le \; & (1+\overline{\eta}) \E\left[U^*_x\bigg(\frac{Z^+_j(S_j^*, K)}{2\sigma^2\c_4|K|\log \J_+(S^*_j)}-Y_j\bigg)\times U^*_x\eta_j Y_j\right]\\ 
\le \; & (1+\overline{\eta}) \E\Bigg[(U^*_x)^2\bigg(\frac{Z^+_j(S_j^*, K)}{2\sigma^2\c_4|K|\log \J_+(S^*_j)}-Y_j\bigg)^2\Bigg]^{\frac12}\E\left[(U^*_x\eta_j Y_j)^2\right]^{\frac12}.
\end{align*}
In a similar manner as in \eqref{supZ+Y}, we can verify that 
\[
  \E\Bigg[(U^*_x)^2\bigg(\frac{Z^+_j(S_j^*, K)}{2\sigma^2\c_4|K|\log \J_+(S^*_j)}-Y_j\bigg)^2\Bigg] 
\]
is uniformly bounded. Combined with \eqref{smallsquareterm}, it follows that
\[
  \E\left[(U^*_x)^2(1+\eta_j)\bigg(\frac{Z^+_j(S_j^*, K)}{2\sigma^2\c_4|K|\log \J_+(S^*_j)}-Y_j\bigg)\eta_j Y_j\right] =o_x(1).
\]
We hence establish that
\begin{align*}
b_x\le &\sum_{j=1}^2\E\left[(U^*_x)^2(1+\eta_j)^2 \bigg(\frac{Z^+_j(S_j^*, K)}{2\sigma^2\c_4|K|\log \J_+(S^*_j)}-Y_j\bigg)^2 \right] +o_x(1)\\
=&\sum_{j=1}^2\E\left[(U^*_x)^2(1+\eta_j)^2\E\bigg[\bigg(\frac{Z^+_j(S_j^*, K)}{2\sigma^2\c_4|K|\log \J_+(S^*_j)}-Y_j\bigg)^2\bigg\vert U^*_x, S_1^*, S_2^*\bigg]\right]+o_x(1)\\
=&\sum_{j=1}^2\E\left[(U^*_x)^2(1+\eta_j)^2 b_{S^*_j}\right]+o_x(1).
\end{align*}
Thanks to \eqref{supbx}, we set $b:=\limsup_{\|x\|\to\infty}b_x\in[0,\infty)$. So, for any $\epsilon>0$, there exists $M_\epsilon>1$ such that
\[
b_x\le b+\epsilon, \quad \forall \|x\|\ge M_\epsilon.
\]
As a result, recalling that $U^*_x(1+\eta_j)=\log\J_+(S^*_j)/\log\J_+(x)$, we get
\begin{align*}
b_x\le & \sum_{j=1}^2\E\left[(U^*_x)^2(1+\eta_j)^2(b+\epsilon)\ind{\|S^*_j\|\ge M_\epsilon}\right]+\sum_{j=1}^2 \E\bigg[\bigg(\frac{\log\J_+(S^*_j)}{\log\J_+(x)}\bigg)^2\overline{b}\ind{\|S^*_j\|<M_\epsilon}\bigg]+o_x(1)\\
\le &\,(b+\epsilon) \sum_{j=1}^2\E\left[(U^*_x)^2(1+\eta_j)^2\right]+2\overline{b}\,\frac{\sup_{\|z\|<M_\epsilon}(\log\J_+(z))^2}{(\log\J_+(x))^2}+o_x(1).
\end{align*}
Taking $\limsup_{\|x\|\to\infty}$ on both sides yields that
\[
b\le (b+\epsilon)\sum_{j=1}^2\limsup_{\|x\|\to\infty}\E\left[(U^*_x)^2(1+\eta_j)^2\right].
\]
It is known from Lemma \ref{Mallows} and \eqref{2subtreeU} that 
\[
\lim_{\|x\|\to\infty}\E[(U^*_x)^2]=\E[U^2]=\frac13.
\]
As $\eta_j\to0$ in probability and $\eta_j$ is bounded, we have for all $j=1,2$,
\[
\lim_{\|x\|\to\infty}\E[(U^*_x)^2(1+\eta_j)^2]=\frac13.
\]
This leads to 
\[
b\le \frac23 (b+\epsilon), \quad \forall \epsilon>0.
\]
As $b\in[0,\infty)$, we must have $b=0$. In other words, $\limsup_{\|x\|\to\infty}b_x=0$. The proof of \eqref{keycvgd2} is therefore complete.
\end{proof}

It remains to verify \eqref{supUx}, \eqref{2subtreeZ} and \eqref{2subtreeU}. Note that the proofs of \eqref{2subtreeZ} and \eqref{2subtreeU} are similar. 

\begin{proof}[Proof of \eqref{supUx}] 
Recall the definition \eqref{eq:defi-U_x^+} of $U^+_x$. Proposition \ref{MRCA} shows that 
\[
\sup_{x\in\Z^4}\E[(U^+_x)^2]<\infty.
\]
We will bound $\E[(U^*_x)^2]$ by comparing $U^*_x$ with $U^+_x$. Observe that
\begin{align*}
\E[(U^*_x)^2]=&\,\frac{\E_x[(\frac{\log\J_+(H_x)}{\log\J_+(x)})^2\ind{Z_\T(K)\ge1, N_x(K)=2}]}{\P_x(Z_\T(K)\ge 1, N_x(K)=2)}\\
\le &\,\frac{\E_x[(\frac{\log\J_+(H_x)}{\log\J_+(x)})^2\ind{Z_\T(K)\ge1}]}{\P_x(N_x(K)=2\vert Z_\T(K)\ge 1)\P_x(Z_\T(K)\ge 1)}=\frac{\E[(U^+_x)^2]}{\P_x(N_x(K)=2\vert Z_\T(K)\ge 1)}
\end{align*}
where $\P_x(N_x(K)=2\vert Z_\T(K)\ge 1)=1+o_x(1)$ by \eqref{cvgNx}. For any $x\notin K$, note that
\[
\P_x(N_x(K)=2\vert Z_\T(K)\ge 1)>0.
\]
Therefore, 
\[
\sup_{x\in K^c}\E[(U^*_x)^2]\le \frac{\sup_{x\in\Z^4}\E[(U^+_x)^2]}{\inf_{x\in K^c}\P_x(N_x(K)=2\vert Z_\T(K)\ge 1)}<\infty
\]
In the same way, we also get 
\[
\sup_{x\in K^c}\E\bigg[\left(\frac{Z^*(x,K)}{2\sigma^2\c_4|K|\log \J_+(x)}\right)^2\bigg]<\infty
\]
by use of \eqref{supZx}.
\end{proof}

\begin{proof}[Proof of \eqref{2subtreeU}] 
It follows from Proposition \ref{MRCA} that 
\[
d_2(\L(U^+_x), \L(U))\xrightarrow{\|x\|\to\infty}0.
\]
Since $d_2(\L(U^*_x),\L(U))\le d_2(\L(U^*_x),\L(U^+_x))+d_2(\L(U^+_x), \L(U))$, we only need to show that $d_2(\L(U^*_x),\L(U^+_x))=o_x(1)$. Recall that $U^*_x$ is distributed as $U^+_x$ conditioned on $\{N^+_x(K)=2\}$. Then for any fixed $x\in\Z^4$, we can find some version of independent $U^*_x, U^>_x, N^+_x(K)$ such that
\begin{equation}\label{condU}
U^+_x\overset{\mathrm{(d)}}{=} U^*_x\ind{N^+_x(K)=2}+U^>_x\ind{N^+_x(K)\neq 2},
\end{equation}
where $U^>_x$ is distributed as $U^+_x$ conditioned on $N^+_x(K)\neq 2$. Then, 
\begin{align*}
d_2(\L(U^*_x),\L(U^+_x))^2\le &\,\E\big[( U^*_x\ind{N^+_x(K)=2}+U^>_x\ind{N^+_x(K)\neq 2}-U^*_x)^2\big]\\
=&\,\E\left[(U^>_x-U^*_x)^2\ind{N^+_x(K)\neq 2}\right]\\
\le & \,2\E\left[\big((U^>_x)^2+(U^*_x)^2\big)\ind{N^+_x(K)\neq 2}\right].
\end{align*}
By independence, we see that 
\begin{align*}
d_2(\L(U^*_x),\L(U^+_x))^2\le & \,2 \left(\E \big[(U_x^>)^2\big]+\E\big[(U^*_x)^2\big]\right)\P(N_x^+(K)\neq 2)\\
=& \,2\E\big[(U^>_x)^2\big]\P(N^+_x(K)\neq 2)+2 \E\big[(U_x^*)^2\big]\P(N_x^+(K)\neq 2) .
\end{align*}
\eqref{cvgNx} implies that $\P(N_x^+(K)\neq 2)=o_x(1)$. Together with \eqref{supUx}, we see that 
\[
\E[(U_x^*)^2]\P(N_x^+(K)\neq 2)=o_x(1).
\]
It remains to show that $\E[(U^>_x)^2]\P(N^+_x(K)\neq 2)=o_x(1)$. In fact,
\begin{align*}
\E[(U^>_x)^2]\P(N^+_x(K)\neq 2)=&\,\frac{\E\Big[(U^+_x)^2\ind{N^+_x(K)\neq 2}\Big]}{\P(N^+_x(K)\neq 2)}\P(N^+_x(K)\neq 2)\\
=&\,\E\Big[(U^+_x)^2\ind{N^+_x(K)\neq 2}\Big]\\
\le & \,\E\Big[(U^+_x)^2\ind{U^+_x\ge M}\Big]+M^2\P(N^+_x(K)\neq 2)
\end{align*}
for any $M\ge 1$. Using the uniform integrability of $U_x^+$ obtained in Proposition \ref{MRCA}, we conclude that $\E[(U^>_x)^2]\P(N^+_x(K)\neq 2)=o_x(1)$.
\end{proof}

\begin{proof}[Proof of \eqref{2subtreeZ}] 
First, let us proceed in the same way as in the proof of \eqref{2subtreeU}. For any fixed $x\in\Z^4$, we take some version of independent $Z^*(x,K), Z^>(x,K), N^+_x(K)$ such that
\[
Z^+(x,K)\overset{\mathrm{(d)}}{=} Z^*(x,K)\ind{N^+_x(K)=2}+Z^>(x,K)\ind{N^+_x(K)\neq2},
\]
where $Z^>(x,K)$ is distributed as $Z^+(x,K)$ conditioned on $N^+_x(K)\neq 2$. Then it follows that
\begin{align*}
&d_2\left(\L\left(\frac{Z^+(x,K)}{2\sigma^2\c_4|K|\log \J_+(x)}\right),\L\left(\frac{Z^*(x,K)}{2\sigma^2\c_4|K|\log \J_+(x)}\right)\right)^2\\
&\le 2\E\left[\bigg(\frac{Z^*(x,K)}{2\sigma^2\c_4|K|\log\J_+(x)}\bigg)^2+\bigg(\frac{Z^>(x,K)}{2\sigma^2\c_4|K|\log\J_+(x)}\bigg)^2\right]\P(N^+_x(K)\neq2).
\end{align*}
By \eqref{supUx} and \eqref{cvgNx}, 
\[
\E\left[\bigg(\frac{Z^*(x,K)}{2\sigma^2\c_4|K|\log\J_+(x)}\bigg)^2\right]\P(N^+_x(K)\neq2)=o_x(1).
\]
It remains to check that
\[
\E\left[\bigg(\frac{Z^>(x,K)}{2\sigma^2\c_4|K|\log\J_+(x)}\bigg)^2\right]\P(N^+_x(K)\neq2)=o_x(1).
\]
From now on, we will use different arguments compared to the proof of \eqref{2subtreeU}. Observe that
\begin{align*}
&\E\left[\bigg(\frac{Z^>(x,K)}{2\sigma^2\c_4|K|\log\J_+(x)}\bigg)^2\right]\P(N^+_x(K)\neq2)=\E\left[\bigg(\frac{Z^+(x,K)}{2\sigma^2\c_4|K|\log\J_+(x)}\bigg)^2\ind{N^+_x(K)\neq2}\right]\\
&=\E\left[\bigg(\frac{Z^+(x,K)}{2\sigma^2\c_4|K|\log\J_+(x)}\bigg)^2\right]-\E\left[\bigg(\frac{Z^+(x,K)}{2\sigma^2\c_4|K|\log\J_+(x)}\bigg)^2\ind{N^+_x(K)=2}\right]\\
&=\E\left[\bigg(\frac{Z^+(x,K)}{2\sigma^2\c_4|K|\log\J_+(x)}\bigg)^2\right]-\E\left[\bigg(\frac{Z^*(x,K)}{2\sigma^2\c_4|K|\log\J_+(x)}\bigg)^2\right]\P(N^+_x(K)=2).
\end{align*}
By \eqref{2momZ} in Lemma \ref{2mom}, one sees that
\[
\mathbf{m}_2(x,K):=\E\left[\bigg(\frac{Z^+(x,K)}{2\sigma^2\c_4|K|\log\J_+(x)}\bigg)^2\right]=2+o_x(1).
\]
To conclude, we only need to verify that
\[
\E\left[\bigg(\frac{Z^*(x,K)}{2\sigma^2\c_4|K|\log\J_+(x)}\bigg)^2\right]=2+o_x(1).
\]
Applying \eqref{2subtree}, we observe that
\begin{align*}
&\E\left[\bigg(\frac{Z^*(x,K)}{2\sigma^2\c_4|K|\log\J_+(x)}\bigg)^2\right]=\E\left[(U^*_x)^2\left(\frac{Z^+_1(S_1^*, K)}{2\sigma^2\c_4|K|\log \J_+(H^*_x)}+\frac{Z^+_2(S_2^*,K)}{2\sigma^2\c_4|K|\log \J_+(H^*_x)}\right)^2\right]\\
&=\E\Bigg[(U^*_x)^2\bigg(\sum_{j=1}^2(1+\eta_j)^2\mathbf{m}_2(S_j^*,K)\bigg)+2(U^*_x)^2(1+\eta_1)(1+\eta_2)\mathbf{m}_1(S^*_1,K)\mathbf{m}_1(S_2^*,K)\Bigg],
\end{align*}
where, for $j=1,2$,
\[
\eta_j=\frac{\log\J_+(S^*_j)}{\log \J_+(H^*_x)}-1, 
\]
and
\[
\mathbf{m}_1(x,K):=\E\left[\frac{Z^+(x,K)}{2\sigma^2\c_4|K|\log\J_+(x)}\right]=1+o_x(1).
\]
Recall the definition of $\overline{\mathbf{m}}_2$ in \eqref{supZx}, and that of $\overline{\eta}$ in \eqref{eq:eta-bar-defi}. We also set
\[
\overline{\mathbf{m}}_1:=\sup_{x\in\Z^4}\E\left[\frac{Z^+(x,K)}{2\sigma^2\c_4|K|\log\J_+(x)}\right]<\infty.
\]
Note that on the one hand, when $\|x\|\to\infty$, $U^*_x\xrightarrow[]{\mathrm{(d)}} U$ and $1+\eta_j$ converges in probability to 1, accordingly
\[
(U^*_x)^2\sum_{j=1}^2(1+\eta_j)^2\mathbf{m}_2(S_j^*,K)+2(U^*_x)^2(1+\eta_1)(1+\eta_2)\mathbf{m}_1(S^*_1,K)\mathbf{m}_1(S_2^*,K)\xrightarrow[\|x\|\to\infty]{\mathrm{(d)}} 6U^2.
\]
On the other hand, the left-hand side
\[
(U^*_x)^2\sum_{j=1}^2(1+\eta_j)^2\mathbf{m}_2(S_j^*,K)+2(U^*_x)^2(1+\eta_1)(1+\eta_2)\mathbf{m}_1(S^*_1,K)\mathbf{m}_1(S_2^*,K)
\]
is bounded by $2(1+\overline{\eta})^2 (\overline{\mathbf{m}}_2+\overline{\mathbf{m}}_1^2)(U^*_x)^2$. Since $(U^+_x)^2$ is uniformly integrable according to Proposition~\ref{MRCA}, the uniform integrability of $(U^*_x)^2$ follows by the same argument used in the proof of \eqref{2subtreeU}. 
Consequently, 
\[
\E\left[\bigg(\frac{Z^*(x,K)}{2\sigma^2\c_4|K|\log\J_+(x)}\bigg)^2\right]\xrightarrow[\|x\|\to\infty]{} \E[6U^2]=2.
\]
This suffices to conclude the proof.
\end{proof}

\section{Occupation times when $1\leq d\leq 3$: proof of Theorem \ref{lowd}}\label{ld}

Let us consider the low dimensions $d\leq 3$ in this section. Here we always assume \eqref{offspring}, \eqref{eq:assumption-moment4} and $d\leq 3$. 
Recall that we can easily adapt the arguments used in \cite[Theorem 7]{LeGall-Lin} to show that for any nonempty compact set $K\subset \Z^d$, as $\|x\|\to \infty$, 
\begin{equation}\label{uK3d}
\P_x(Z_\T(K)\ge 1)=(1+o_x(1)) \frac{2(4-d)}{d\sigma^2\J(x)^2}.
\end{equation}
We will need the following result on the tail of the total progeny of the critical Bienaym\'e--Galton--Watson tree $\T$, which is borrowed from (3) in \cite{LeGall-Lin}: as $n\to \infty$,
\begin{equation}\label{tailT}
  \P(\#\T\ge n)=\frac{2+o_n(1)}{\sqrt{2\pi\sigma^2}n^{1/2}}.
\end{equation}

Throughout this section, we consider only integers $n\ge1$ such that $\P(\#\T=n)>0$ and when we let $n\to\infty$, we mean along such values of $n$.

Let $f \colon \R_+\to\R_+$ be a continuous and compactly supported function with $f(0)=0$. In particular, its uniform norm $\|f\|_\infty$ is a finite constant. For sufficiently small $\delta\in(0,1)$, we observe that
\begin{align}\label{givenT}
&\E_x\left[f\bigg(\frac{Z_\T(K)}{\J(x)^{4-d}}\bigg)\ind{Z_\T(K)\ge1}\right]=\sum_{n=1}^\infty \E_0\left[f\bigg(\frac{Z_\T(K-x)}{\J(x)^{4-d}}\bigg)\ind{Z_\T(K-x)\ge 1}\ind{\#\T=n}\right]\nonumber\\
=& \,\E_\eqref{badn}+\sum_{n=\lfloor\delta\J(x)^4\rfloor}^{\lfloor\J(x)^4/\delta\rfloor}\E_0\left[f\bigg(\frac{Z_\T(K-x)}{\J(x)^{4-d}}\bigg)\ind{Z_\T(K-x)\ge 1}\ind{\#\T=n}\right],
\end{align}
where $Z_\T(K-x):=\sum_{y\in K}Z_\T(y-x)$ and
\begin{equation}\label{badn}
\E_\eqref{badn}:=\E_0\left[f\bigg(\frac{Z_\T(K-x)}{\J(x)^{4-d}}\bigg) \ind{Z_\T(K-x)\ge 1}\ind{\#\T\notin [ \lfloor\delta \J(x)^4 \rfloor, \lfloor\J(x)^4/\delta\rfloor]}\right].
\end{equation}
We first prove that as $\delta\to 0$,
\begin{equation}\label{bdbadn}
\E_\eqref{badn}=\frac{o_\delta(1)+o_x(1)}{\J(x)^2}=(o_\delta(1)+o_x(1))\P_x(Z_\T(K)\ge 1).
\end{equation}
In fact, by \eqref{tailT},
\begin{align*}
\E_\eqref{badn}\le & \,\|f\|_\infty \P_0\left(Z_\T(K-x)\ge1, \#\T\le \delta \J(x)^4\right)+\|f\|_\infty \P(\#\T\ge \J(x)^4/\delta)\\
\le &\,\|f\|_\infty \P_x\left(Z_\T(K)\ge1, \#\T\le \delta \J(x)^4\right)+\frac{o_\delta(1)}{\J(x)^2}.
\end{align*}
We define $\B_r(x):=\{z\in \R^d\colon \J(z-x)\leq r\}$.
In the case of small progeny, let us consider the stopping line
\[
\L_{\B_r(x)^c}:=\{u\in\T: \J(S_u-x)>r, \max_{v<u}\J(S_v-x)\le r\},
\]
with $r=\J(x)/2$. Observe that under $\P_x$, when $Z_\T(K)\ge1$, $\L_{\B_r(x)^c}$ is not empty. We write as usual $L_{\B_r(x)^c}=\#\L_{\B_r(x)^c}$. For $\varepsilon=\delta^{1/4}$ and $\J(x)\gg1$, it follows from the branching property at $\L_{\B_r(x)^c}$ that
\begin{align}\label{smallT}
&\P_x\left(Z_\T(K)\ge1, \#\T\le \delta \J(x)^4\right)\nonumber\\
\le &\, \P_x\left(L_{\B_r(x)^c}\ge \varepsilon\J(x)^2, \max_{u\in\L_{\B_r(x)^c}}\#\T_u\le \delta\J(x)^4\right)\nonumber\\
&+\P_x\Bigg( L_{\B_r(x)^c}\le \varepsilon\J(x)^2; \sum_{u\in\L_{\B_r(x)^c}}Z_{\T_u}(K)\ge 1\Bigg)\nonumber\\
\le & \, \P_x\left(L_{\B_r(x)^c}\ge \varepsilon\J(x)^2\right)\P(\#\T\le \delta\J(x)^4)^{\varepsilon\J(x)^2}\nonumber\\
&+\E_x\Bigg[\sum_{u\in\L_{\B_r(x)^c}}\P_{S_u}(Z_\T(K)\ge 1)\ind{1\le L_{\B_r(x)^c}\le \varepsilon\J(x)^2}\Bigg].
\end{align}
By abuse of notation, let hereafter $c>0$ stand for some constant independent of $\delta$ and $x$, whose value may change from line to line.
For the first term, by Markov's inequality and \eqref{tailT}, as $\varepsilon=\delta^{1/4}$, one has
\begin{align*}
&\,\P_x\left(L_{\B_r(x)^c}\ge \varepsilon\J(x)^2\right)\P\left(\#\T\le \delta\J(x)^4 \right)^{\varepsilon\J(x)^2} \\ 
\le &\, \E_x\left[\frac{L_{\B_r(x)^c}}{\varepsilon \J(x)^2}\right] \left( 1 - \P\left(\#\T> \delta\J(x)^4 \right)\right)^{\varepsilon \J(x)^2}\\
\le & \,\frac{1}{\delta^{1/4}\J(x)^2}\E_x[L_{\B_r(x)^c}] \cdot e^{-\frac{c}{\delta^{1/4}}}=\frac{o_\delta(1)}{\J(x)^2},
\end{align*}
because $\E_x[L_{\B_r(x)^c}]=\P_x(\exists n\ge 0 \colon S_n\in \B_r(x)^c)\le 1$. In the second term
\[
  \E_x\Bigg[\sum_{u\in\L_{\B_r(x)^c}}\P_{S_u}(Z_\T(K)\ge 1)\ind{1\le L_{\B_r(x)^c}\le \varepsilon\J(x)^2}\Bigg],
\]
we will separate the two cases $\J(S_u)\ge r/2$ or $\J(S_u)<r/2$ for $u\in \L_{\B_r(x)^c}$. 
On the one hand, by \eqref{uK3d}, one sees that
\begin{align}
&\E_x\left[\sum_{u\in\L_{\B_r(x)^c}}\P_{S_u}(Z_\T(K)\ge 1)\ind{ \J(S_u)\ge r/2}\ind{1\le L_{\B_r(x)^c}\le \varepsilon\J(x)^2}\right]\nonumber\\
\le & \sup_{y: \J(y)\ge r/2} \P_y(Z_\T(K)\ge 1)\E_x\left[L_{\B_r(x)^c}\ind{1\le L_{\B_r(x)^c}\le \varepsilon\J(x)^2}\right]\nonumber\\
\le & \frac{c}{\J(x)^2} \varepsilon \J(x)^2\P_x(L_{\B_r(x)^c}\ge 1)=o_\delta(1)\P_x(L_{\B_r(x)^c}\ge 1).\label{smallL}
\end{align}
Since $r=\J(x)/2$, observe that 
\[
\P_x(L_{\B_r(x)^c}\ge 1)\le \P_0(Z_\T(\B_{r/2}^c)\ge 1).
\]
Let us consider $k=\lfloor r^2\rfloor$ i.i.d.~CBRWs started from the origin. We denote by $\mathcal{V}^{[k]}$ the sum of the $k$ corresponding total occupation measures. Then,
\[
\P\left(\mathcal{V}^{[k]}(\B^c_{r/2})\ge 1\right)=1-\left(1-\P_0\left(Z_\T(\B^c_{r/2})\ge 1\right)\right)^{\lfloor r^2\rfloor}.
\]
By Proposition 6 in \cite{LeGall-Lin}, one deduces that
\[
\limsup_{r\to\infty}\P\left(\mathcal{V}^{[k]}(\B^c_{r/2})\ge 1\right)\le \P_0\left(\sup_{0\le s\le \tau}\|\hat{W}_s\|\ge \frac{\sqrt{d\sigma}}{2\sqrt{2}}\right)\in (0,1),
\]
where $(W_s)_{s\ge0}$ is a standard Brownian snake in $\R^d$ with lifetime process $(\zeta_s)_{s\ge0}$, $\hat{W}_s=W_s(\zeta_s)$ and $\tau$ denotes the first hitting time of $2/\sigma$ by the local time at $0$ of the lifetime process. Consequently, for sufficiently large $r$,
\[
\P(Z_\T(\B^c_{r/2})\ge 1)\le \frac{c}{r^2}= \frac{4c}{\J(x)^2}.
\]
Plugging it in \eqref{smallL} yields that
\[
\E_x\left[\sum_{u\in\L_{\B_r(x)^c}}\P_{S_u}(Z_\T(K)\ge 1)\ind{ \J(S_u)\ge r/2}\ind{1\le L_{\B_r(x)^c}\le \varepsilon\J(x)^2}\right]=\frac{o_\delta(1)}{\J(x)^2}.
\]
On the other hand, we have the upper bound
\begin{align*}
&\,\E_x\left[\sum_{u\in\L_{\B_r(x)^c}}\P_{S_u}(Z_\T(K)\ge 1)\ind{\J(S_u)< r/2}\ind{1\le L_{\B_r(x)^c}\le \varepsilon\J(x)^2}\right]\\
\le & \, \E_x\left[\sum_{u\in\L_{\B_r(x)^c}}\ind{\J(S_u)< r/2}\Bigg]= \sum_{n\geq 0}\E_x\Bigg[\sum_{|u|=n} \ind{u\in \L_{\B_r(x)^c}, \J(S_u)< r/2}\right]\\
= & \,\sum_{n\geq 0}\P_x(T_{\B_r(x)^c}=n, \J(S_n)<r/2) \qquad \mbox{ by the many-to-one lemma} \\ 
= & \, \P_x(T_{\B_r(x)^c}<\infty, \J(S_{T_{\B_r(x)^c}})<r/2). 
\end{align*}
By translation, we have
\begin{align*}
&\P_x(T_{\B_r(x)^c}<\infty, \J(S_{T_{\B_r(x)^c}})<r/2) =\P_0(T_{\B_r^c}<\infty, \J(S_{T_{\B_r^c}}+x)< r/2 )\\
\le &\,\P_0(\tau_r<\infty, \J(S_{\tau_r}) > 3r/2) \le \E_0[\tau_r]\times\mu(\{z\in\Z^d: \J(z) > r/2\}),
\end{align*}
as $\J(x) = 2r$. Using \eqref{eq:assumption-moment4} and the fact that $\E_0[\tau_r]=O(r^2)$, we see that 
\[
  \P_x(T_{\B_r(x)^c}<\infty, \J(S_{T_{\B_r(x)^c}})<r/2)=o(r^{-2})=\frac{o_x(1)}{\J(x)^2}.
\]
Going back to \eqref{smallT} and putting the previous estimates together, one gets that
\[
\P_x\left(Z_\T(K)\ge1, \#\T\le \delta \J(x)^4\right)=\frac{o_\delta(1)+o_x(1)}{\J(x)^2}.
\]
Therefore, we obtain \eqref{bdbadn}. Now, \eqref{givenT} becomes that
\begin{align}
&\E_x\left[f\left(\frac{Z_\T(K)}{\J(x)^{4-d}}\right)\ind{Z_\T(K)\ge1}\right]=(o_\delta(1)+o_x(1))\P_x(Z_\T(K)\ge 1)\nonumber\\
&\qquad+\sum_{n=\lfloor\delta\J(x)^4\rfloor}^{\lfloor\J(x)^4/\delta\rfloor}\E_0\left[f\left(\sum_{y\in K}\frac{Z_\T(y-x)}{\J(x)^{4-d}}\right)\bigg\vert\#\T=n\right]\P(\#\T=n).\label{givenTmain}
\end{align}

Given the covariance matrix $\Gamma$ of $\mu$, let $\Gamma^{1/2}$ be the unique positive definite symmetric matrix such that $\Gamma=(\Gamma^{1/2})^2$. We will use the notation $\Gamma^{-1/2}:=(\Gamma^{1/2})^{-1}$. 
In the following, we are going to consider a sequence of $x$ such that $\frac{\Gamma^{-1/2}x}{\|\Gamma^{-1/2}x\|}$ converges to some $\vartheta\in\mathbb{S}^{d-1}$ as $\J(x)\to\infty$ for $x\in\Z^d$ and prove the convergence in law of $\frac{Z_\T(K)}{\J(x)^{4-d}}$ under $\P_x(\cdot\vert Z_\T(K) \ge 1)$, along such sequence.  

Following the notation in \cite{LeGall-Lin}, let $\{L_n(z)\}_{z\in \Z^d}$ be distributed as $\{Z_\T(z)\}_{z\in \Z^d}$ under $\P_0$ conditioned on $\#\T=n$. Let $\mathbf{N}_z, z\in\R^d$ denote the $\sigma-$finite excursion measures of the Brownian snake $(W_s)_{s\ge0}$ started from $z$. We define the following probability measure for every $z\in\R^d$ and every $r>0$,
\[
\mathbf{N}_z^{(r)}:=\mathbf{N}_z(\,\cdot\,\vert \sup\{s\ge 0\colon \zeta_s>0\}=r).
\]
Under $\mathbf{N}_z^{(r)}$, the lifetime process $(\zeta_s)_{0\leq s\leq r}$ is a Brownian excursion with duration~$r$.
Let $W^{(1)}=(W^{(1)}_s)_{0\le s\le 1}$ be a process distributed according to $\mathbf{N}_0^{(1)}$. Let $(\ell^z, z\in\R^d)$ stand for the collection of local times of $(W^{(1)}_s)_{0\le s\le 1}$, which can be defined as the continuous density of the occupation measure of $W^{(1)}$ (see Proposition~1 in \cite{LeGall-Lin}). Theorem 4 of \cite{LeGall-Lin} implies that for a sequence $(x_n)$ in $\Z^d$ such that $\sqrt{\frac{\sigma}{2}} n^{-1/4}\Gamma^{-1/2}x_n \to z\in\R^d\setminus\{0\}$ as $n\to\infty$, we have the joint convergence in distribution as follows:
\begin{equation}
\left(\frac{L_n(y-x_n)}{n^{1-d/4}}\right)_{y\in K} \xrightarrow[n\to\infty]{\mathrm{(d)}} \underbrace{(1,\cdots,1)}_{\in\R^{|K|}}\times \frac{1}{\sqrt{\det\Gamma}}(\sigma/2)^{d/2} \ell^{-z}.
\label{joint-cv-local-time}
\end{equation}
We can replace $\ell^{-z}$ by $\ell^z$ in the previous display, because they have the same law by symmetry of the Brownian snake.
Moreover, from the proof of Lemma 3 in \cite{LeGall-Lin}, we can deduce that for any fixed $\delta>0$,
\begin{equation}
\lim_{\eta\downarrow0}\;\limsup_{n\to \infty}\sup_{\J(a_n)\wedge \J(b_n) \ge \delta n^{1/4}, \J(a_n-b_n)\le \eta n^{1/4}}\E\left[\frac{(L_n(a_n)-L_n(b_n))^2}{n^{2-d/2}}\right]=0.
\label{L2-continuity-local-time}
\end{equation}
Take a sufficiently small $\eta>0$, we are going to divide the interval $[\delta\J(x)^4, \J(x)^4/\delta]$ into small disjoint intervals of length $\eta\J(x)^4$. To this end, we set $A:=\lfloor\frac{1/\delta-\delta}{\eta}\rfloor$ and for every $1\leq i\leq A$, we define the interval
\[
I_i:=\delta\J(x)^4+ [\eta\J(x)^4 (i-1), \eta\J(x)^4i[,
\]
and finally, $I_{A+1}:=[\delta\J(x)^4+ A\eta\J(x)^4, \J(x)^4/\delta]$. In other words, 
\[
   I_i = [a_{i-1} \J(x)^4, a_i \J(x)^4[
\] 
with $a_i=\delta+i\eta$ for any $1\le i\le A$.

Suppose that $\frac{\Gamma^{-1/2}x}{\|\Gamma^{-1/2}x\|}$ converges to some $\vartheta\in\mathbb{S}^{d-1}$ as $\|x\|\to\infty$ for $x\in\Z^d$. In the remaining part the proof, when we take $x \to \infty$ or $\|x\|\to\infty$, we always refer to this limiting procedure unless otherwise specified. 

Recall that $\J(x)=\|\Gamma^{-1/2}x\|/\sqrt{d}$. For any integer $n\in I_i$ with $\J(x)\gg1$, let $x_n$ denote the unique point in $\Z^d$ such that
\[
x_n\in \Gamma^{1/2}\vartheta (\frac{n}{a_i})^{1/4}\sqrt{d}+[0,1)^d.
\]

Recall that $f$ is a continuous function with compact support. 
First, using the uniform continuity of $f$, together with \eqref{L2-continuity-local-time}, we can replace $Z_\T(y-x)$ under the conditioning $\T=n$ by $L_n(y-x_n)$ for each $y\in K$. It means that
\begin{multline*}
\varepsilon_1(x,\eta):=\sup_{1\le i\le A+1}\sup_{n\in I_i}\Bigg|\E_0\bigg[f\bigg(\sum_{y\in K}\frac{Z_\T(y-x)}{\J(x)^{4-d}}\bigg)\Big\vert\#\T=n\bigg]- \E \bigg[ f\bigg(\sum_{y\in K} \frac{L_n(y -x_n)}{\J(x)^{4-d}} \bigg)\bigg] \Bigg|
\end{multline*}
tends to zero as $x\to\infty$ and then $\eta\downarrow0$.

Secondly, using again the uniform continuity of $f$, together with the tightness of $\frac{L_n(y-x_n)}{n^{1-d/4}}$ from \eqref{joint-cv-local-time}, we get that
\begin{multline*}
\varepsilon_2(x,\eta):=\sup_{1\le i\le A+1}\sup_{n\in I_i}\Bigg| \E \bigg[ f\bigg(\sum_{y\in K} \frac{L_n(y -x_n)}{\J(x)^{4-d}} \bigg)\bigg] - \E \bigg[ f\bigg(\sum_{y\in K} \frac{L_n(y -x_n)}{n^{1-d/4}}a_i^{1-d/4} \bigg)\bigg] \Bigg|
\end{multline*}
tends to zero as $x\to\infty$ and then $\eta\downarrow0$.

Finally, by the weak convergence \eqref{joint-cv-local-time}, we obtain that
\begin{multline*}
\varepsilon_3(x,\eta):= \\
\sup_{1\le i\le A+1}\sup_{n\in I_i}\Bigg| \E \bigg[ f\bigg(\sum_{y\in K} \frac{L_n(y -x_n)}{n^{1-d/4}}a_i^{1-d/4} \bigg)\bigg] - \mathbf{N}_0^{(1)}\left[f\left(|K|a_i^{1-d/4}\frac{(\sigma/2)^{d/2}}{\sqrt{\det\Gamma}}\ell^{\sqrt{\sigma d/2}\vartheta a_i^{-1/4}}\right)\right]\Bigg|
\end{multline*}
tends to zero as $x\to\infty$ and then $\eta\downarrow0$.

Going back to \eqref{givenTmain}, we deduce therefore that
\begin{multline}
\E_x\left[f\bigg(\frac{Z_\T(K)}{\J(x)^{4-d}}\bigg)\ind{Z_\T(K)\ge1}\right]=(o_\delta(1)+o_x(1))\P_x(Z_\T(K)\ge 1)\\
+\sum_{i=1}^{A+1}\sum_{n\in I_i}\mathbf{N}_0^{(1)}\left[f\left(|K|a_i^{1-d/4}\frac{(\sigma/2)^{d/2}}{\sqrt{\det\Gamma}}\ell^{\sqrt{\sigma d/2}\vartheta a_i^{-1/4}}\right)\right]\P(\#\T\in I_i)\\
+\varepsilon(x,\eta)\P(\delta\J(x)^4\le \#\T\le \J(x)^4/\delta), \nonumber
\end{multline}
where $|\varepsilon(x,\eta)|\le \varepsilon_1(x,\eta)+\varepsilon_2(x,\eta)+\varepsilon_3(x,\eta)$ tends to zero as $x\to\infty$ and then $\eta\downarrow0$.

Applying \eqref{uK3d} and \eqref{tailT} here, we obtain that
\begin{align*}
&\E_x\left[f\bigg(\frac{Z_\T(K)}{\J(x)^{4-d}}\bigg)\Big\vert Z_\T(K)\ge1\right]=o_\delta(1)+o_x(1)+\varepsilon(x,\eta)(1+o_x(1))(\frac{1}{\sqrt{\delta}}-\sqrt{\delta}) \\
&+\frac{1}{\P_x(Z_\T(K)\ge 1)}\sum_{i=1}^{A+1}\mathbf{N}_0^{(1)}\left[f\left(|K|a_i^{1-d/4}\frac{(\sigma/2)^{d/2}}{\sqrt{\det\Gamma}}\ell^{\sqrt{\sigma d/2}\vartheta a_i^{-1/4}}\right)\right]\P(\#\T\in I_i).
\end{align*}
Since $\P(\#\T\in I_i) = \P(\#\T\ge a_{i-1}\J(x)^4)- \P(\#\T \ge a_i \J(x)^4)$, we apply \eqref{uK3d} and \eqref{tailT} again to see that 
\begin{align*}
&\frac{1}{\P_x(Z_\T(K)\ge 1)}\sum_{i=1}^{A+1}\mathbf{N}_0^{(1)}\left[f\left(|K|a_i^{1-d/4}\frac{(\sigma/2)^{d/2}}{\sqrt{\det\Gamma}}\ell^{\sqrt{\sigma d/2}\vartheta a_i^{-1/4}}\right)\right]\P(\#\T\in I_i)\\
=&(1+o_x(1))\frac{\sigma d/2}{\sqrt{2\pi}(4-d)}\sum_{i=1}^{A+1}\mathbf{N}_0^{(1)}\left[f\left(|K|a_i^{1-d/4}\frac{(\sigma/2)^{d/2}}{\sqrt{\det\Gamma}}\ell^{\sqrt{\sigma d/2}\vartheta a_i^{-1/4}}\right)\right](\frac{1}{\sqrt{a_{i-1}}}-\frac{1}{\sqrt{a_i}}).
\end{align*}
According to Proposition 1 of \cite{LeGall-Lin}, the density $y\mapsto\ell^y$ is $\mathbf{N}_0^{(1)}$-a.s.~continuous. Then as the function $f$ is bounded, by dominated convergence,
\[
a\mapsto \mathbf{N}_0^{(1)}\left[f\left(|K|a^{1-d/4}\frac{(\sigma/2)^{d/2}}{\sqrt{\det\Gamma}}\ell^{\sqrt{\sigma d/2}\vartheta a^{-1/4}}\right)\right]
\]
is continuous for $a>0$. Letting $\|x\|\to\infty$ and then $\eta\downarrow 0$, and finally $\delta\downarrow0$ yields that
\begin{align*}
&\lim_{\|x\|\to \infty}\E_x\left[f\bigg(\frac{Z_\T(K)}{\J(x)^{4-d}}\bigg)\Big\vert Z_\T(K)\ge1\right]\\
&= \frac{\sigma d/2}{\sqrt{2\pi}(4-d)}\int_0^{\infty} \mathbf{N}_0^{(1)}\left[f\left(|K|a^{1-d/4}\frac{(\sigma/2)^{d/2}}{\sqrt{\det\Gamma}}\ell^{\sqrt{\sigma d/2}\vartheta a^{-1/4}}\right)\right] \frac{\d a}{a^{3/2}},
\end{align*}
where the limiting integration is finite. As $f(0)=0$ and $f$ is bounded, we claim that
\begin{align*}
&\frac{\sigma d/2}{\sqrt{2\pi}(4-d)}\int_0^{\infty} \mathbf{N}_0^{(1)}\left[f\left(|K|a^{1-d/4}\frac{(\sigma/2)^{d/2}}{\sqrt{\det\Gamma}}\ell^{\sqrt{\sigma d/2}\vartheta a^{-1/4}}\right)\right] \frac{\d a}{a^{3/2}} \\
=&\,\mathbf{N}_0\left[f\left(\frac{(\sigma/2)^{d/2}}{\sqrt{\det\Gamma}}(\frac{2}{\sigma d})^{\frac{d-4}{2}}\ell^{\vartheta }\right)\Big\vert \ell^{\vartheta}>0\right]<\infty.
\end{align*}
In fact, by scaling of the Brownian snake, $(r^{d/4-1}\ell^{zr^{1/4}})_{z\in\R^d}$ under $\mathbf{N}_0^{(r)}$ has the same distribution as $(\ell^z)_{z\in\R^d}$ under $\mathbf{N}_0^{(1)}$. So, by letting $r=a(\frac{2}{\sigma d})^2$, we have
\[
\mathbf{N}_0^{(1)}\left[f\left(|K|a^{1-d/4}\frac{(\sigma/2)^{d/2}}{\sqrt{\det\Gamma}}\ell^{\sqrt{\sigma d/2}\vartheta a^{-1/4}}\right)\right]=\mathbf{N}^{(a(\frac{2}{\sigma d})^2)}_0\left[f\left(|K| \frac{(\sigma/2)^{d/2}}{\sqrt{\det\Gamma}}(\frac{2}{\sigma d})^{\frac{d-4}{2}}\ell^{\vartheta }\right)\right].
\]
As a consequence,
\begin{align*}
&\lim_{\|x\|\to\infty}\E_x\left[f\bigg(\frac{Z_\T(K)}{|K|\J(x)^{4-d}}\bigg)\Big\vert Z_\T(K)\ge1\right]\\
=&\frac{\sigma d/2}{\sqrt{2\pi}(4-d)}\int_0^{\infty}\mathbf{N}^{(a(\frac{2}{\sigma d})^2)}_0\left[f\left(\frac{(\sigma/2)^{d/2}}{\sqrt{\det\Gamma}}(\frac{2}{\sigma d})^{\frac{d-4}{2}}\ell^{\vartheta }\right)\right]\frac{\d a}{a^{3/2}}\\
=&\frac{2}{4-d}\int_0^\infty \mathbf{N}_0^{(r)} \left[f\left(\frac{(\sigma/2)^{d/2}}{\sqrt{\det\Gamma}}(\frac{2}{\sigma d})^{\frac{d-4}{2}}\ell^{\vartheta }\right)\ind{\ell^{\vartheta}>0}\right]\frac{\d r}{2\sqrt{2\pi r^3}}.
\end{align*}
For the last equality, we have used the fact that $f(0)=0$.
Recall the decomposition 
\[
\mathbf{N}_0=\int_0^\infty \frac{\d r}{2\sqrt{2\pi r^3}}\mathbf{N}_0^{(r)},
\]
and that
\[
\mathbf{N}_x(\ell^y>0)=\frac{4-d}{2} \|x-y\|^{-2}, \forall x\neq y
\]
(see for instance the formula (5) and the proof of Proposition 2 in \cite{LeGall-Lin}).
Therefore, for any $\vartheta\in\mathbb{S}^{d-1}$,
\[
\lim_{\|x\|\to\infty}\E_x\left[f\bigg(\frac{Z_\T(K)}{|K|\J(x)^{4-d}}\bigg)\Big\vert Z_\T(K)\ge1\right]=\mathbf{N}_0\left[f\left(\frac{(\sigma/2)^{d/2}}{\sqrt{\det\Gamma}}(\frac{2}{\sigma d})^{\frac{d-4}{2}}\ell^{\vartheta }\right)\Big\vert \ell^{\vartheta}>0\right].
\]

We conclude in the end that if $\frac{\Gamma^{-1/2}x}{\|\Gamma^{-1/2}x\|}$ converges as $\|x\|\to\infty$, then under $\P_x(\cdot\vert Z_\T(K) \geq 1)$, the following convergence in law holds:
\[
\frac{Z_\T(K)}{|K|\J(x)^{4-d}} \xrightarrow[\|x\|\to\infty]{\mathrm{(d)}}  \frac{(\sigma/2)^{d/2}}{\sqrt{\det\Gamma}}\bigg(\frac{2}{\sigma d}\bigg)^{\frac{d-4}{2}}\ell_+,
\]
where $\ell_+$ is defined in \eqref{localtime+}. Note that $\ell_+$ does not depend on $\vartheta$. For a general sequence of $x$ such that $\|x\|\to\infty$, by the compactness of $\S^{d-1}$, we can find a sub-subsequence $(x_{n_k})$ in any subsequence $(x_n)$, such that $\frac{\Gamma^{-1/2}x_{n_k}}{\|\Gamma^{-1/2}x_{n_k}\|}$ converges and along this sub-subsequence, the above weak convergence holds. Theorem \ref{lowd} is thus proved. 
\medskip

\appendix

\section{Appendix: Proof of Lemmas in Section \ref{SRW}}
\label{appendixA}

\begin{proof}[Proof of Lemma \ref{fargreen}]
Observe that 
\begin{align*}
g(x,y)=&\sum_{n=0}^\infty \P_x(S_n=y)=\sum_{n=0}^\infty \P_x(S_n=y, T_K>n)+\sum_{n=0}^\infty \P_x(S_n=y, T_K\le n)\\
=&\sum_{n=0}^\infty \P_x(S_n=y, T_K>n)+\sum_{n=0}^\infty \sum_{k=0}^n\sum_{z\in K} \P_x(T_K=k, S_{T_K}=z, S_n=y).
\end{align*}
We set $T_z:=\inf\{n\ge0 \colon S_n=z\}$ for any $z\in K$. The strong Markov property implies that $g(z,y)=\P_z(T_y<\infty)g(y,y)=\P_z(T_y<\infty)g(0,0)$. 

Then, by the strong Markov property at the stopping time $T_K$, one has
\begin{align*}
g(x,y)-\sum_{n=0}^\infty \P_x(S_n=y, T_K>n)=&\sum_{z\in K}\sum_{k=0}^\infty \P_x(T_K=k, S_{T_K}=z)g(z,y)\\
=&\sum_{z\in K}\P_x(T_K<\infty, S_{T_K}=z)\P_z(T_y<\infty)g(0,0)\\
\leq &\; g(0,0)\cdot \P_x(T_K<\infty)\cdot\sum_{z\in K}\P_z(T_y<\infty).
\end{align*}
Recall that $\bar \mu$ denotes the probability distribution on $\Z^d$ such that $\bar \mu(\{x\})=\mu(\{-x\})$ for all $x\in \Z^d$. 
By considering the random walk defined under $\bar \P$ with jump law $\bar \mu$, we have $\P_z(T_y<\infty)=\bar \P_y(T_z<\infty)$.
Lemma \ref{fargreen} then follows from \eqref{hittingprobabS-2}. 
\end{proof}

\begin{proof}[Proof of Lemma \ref{bigjump}]
For $R>r>\J(x)$, we have
\begin{align*}
  &\P_x(\J(S_{\tau_r})>R)\\
= &\, \sum_{n=0}^\infty \P_x(\J(S_{\tau_r})>R, \tau_r=n+1)\\
\leq & \,\sum_{n=0}^\infty \P_x(\tau_r=n+1,\J(S_{n})\leq r, \J(S_{n+1}-S_n)>R-r)\\
\leq & \,\sum_{n=0}^\infty \P_x(\tau_r>n,\J(S_{n})\leq r, \J(S_{n+1}-S_n)>R-r).
\end{align*}
By the Markov property at time $n$, the probability appearing in the last sum is equal to 
\[
  \P_x(\tau_r>n,\J(S_{n})\leq r)\cdot \P(\J(X_1)>R-r).
\]
It follows that 
\begin{eqnarray*}
  \P_x(\J(S_{\tau_r})>R)& \leq & \Big(\sum_{n=0}^\infty \P_x(\tau_r>n,\J(S_{n})\leq r)\Big) \P(\J(X_1)>R-r)\\
&= & \Big(\sum_{n=0}^\infty \P_x(\tau_r>n)\Big)  \P(\J(X_1)>R-r)\\
&=& \E_x[\tau_r] \cdot\P(\J(X_1)>R-r).
\end{eqnarray*}

If the jump law $\mu$ has a finite $(d-1)$-th moment, 
\[
  \lim_{r\to\infty}r^{d-1}\P(\J(X_1)>r)=0.
\]
At the same time, it is well known that $\E_x[\tau_r]=O(r^2)$. We have therefore $\lim_{r\to\infty}r^{d-3}\P_x(\J(S_{\tau_r})>2r)=0$.
\end{proof}

\begin{proof}[Proof of Lemma \ref{bdhgg}]
Within this proof, the dimension $d$ is always greater or equal to 4.
By \eqref{green-not-finite-support} and the assumption on $h$, one sees that for any $w\in\Z^d$,
\begin{equation}\label{upbdhgg}
h(w)g(a,w)g(w,b)\le \frac{C(\log(2+\|w\|))^k}{(1+\|w\|^{d-2})(1+\|a-w\|^{d-2})(1+\|w-b\|^{d-2})}.
\end{equation}
Recall that $M:=\max\{\|a\|,\|b\|,\|a-b\|\}$.
Without loss of generality, we assume that $m:=\min\{\|a\|,\|b\|,\|a-b\|\}=\|a-b\|$. The other cases can be treated similarly.
Let 
\begin{align*}
K_a=&\{w\in\Z^d \colon \|w-a\|\le \frac34 m\},\\
K_b=&\{w\in\Z^d \colon \|w-b\|\le \frac34 m\},\\
K_c=&\{w\in\Z^d \colon \|w\|\le \frac14 M\},\\
K_t=&\{w\in\Z^d \colon \|w-\frac{a+b}{2}\|\le 2M\}.
\end{align*}
We will bound the sums on $K_a,K_b,K_c, K_t\setminus(K_a\cup K_b\cup K_c)$ and $K_t^c$ separately. 

First, observe that when $w\in K_a$, 
\[
\|w\|\le \|a\|+\|w-a\|\le 2M,
\]
and
\[
(1+\|w\|^{d-2})(1+\|w-b\|^{d-2})\ge C(1+m^{d-2})(1+M^{d-2}).
\]
In fact, there are three cases as follows.
\begin{itemize}
\item If $M=\|a\|$, then $\|w\|\ge \|a\|-\|a-w\|\ge M/4$ and $\|w-b\|\ge \|a-b\|-\|a-w\|\ge m/4$.
\item If $M=\|b\|$ and $\|w\|\ge M/2$, then $\|w-b\|\ge \|a-b\|-\|a-w\|\ge m/4$.
\item If $M=\|b\|$ and $\|w\|\le M/2$, then $\|w-b\|\ge \|b\|-\|w\|\ge M/2$ and $\|w\|\ge \|a\|-\|a-w\|\ge m/4$.
\end{itemize}
Then it follows from \eqref{upbdhgg} that
\begin{align*}
\sum_{w\in K_a}h(w)g(a,w)g(w,b)\le &\,\frac{C(\log(2+M))^k}{(1+m^{d-2})(1+M^{d-2})}\sum_{w\in K_a}\frac{1}{1+\|a-w\|^{d-2}}\\
\le& \, \frac{C(\log(2+M))^k}{(1+m^{d-2})(1+M^{d-2})}m^2\\
\le& \, \frac{C(\log(2+M))^k(1+m)^{4-d}}{1+M^{d-2}}.
\end{align*}
Secondly, when $w\in K_b$, similarly we have
\begin{align*}
\sum_{w\in K_b}h(w)g(a,w)g(w,b)\le &\,\frac{C(\log(2+M))^k}{(1+m^{d-2})(1+M^{d-2})}\sum_{w\in K_b}\frac{1}{1+\|w-b\|^{d-2}}\\
\le&  \,\frac{C(\log(2+M))^k}{(1+m^{d-2})(1+M^{d-2})}m^2\\
\le & \, \frac{C(\log(2+M))^k(1+m)^{4-d}}{1+M^{d-2}}.
\end{align*}
Thirdly, when $w\in K_c$, we have $\|w\|\le \frac{M}{4}$ and
\[
(1+\|a-w\|^{d-2})(1+\|w-b\|^{d-2})\le C(1+M^{d-2})^2.
\]
In fact, there are two following cases.
\begin{itemize}
\item If $M=\|a\|$, then $\|b\| = \frac{\|b\|}{2}+\frac{\|b\|}{2}\ge \frac{\|a\|-\|a-b\|}{2}+\frac{\|b\|}{2}\ge \frac{M-m}{2}+\frac{m}{2}=\frac{M}{2}$. Consequently, 
\[
\|a-w\|\ge \|a\|-\|w\|\ge \frac34 M, \quad \|w-b\|\ge \|b\|-\|w\|\ge \frac14 M.
\]
\item If $M=\|b\|$, things are similar. 
\end{itemize}
Again by \eqref{upbdhgg}, we get that
\begin{align*}
\sum_{w\in K_c}h(w)g(a,w)g(w,b)\le &\,\frac{C(\log(2+M))^k}{(1+M^{d-2})^{2}}\sum_{w\in K_c}\frac{1}{1+\|w\|^{d-2}}\\
\le& \, \frac{C(\log(2+M))^k}{(1+M^{d-2})^2}M^2\\
\le & \, \frac{C(\log(2+M))^k(1+M)^{4-d}}{1+M^{d-2}}.
\end{align*}
Next, when $w\in K_t\setminus(K_a\cup K_b\cup K_c)$, we have $\frac14 M< \|w\|\le 3M$ and
\[
(1+\|a-w\|^{d-2})(1+\|w-b\|^{d-2})\ge C \Big(1+\|w-\frac{a+b}{2}\|^{d-2}\Big)^2,
\]
and $ \|w-\frac{a+b}{2}\|\ge \frac14 m$. In fact, observe that 
\begin{itemize}
\item if $\|w-\frac{a+b}{2}\|>m$, $\|a-w\|\ge \|w-\frac{a+b}{2}\|-\|\frac{a-b}{2}\|\ge \frac12  \|w-\frac{a+b}{2}\|$;
\item if $ \|w-\frac{a+b}{2}\|\le m$, $\|a-w\|\ge \frac34 m \ge \frac34  \|w-\frac{a+b}{2}\|$ as $w\notin K_a$.
\end{itemize}
The same argument works for $\|w-b\|$. Moreover, as $w\notin K_a$ and $w\notin K_b$,
\[
2\|w-\frac{a+b}{2}\| \ge \|w-a\|-\|\frac{a-b}{2}\|+\|w-b\|-\|\frac{a-b}{2}\|\ge \frac32 m-\|a-b\|=\frac12 m.
\]
Therefore, as $w\notin K_c$,
\begin{align*}
&\sum_{w\in K_t\setminus(K_a\cup K_b\cup K_c)}h(w)g(a,w)g(w,b)\\
\le\; &\frac{C(\log(2+M))^k}{1+M^{d-2}} \sum_{w\in K_t\setminus(K_a\cup K_b\cup K_c)}\frac{1}{(1+ \|w-\frac{a+b}{2}\|^{d-2})^2}\\
\le\; & \frac{C(\log(2+M))^k}{1+M^{d-2}} \sum_{\frac14 m\le \|w-\frac{a+b}{2}\|\le 2M}\frac{1}{(1+ \|w-\frac{a+b}{2}\|^{d-2})^2}\\
\le\; &\begin{cases}
 \frac{C(\log(2+M))^k}{1+M^{2}}[1+\log\frac{1+M}{1+m}] \textrm{ when } d=4,\\
 \frac{C(\log(2+M))^k(1+m)^{4-d}}{1+M^{d-2}} \textrm{ when } d\ge 5.
\end{cases}
\end{align*}
Finally, when $w\in K_t^c$, $\|w-\frac{a+b}{2}\|\ge 2M$ and
\[
h(w)g(a,w)g(w,b)\le C \frac{(\log (\|w-\frac{a+b}{2}\|+2))^k}{(1+\|w-\frac{a+b}{2}\|^{d-2})^3}.
\] 
In fact, 
\begin{align*}
\|w-a\|\ge &\, \|w-\frac{a+b}{2}\|-\|\frac{a-b}{2}\|=\|w-\frac{a+b}{2}\|-\frac{m}{2}\ge\frac32 \|w-\frac{a+b}{2}\|,\\
\|w-b\|\ge & \, \|w-\frac{a+b}{2}\|-\|\frac{a-b}{2}\|=\|w-\frac{a+b}{2}\|-\frac{m}{2}\ge\frac32 \|w-\frac{a+b}{2}\|,\\
\|w\| \ge & \,\|w-\frac{a+b}{2}\|-\|\frac{a+b}{2}\|\ge \frac12  \|w-\frac{a+b}{2}\|.
\end{align*}
As a consequence,
\begin{align*}
\sum_{w\in K_t^c}h(w)g(a,w)g(w,b)\le &\sum_{\|w-\frac{a+b}{2}\|\ge 2M}C \frac{(\log (\|w-\frac{a+b}{2}\|+2))^k}{(1+\|w-\frac{a+b}{2}\|^{d-2})^3}\\
\le &\,\frac{C}{1+M^{d-2}}\sum_{\|w-\frac{a+b}{2}\|\ge 2M} \frac{(\log (\|w-\frac{a+b}{2}\|+2))^k}{(1+\|w-\frac{a+b}{2}\|^{d-2})^2}\\
\le & \, C \frac{(\log(2+M))^k(1+M)^{4-d}}{1+M^{d-2}}.
\end{align*}
We conclude by combining the upper bounds of the sums on $K_a,K_b,K_c, K_t\setminus(K_a\cup K_b\cup K_c)$ and $K_t^c$.
\end{proof}

From now on, we take $d=4$ to establish Lemma \ref{sumh}. We first prove the following result which generalizes Lemmas 17, 18 of \cite{LeGall-Lin16} and Lemmas 3.3, 3.4 of \cite{Zhu}. Recall that $\tau_r=\inf\{n\geq 0 \colon S_n\notin \B_r\}$.

\begin{lemma}\label{sumhmomp}
Suppose that $h\colon \Z^4\to \R_+$ is a non-negative measurable function such that 
\[
h(x)\sim c_h\frac{(\log \J(x))^k }{J(x)^2},\textrm{  as } \|x\|\to\infty,
\]
for some integer $k\in\mathbb{N}$ and some constant $c_h>0$. For any $y\in K$ and $p\in\n^*$, there exists a positive constant $C_{h,p}$ such that for any $r\ge2$,
\begin{equation}\label{sumhmompbd}
\E_y\Bigg[\bigg(\sum_{j=0}^{\tau_r-1}h(S_j)\bigg)^p \Bigg]\le C_{h,p} (\log r)^{p(k+1)}.
\end{equation}
Moreover, for any $a>0$, there exists some constant $C_{h,p,a}>0$ such that for any $r\ge2$,
\begin{equation}\label{cvgPsumh}
\P_y\bigg(\bigg\vert \sum_{j=0}^{\tau_r-1}h(S_j)-\frac{4c_h}{k+1}(\log r)^{k+1}\bigg\vert \ge a (\log r)^{k+1}\bigg)\le C_{h,p,a}(\log r)^{-p}.
\end{equation}
\end{lemma}

\begin{proof}[Proof of Lemma \ref{sumhmomp}]
Let us first prove \eqref{sumhmompbd}. Thanks to the assumption on $h$, there exists a constant $C>0$ such that 
\[
	h(x)\leq C\frac{(\log (2+\J(x)))^k }{(1+J(x))^2} \quad \textrm{ for every } x\in \Z^4.
\]
Then it is immediate that
\[
\sum_{j=0}^{\tau_r-1}h(S_j)\le C \sum_{j=0}^{\tau_r-1}\frac{(\log(2+\J(S_j)))^k}{(1+\J(S_j))^2}\le C(\log(2+r))^k \sum_{j=0}^{\tau_r-1}\frac{1}{1+\J(S_j)^2}.
\]
Moreover, we know from Lemma 3.3 of \cite{Zhu} that for any $r\ge2$,
\[
\E_y\left[\bigg( \sum_{j=0}^{\tau_r-1}\frac{1}{1+\J(S_j)^2}\bigg)^p\right]\le C_{h,p}(\log r)^p.
\]
We thus conclude \eqref{sumhmompbd}.

Now let us turn to \eqref{cvgPsumh}. First, we take $r_0=(\log r)^\beta$ with some $\beta>0$ that will be fixed later, then 
\begin{align*}
&\P_y\bigg(\bigg\vert \sum_{j=0}^{\tau_r-1}h(S_j)-\frac{4c_h}{k+1}(\log r)^{k+1}\bigg\vert \ge a (\log r)^{k+1}\bigg) \\
\leq \; & \P_y\bigg(\sum_{j=0}^{\tau_{r_0}-1}h(S_j)\geq \frac{a}{2}(\log r)^{k+1}\bigg)+ \P_y\bigg(\bigg\vert \sum_{j=\tau_{r_0}}^{\tau_r-1}h(S_j)-\frac{4c_h}{k+1}(\log r)^{k+1}\bigg\vert \ge \frac{a}{2}(\log r)^{k+1}\bigg).
\end{align*}
For the first term on the right-hand side,
\[
	 \P_y\bigg(\sum_{j=0}^{\tau_{r_0}-1}h(S_j)\geq\frac{a}{2}(\log r)^{k+1}\bigg) \leq \Big(\frac{2}{a}\Big)^p(\log r)^{-p(k+1)}\E_y\bigg[\bigg(\sum_{j=0}^{\tau_{r_0}-1}h(S_j)\bigg)^p\bigg],
\]
which, by \eqref{sumhmompbd}, is bounded by $(\frac{\log r_0}{\log r})^{p(k+1)}$ up to some multiplicative constant $C_{h,p,a}$. To bound the second term, we will derive an analogous result for Brownian motion in $\R^4$, and then use a strong invariance principle to transfer this result to the random walk, as in \cite{LeGall-Lin16}.
Let $\W=(\W_t, t\ge0)$ denote a standard Brownian motion in $\R^4$ started from the origin. We write $\rho_t:=\|\W_t\|$, for all $t\ge0$. Then $(\rho_t, t\ge0)$ is a 4-dimensional Bessel process started from $0$. The following result is the analog of Lemma 19 in \cite{LeGall-Lin16}. 

\noindent\textbf{Claim.} Let $\delta>0$ and $k\in\n$. For any real number $p\ge 1$, there exists constants $t_0\ge 2$, $\beta>0$ and $C_{\delta, k, p}>0$ such that for any $t\ge t_0$ and $r_t=( \log t)^\beta$,
\begin{equation}\label{BMsumh}
\P\left(\bigg|\int_{\tau_{r_t}^{(\rho)}}^{\tau_t^{(\rho)}}\frac{(\log_+ \rho_s)^k}{\rho_s^2}ds-\frac{1}{k+1}(\log t)^{k+1}\bigg|>\delta (\log t)^{k+1}\right)
\le C_{\delta, k,p}(\log t)^{-p},
\end{equation}
where $\tau^{(\rho)}_{r_t}:=\inf\{t\ge0 \colon \rho_t\ge r_t\}$.

We then deduce \eqref{cvgPsumh} from \eqref{BMsumh} using Zaitsev's extension \cite{Zaitsev} of the well-known Koml\'os-Major-Tusn\'ady strong invariance principle. This means that we approximate $2\J(S_j)$ by $\rho_j$ up to time $n$ with an error term of order $O(\log n)$, with sufficiently high probability. Using the same arguments as in the proof of Lemma 18 in \cite{LeGall-Lin16}, we obtain \eqref{cvgPsumh}.

It remains to verify \eqref{BMsumh}. For any integer $j\in\n$, let 
\[
\gamma_j:=\tau^{(\rho)}_{e^j},\ X_j:=\int_{\gamma_j}^{\gamma_{j+1}}\frac{ds}{\rho_s^2},\ Y_j:=\int_{\gamma_j}^{\gamma_{j+1}}\frac{(\log_+\rho_s)^k}{\rho_s^2}ds,\ I_j:=\inf_{s\ge \gamma_j}\rho_s.
\]
It follows from scaling and the Markov property of the Bessel process that $X_j, j\ge0$ are i.i.d.~random variables. Moreover, we know from the proof of Lemma 19 of \cite{LeGall-Lin16} that 
\begin{equation}\label{exp38}
\E[X_0]=1 \quad \mbox{ and }\quad \E[e^{\frac38 X_0}]=\sqrt{e}<\infty.
\end{equation}
Let $n:=\lfloor \log t\rfloor$ and $m:=\lfloor \log r_t \rfloor=\lfloor \beta \log \log t \rfloor$. We will only consider $t$ sufficiently large so that $t>r_t$. Observe that
\begin{equation}\label{sumJ}
\sum_{j=m}^n j^k =\frac{n^{k+1}}{k+1}(1+o_n(1))= \frac{(\log t)^{k+1}}{k+1}(1+o_t(1)),
\end{equation}
and
\begin{equation}\label{sumY}
\sum_{j=m+1}^n Y_j\le \int_{\tau_{r_t}^{(\rho)}}^{\tau_t^{(\rho)}}\frac{(\log_+ \rho_s)^k}{\rho_s^2}ds\le \sum_{j=m}^{n+1}Y_j,
\end{equation}
where $Y_j$ can be approximated by $j^k X_j$ with high probability. In fact, it is known (see Ex 1.18 of Chapter XI in \cite{RY99}) that $I_j e^{-j}$ is distributed as the square root of a uniform random variable on $(0,1)$. In other words, 
\begin{equation}\label{infBessel}
\P(I_j\le e^j s)= s^2, \forall s\in [0,1].
\end{equation}
If we take $\beta=\frac{p}{2\varepsilon}$, then there exists some constant $C_\varepsilon>0$ such that 
\[
  \P\bigg( \bigcup_{j=m}^n \{I_j\le e^{j-\varepsilon j} \}\bigg)\leq  \sum_{j= m}^n \P(I_j\le e^{j-\varepsilon j})= \sum_{j= m}^n e^{-2\varepsilon j} \leq C_\varepsilon (\log t)^{-p}.
\]

Now notice that
\begin{align*}
&\P\bigg(\Big|\sum_{j=m}^n Y_j -\sum_{j=m}^n j^k\Big|\ge \delta \sum_{j=m}^n j^k\bigg) \\
\le\; & \P\bigg( \bigcup_{j=m}^n \{I_j\le e^{j-\varepsilon j} \}\bigg) +\P\bigg(I_j> e^{j-\varepsilon j}, \forall m \le j\le n;  \Big|\sum_{j=m}^n Y_j - \sum_{j=m}^n j^k \Big|\geq \delta \sum_{j=m}^n j^k\bigg). 
\end{align*}
On the event $\bigcap_{m \le j\le n}\{I_j> e^{j-\varepsilon j}\}$, we have $\rho_s>e^{j-\varepsilon j}$ for all $s\in [\gamma_j,\gamma_{j+1}]$, so that $Y_j\geq j^k X_j(1-\varepsilon)^k$. At the same time, we always have $Y_j\leq (j+1)^k X_j$. It follows that 
\begin{align*}
&\P\bigg(I_j> e^{j-\varepsilon j}, \forall m \le j\le n;  \Big|\sum_{j=m}^n Y_j - \sum_{j=m}^n j^k \Big|\geq \delta \sum_{j=m}^n j^k\bigg)\\
\leq \; & \P\bigg( \sum_{j=m}^n  (j+1)^k X_j \geq (1+\delta)\sum_{j=m}^n j^k \bigg) +\P\bigg((1-\varepsilon)^k\sum_{j=m}^n j^k X_j \leq (1-\delta) \sum_{j=m}^n j^k\bigg) \\
\leq \; & \P\bigg( \sum_{j=m}^n  j^k X_j \geq \frac{(1+\delta)}{(1+1/m)^k}\sum_{j=m}^n j^k \bigg) +\P\bigg(\sum_{j=m}^n j^k X_j \leq \frac{(1-\delta)}{(1-\varepsilon)^k} \sum_{j=m}^n j^k\bigg).
\end{align*}
For $m=\lfloor \beta \log \log t \rfloor$ with sufficiently large $t$, by taking $\varepsilon=\frac{\delta}{k+1}$, we get
\[
  \frac{(1+\delta)}{(1+1/m)^k}\geq 1+\delta/2 \geq 1+\varepsilon \quad \mbox{ and } \quad \frac{(1-\delta)}{(1-\varepsilon)^k} \leq 1-\varepsilon.
\]
Therefore, for any $\lambda>0$, by independence of $X_j, j\geq 0$,
\begin{align*}
&\P\bigg(I_j> e^{j-\varepsilon j}, \forall m \le j\le n;  \Big|\sum_{j=m}^n Y_j - \sum_{j=m}^n j^k \Big|\geq \delta \sum_{j=m}^n j^k\bigg)\\
\leq \; & \P\bigg( \Big|\sum_{j=m}^n j^k X_j - \sum_{j=m}^n j^k \Big|\geq \varepsilon \sum_{j=m}^n j^k\bigg)\\
\leq \; & e^{-\lambda (1+\epsilon)\sum_{j=m}^n j^k}\prod_{j=m}^n\E[e^{\lambda j^k X_j}]+e^{\lambda(1-\varepsilon)\sum_{j=m}^n j^k}\prod_{j=m}^n\E[e^{-\lambda j^k X_j}].
\end{align*}
Thanks to \eqref{exp38}, we can use the logarithm of the moment generating function of $X_0$ to derive the existance of a constant $C>0$ such that for sufficiently large $n$,
\[
  \E[e^{\frac{X_0}{\sqrt{n}}}]\le e^{\frac{1}{\sqrt{n}}+\frac{C}{n}}\quad \mbox{and} \quad \E[e^{-\frac{X_0}{\sqrt{n}}}]\le e^{-\frac{1}{\sqrt{n}}+\frac{C}{n}}.
\]
Let us take $\lambda=\frac{1}{\sqrt{n}n^k}$. Then
\[
\E[e^{\lambda j^k X_j}]\le \E[e^{\frac{X_0}{\sqrt{n}}}]^{(\frac{j}{n})^k}\le e^{(\frac{1}{\sqrt{n}}+\frac{C}{n})(\frac{j}{n})^k} \quad\mbox{and}\quad \E[e^{-\lambda j^k X_j}]\le e^{(-\frac{1}{\sqrt{n}}+\frac{C}{n})(\frac{j}{n})^k}.
\]
As a consequence, we get that for sufficiently large $n$,
\[
  \P\bigg(I_j> e^{j-\varepsilon j}, \forall m \le j\le n;  \Big|\sum_{j=m}^n Y_j - \sum_{j=m}^n j^k \Big|\geq \delta \sum_{j=m}^n j^k\bigg)\leq 2 e^{-\frac{\varepsilon}{4(k+1)}\sqrt{n}},
\]
which is $o_t(1)(\log t)^{-p}$ when $t\to \infty$. Therefore, for any $\delta>0$, we can find a constant $C_{\delta,k,p}>0$ such that
\[
\P\bigg(\Big|\sum_{j=m}^n Y_j -\sum_{j=m}^n j^k \Big|\ge \delta \sum_{j=m}^n j^k\bigg)
\le  C_{\delta,k,p} (\log t)^{-p}.
\]
Since $\delta$ is arbitrary, in view of \eqref{sumJ} and \eqref{sumY}, we obtain the concentration inequality \eqref{BMsumh} and complete the proof of \eqref{cvgPsumh}.
\end{proof}

In fact, using similar arguments as above, one can also prove that for any $y\in K$, if $\gamma=\frac{\J(x)}{(\log\J(x))^{\frac14}}$ and $r=\J(x)^a$ with some $a\in(0,1)$, then for any $\delta>0$, $p>1$, there exists some constant $C_{\delta, p}>0$ such that
\begin{equation}\label{sumu}
\P_y\left(\Big|\sum_{k=\tau_r}^{\tau_\gamma-1}\frac{1}{\J(S_k)^2\log\J(S_k)}- 4\log \frac1a\Big|\ge \delta\right)\le C_{\delta,p} (\log \J(x))^{-p}
\end{equation}

Now we are ready to prove Lemma \ref{sumh}.
\begin{proof}[Proof of Lemma \ref{sumh}]
Similarly to $\sigma_r(w)$, let $\sigma_r:=\max\{0\le k<T_K \colon S_k\notin \B_r\}$ be the last time that the random walk $\{S_k\}_{k \geq 0}$ stays outside the ball $\B_r$ before hitting $K$. We will take $r=\gamma=\frac{\J(x)}{(\log\J(x))^{\frac14}}$ for $\|x\|\gg1$. Under $\P_x$, on the event $\{T_K<\infty\}$,
\[
\sum_{k=0}^{T_K}h(S_k)=\sum_{k=0}^{\sigma_\gamma}h(S_k)+\sum_{k=\sigma_\gamma+1}^{T_K}h(S_k).
\]
Then it follows from the Markov property that
\begin{align*}
&\E_x\left[\sum_{k=0}^{T_K}h(S_k); T_K<\infty\right]\\
&=\sum_{y\in K}\sum_{z\in\partial_\mu\B_\gamma}\sum_{i=0}^\infty\sum_{j=0}^\infty\E_x\left[\sum_{k=0}^{i}h(S_k)+\sum_{k=i+1}^{i+j}h(S_k); T_K=i+j, S_{i+j}=y, \sigma_\gamma=i, S_i=z\right]\\
&=\sum_{y\in K}\sum_{z\in\partial_\mu\B_\gamma}\sum_{i=0}^\infty\sum_{j=0}^\infty\E_x\left[\sum_{k=0}^{i}h(S_k)\ind{T_K>i, S_i=z}\right]\P_z(T_K=j, S_j=y, \{S_k\}_{1\le k\le j}\subset\B_\gamma)\\
&+\sum_{y\in K}\sum_{z\in\partial_\mu\B_\gamma}\sum_{i=0}^\infty\sum_{j=0}^\infty\P_x(T_K>i, S_i=z)\E_z\left[\sum_{k=1}^j h(S_k)\ind{T_K=j, S_j=y, \{S_k\}_{1\le k\le j}\subset\B_\gamma}\right].
\end{align*}
Recall that $\tau_r=\inf\{n\ge 0\colon S_n\notin\B_r\}$. By time reversal, we observe that
\[
\sum_{j=0}^\infty\P_z(T_K=j, S_j=y, \{S_k\}_{1\le k\le j}\subset\B_\gamma)=\P_y(T_K^+>\tau_\gamma, S_{\tau_\gamma}=z),
\]
and 
\[
\sum_{j=0}^\infty\E_z\left[\sum_{k=1}^j h(S_k)\ind{T_K=j, S_j=y, \{S_k\}_{1\le k\le j}\subset\B_\gamma}\right]=\E_y\left[\sum_{k=0}^{\tau_\gamma-1}h(S_k)\ind{T^+_K>\tau_\gamma, S_{\tau_\gamma}=z}\right].
\]
As a result, 
\begin{equation}\label{sumh2parts}
\E_x\left[\sum_{k=0}^{T_K}h(S_k); T_K<\infty\right]=\Sigma_0(h)+\Sigma_1(h),
\end{equation}
where
\begin{align*}
&\Sigma_0(h):=\sum_{y\in K}\sum_{z\in\partial_\mu\B_\gamma}\sum_{i=0}^\infty\E_x\left[\sum_{k=0}^{i}h(S_k)\ind{T_K>i, S_i=z}\right]\P_y(T_K^+>\tau_\gamma, S_{\tau_\gamma}=z),\\
&\Sigma_1(h):=\sum_{y\in K}\sum_{z\in\partial_\mu\B_\gamma}\sum_{i=0}^\infty\P_x(T_K>i, S_i=z)\E_y\left[\sum_{k=0}^{\tau_\gamma-1}h(S_k)\ind{T^+_K>\tau_\gamma, S_{\tau_\gamma}=z}\right].
\end{align*}
First, note that $\sum_{i=0}^\infty\P_x(T_K>i, S_i=z)\le g(x,z)\le \frac{C}{(1+\J(x))^{2}}$ for all $z\in\partial_\mu\B_\gamma$. Hence, 
\begin{align*}
&\Sigma_1(h)\le  \sum_{y\in K}\frac{C}{(1+\J(x))^{2}}\E_y\left[\sum_{k=0}^{\tau_\gamma-1}h(S_k)\ind{T^+_K>\tau_\gamma}\right]\\
&\le \frac{C}{(1+\J(x))^{2}}\sum_{y\in K}\E_y\left[\Big|\sum_{k=0}^{\tau_\gamma-1}h(S_k)-\frac{4c_h}{k+1}(\log \gamma)^{k+1}\Big|\ind{T^+_K>\tau_\gamma}\right]\\
& \quad + \frac{C|K|}{(1+\J(x))^{2}}\frac{4c_h}{k+1}(\log \gamma)^{k+1}.
\end{align*}
Here, we notice that for any $\delta>0$,
\begin{multline*}
\E_y\left[\Big|\sum_{k=0}^{\tau_\gamma-1}h(S_k)-\frac{4c_h}{k+1}(\log \gamma)^{k+1}\Big|\ind{T^+_K>\tau_\gamma}\right]\le \delta(\log \gamma)^{k+1}+\\
\qquad \E_y\left[\Big|\sum_{k=0}^{\tau_\gamma-1}h(S_k)-\frac{4c_h}{k+1}(\log \gamma)^{k+1}\Big|\ind{|\sum_{k=0}^{\tau_\gamma-1}h(S_k)-\frac{4c_h}{k+1}(\log \gamma)^{k+1}|\ge \delta(\log\gamma)^{k+1}}\right],
\end{multline*}
By the Cauchy-Schwarz inequality, the last expectation on the right-hand side is bounded by
{\scriptsize
\[
\E_y\left[\Big(\sum_{k=0}^{\tau_\gamma-1}h(S_k)-\frac{4c_h}{k+1}(\log \gamma)^{k+1}\Big)^2\right]^{\frac12}\P_y\left(\Big|\sum_{k=0}^{\tau_\gamma-1}h(S_k)-\frac{4c_h}{k+1}(\log \gamma)^{k+1}\Big|\ge \delta(\log\gamma)^{k+1}\right)^{\frac12}.
\]
}
Applying Lemma \ref{sumhmomp} with $p=2$, we obtain that for any $\delta>0$, $y\in K$, 
\[
	\E_y\left[\Big|\sum_{k=0}^{\tau_\gamma-1}h(S_k)-\frac{4c_h}{k+1}(\log \gamma)^{k+1}\Big|\ind{T^+_K>\tau_\gamma}\right]\le \delta(\log \gamma)^{k+1}+ C (\log\gamma)^k.
\]
If $\gamma$ is sufficiently large, we also have the estimate
\begin{equation}\label{centralsumh}
\E_y\left[\Big|\sum_{k=0}^{\tau_\gamma-1}h(S_k)-\frac{4c_h}{k+1}(\log \gamma)^{k+1}\Big|\ind{T^+_K>\tau_\gamma}\right]\le C \delta(\log\gamma)^{k+1}.
\end{equation}
Therefore, for all $x\in \Z^4$,
\begin{equation}\label{eq:Sigma1h-upperbound}
\Sigma_1(h)\le \frac{C(\log (2+\J(x)))^{k+1}}{(1+\J(x))^{2}}.
\end{equation}
Moreover, by Lemma \ref{fargreen} and \eqref{green}, $\sum_{i=0}^\infty\P_x(T_K>i, S_i=z)=\frac{\c_4}{\J(x)^2}(1+o_x(1))$ for $z\in\partial_\mu\B_\gamma$ and $\|x\|\gg1$. Then, as $\|x\|\to \infty$, 
\[
	\Sigma_1(h)= \frac{\c_4}{\J(x)^2}(1+o_x(1))\sum_{y\in K}\E_y\left[\sum_{k=0}^{\tau_\gamma-1}h(S_k)\ind{T^+_K>\tau_\gamma}\right].
\]
Using \eqref{centralsumh} and taking $\delta\to 0$, we have
\[
	\Sigma_1(h)= \frac{\c_4}{\J(x)^2}(1+o_x(1))\sum_{y\in K}\frac{4c_h}{k+1}(\log \gamma)^{k+1}\P_y(T^+_K>\tau_\gamma). 
\]
When $\|x\|\to \infty$, as $\P_y(T^+_K>\tau_\gamma)\to \P_y(T^+_K=\infty)$, we deduce that
\begin{equation}\label{eq:Sigma1h-asymp}
\Sigma_1(h)=\frac{4\c_4c_h}{(k+1)\J(x)^2}(\log \J(x))^{k+1}\cap(K) (1+o_x(1)).
\end{equation}
It remains to bound $\Sigma_0(h)$. In fact,
\begin{align*}
\Sigma_0(h)\le &\sum_{y\in K}\sum_{z\in \partial_\mu\B_\gamma}\sum_{i=0}^\infty \E_x\left[\sum_{k=0}^i h(S_k)\ind{S_i=z} \right]\P_y(S_{\tau_\gamma}=z, \tau_\gamma<T^+_K)\\
=&\sum_{y\in K}\sum_{z\in\partial_\mu\B_\gamma}\sum_{w\in\Z^4}\sum_{k=0}^\infty\sum_{i=k}^\infty h(w)\P_x(S_k=w, S_i=z)\P_y(S_{\tau_\gamma}=z, \tau_\gamma<T^+_K)\\
=&\sum_{y\in K}\sum_{z\in\partial_\mu\B_\gamma}\sum_{w\in\Z^4}h(w)g(x,w)g(w,z)\P_y(S_{\tau_\gamma}=z, \tau_\gamma<T^+_K).
\end{align*}
It then follows from \eqref{hgg4d} that
\begin{align}
\Sigma_0(h)\le & \,C\frac{(\log(2+\J(x)))^k\log\log(2+\J(x))}{(1+\J(x))^2} \sum_{y\in K}\sum_{z\in\partial_\mu\B_\gamma}\P_y(S_{\tau_\gamma}=z, \tau_\gamma<T^+_K)\nonumber \\
\le & \,C\frac{(\log(2+\J(x)))^k\log\log(2+\J(x))}{(1+\J(x))^2}|K|. \label{eq:Sigma0h}
\end{align}
Going back to \eqref{sumh2parts}, and comparing \eqref{eq:Sigma0h} with \eqref{eq:Sigma1h-upperbound} and \eqref{eq:Sigma1h-asymp}, we thus obtain 
\[
\E_x\left[\sum_{k=0}^{T_K}h(S_k); T_K<\infty\right]=\Sigma_0(h)+\Sigma_1(h)\le \frac{C(\log (2+\J(x)))^{k+1}}{(1+\J(x))^{2}}, \forall x\in\Z^4,
\]
and as $\|x\|\to \infty$,
\[
\E_x\left[\sum_{k=0}^{T_K}h(S_k); T_K<\infty\right]=\frac{4\c_4c_h}{(k+1)\J(x)^2}(\log \J(x))^{k+1}\cap(K) (1+o_x(1)).
\]
Thanks to \eqref{hittingprobabS-2}, this completes the proof of Lemma \ref{sumh}.
\end{proof}

\section{Appendix: Proof of Proposition \ref{moments} and Lemma \ref{mainHx}}\label{AppendixB}

\begin{proof}[Proof of Proposition \ref{moments}]
We will study the moments of $L_K$ and $Z_\T(K)$ recursively. Recall that $\ell_j(x,K)=\E_x[L_K^j]$ and $m_j(x,K)=\E_x[Z_\T(K)^j]$ for any $j\in\n$. The claimed results for $j=1,2$ have been shown in \eqref{meanL4d}, \eqref{2momL} and \eqref{2momZ}.

Take some integer $n\ge2$, and let us work under the assumption that $\sum_k k^{n+1}p_k<\infty$. Then, obviously, for all $j\in\{1,\cdots, n\}$, $\sum_k k^j p_k<\infty$. Assume that \eqref{momL} and \eqref{momZ} hold for all $1\leq j\leq n$. We also assume that for all $1\leq j\leq n$, there exist constants $\C_j>0$ such that
\begin{align}
\ell_j (x,K)\le & \,\C_j j!\frac{(\log(2+\J(x)))^{j-1}}{(1+\J(x))^2}, \forall x\in\Z^4;\label{bdmomL}\\
m_j(x,K)\le & \,\C_j j!\frac{(\log(2+\J(x)))^{j-1}}{(1+\J(x))^2}, \forall x\in\Z^4.\label{bdmomZ}
\end{align}
We are going to prove that the convergences \eqref{momL} and \eqref{momZ}, together with the upper bounds \eqref{bdmomL} and \eqref{bdmomZ}, hold for $j=n+1$.

First, observe that
{\small
\begin{align*}
\ell_{n+1}(x,K)=&\,\E_x\left[\bigg(\sum_{u\in\T}\ind{u\in\L_K}\bigg)^{n+1}\right]=\E_x\left[\sum_{u_1,\cdots, u_{n+1}\in\T}\ind{u_1,\cdots, u_{n+1}\in\L_K}\right],\\
m_{n+1}(x,K)=&\,\E_x\left[\bigg(\sum_{u\in\T}Z_{\T_u}(K)\ind{u\in\L_K}\bigg)^{n+1}\right]=\E_x\left[\sum_{u_1,\cdots,u_{n+1}\in\T}\prod_{k=1}^{n+1}(Z_{\T_{u_k}}(K)\ind{u_k\in \L_K})\right].
\end{align*}
}
In the following, we will only focus on $\ell_{n+1}(x,K)$ as $m_{n+1}(x,K)$ can be dealt with similarly.
There are two cases for $u_1,\cdots, u_{n+1}\in\L_K$: 
\begin{itemize}
\item either $u_1=u_2=\cdots=u_{n+1}$;
\item or $\{u_1,\cdots, u_{n+1}\}$ contains at least two distinct individuals.
\end{itemize}
In the second case, we consider the most recent common ancestor $\omega$ of $u_1,\cdots, u_{n+1}$ in $\T$. Let $N_\omega$ be the number of children of $\omega$. Among the $N_\omega=m$ children of $\omega$, if there are exactly $r$ children having descends in $\{u_1,\cdots,u_{n+1}\}$, we must have $2\le r\le m\wedge (n+1)$ and denote the $r$ children by $\omega i_1,\cdots, \omega i_r$.
We write $\Xi_{r,m,n+1}=\aleph_{r,m}\times \Upsilon_{r,n+1}$ where 
\begin{align*}
\aleph_{r,m}:=&\{(i_1,\cdots, i_r)\in \n^r\vert 1\le i_1<\cdots<i_r\le m\},\\
\Upsilon_{r,n+1}:=&\Big\{(k_1,\cdots, k_r)\in(\n^*)^r \Big\vert \sum_{j=1}^r k_j=n+1\Big\}.
\end{align*}
It then follows that
{\small
\begin{align*}
&\ell_{n+1}(x,K)=  \E_x\left[\sum_{u\in\T}\ind{u\in\L_K}\right]\\
&+\E_x\left[\sum_{\omega\in\T}\ind{\omega<\L_K}\sum_{m=2}^\infty \ind{N_\omega=m}\sum_{r=2}^{m\wedge(n+1)}\sum_{\Xi_{r,m, n+1}}\frac{(n+1)!}{k_1!\cdots k_r!}\prod_{j=1}^r \bigg(\sum_{u\in\T_{\omega i_j}}\ind{u\in\L_K}\bigg)^{k_j}\right]\\
&=\ell_1(x,K)+\sum_{r=2}^{n+1}\E_x\left[\sum_{\omega\in\T}\ind{\omega<\L_K}\sum_{m=r}^\infty \ind{N_\omega=m}\sum_{\Xi_{r,m, n+1}}\frac{(n+1)!}{k_1!\cdots k_r!}\prod_{j=1}^r \ell_{k_j}(S_{\omega i_j},K)\right].
\end{align*}
}
For the last sum on the right-hand side, the term $r=2$ gives the main contribution while the terms $r\ge 3$ are all negligeable. To treat them separately, for any integer $r\in[2,n+1]$, we set
{\small
\begin{equation}\label{rsubtrees}
\E_\eqref{rsubtrees}(r):=\E_x\bigg[\sum_{\omega\in\T}\ind{\omega<\L_K}\sum_{m=r}^\infty \ind{N_\omega=m}\sum_{\Xi_{r,m, n+1}}\frac{(n+1)!}{k_1!\cdots k_r!}\prod_{j=1}^r \ell_{k_j}(S_{\omega i_j},K)\bigg].
\end{equation}
}
For $r=2$, it follows from the many-to-one lemma that
{\small
\begin{align*}
&\E_\eqref{rsubtrees}(2)\\
&=\sum_{k=0}^\infty \E_x\bigg[\ind{T_K>k}\sum_{m\ge 2}p_m{m\choose 2}\sum_{j=1}^n {n+1\choose j} \sum_e \ell_j(S_k+e, K)\mu(e) \sum_{e'}\ell_{n+1-j}(S_k+e',K)\mu(e')\bigg]\\
&= \frac{\sigma^2(n+1)!}{2}\sum_{j=1}^n \sum_{k=0}^\infty \E_x\left[\ind{T_K>k}F_{j,n+1-j}(S_k)\P_{S_k}(T_K<\infty)\right]\\
&=\frac{\sigma^2(n+1)!}{2}\sum_{j=1}^n\sum_{k=0}^\infty \E_x\left[\ind{T_K>k}F_{j,n+1-j}(S_k)\ind{T_K<\infty}\right],
\end{align*}
}
where the last equality comes from the Markov property and for all $z\notin K$,
\[
F_{j,n+1-j}(z):=\frac{\sum_e \ell_j(z+e, K)\mu(e)}{j!}\frac{\sum_{e'} \ell_{n+1-j}(z+e', K)\mu(e')}{(n+1-j)!} \frac{1}{\P_z(T_K<\infty)}.
\]
Using the assumptions for $j$ and $n+1-j\in\{1,\cdots,n\}$, together with \eqref{hittingprobabS-2} and \eqref{hittingprobab4d}, one sees that as $\| x\|\to\infty$,
\[
\frac{\sigma^2}{2}F_{j,n+1-j}(x)\sim (2\c_4\sigma^2\cap(K))^{n}\frac{(\log \J(x))^{n-1}}{4\J(x)^2},
\]
We hence deduce from Lemma \ref{sumh} and \eqref{hittingprobabS-2} that
\begin{align*}
\E_\eqref{rsubtrees}(2)=& \frac{\sigma^2(n+1)!}{2}\sum_{j=1}^n\E_x\left[\sum_{k=0}^{T_K-1}F_{j,n+1-j}(S_k)\vert T_K<\infty\right]\P_x(T_K<\infty)\\
=&(n+1)!\sum_{j=1}^n \frac{(2\c_4\sigma^2\cap(K))^{n}}{n}(\log\J(x))^n \frac{\c_4\cap(K)}{\J(x)^2}(1+o_x(1))\\
=&(n+1)!(2\c_4\sigma^2\cap(K))^{n} \frac{\c_4\cap(K)(\log \J(x))^{n}}{\J(x)^2}(1+o_x(1)),
\end{align*}
which suffices to conclude that 
\[
\E_\eqref{rsubtrees}(2)\sim (n+1)! (2\c_4\sigma^2\cap(K))^{n+1} (\log\J(x))^{n+1}u_K(x).
\]
On the other hand, for $r\ge 3$, similarly as above,
\begin{align*}
&\sum_{r=3}^{n+1}\E_\eqref{rsubtrees}(r)\\
&=\sum_{r=3}^{n+1}\sum_{m\ge r}\frac{p_m m!}{r!(m-r)!} \sum_{k\ge0}\E_x\bigg[\ind{T_K>k}\sum_{\Upsilon_{r,n+1}}\frac{(n+1)!}{k_1!\cdots k_r!}\prod_{j=1}^r\Big(\sum_e\ell_{k_j}(S_k+e,K)\mu(e)\Big)\bigg]\\
&=\sum_{r=3}^{n+1}\sum_{m\ge r}\frac{p_m m!}{r!(m-r)!}\E_x\left[\sum_{k=0}^{T_K-1}G_{r,n+1}(S_k);T_K<\infty\right],
\end{align*}
where
\[
G_{r,n+1}(z):=\frac{(n+1)!}{\P_z(T_K<\infty)}\sum_{\Upsilon_{r,n+1}}\prod_{j=1}^r\left(\frac{\sum_e\ell_{k_j}(z+e,K)\mu(e)}{k_j!}\right), \forall z\notin K.
\]
Under the assumption \eqref{bdmomL}, one see that 
\[
G_{r,n+1}(z)\le C (n+1)! \frac{(\log(2+\J(z)))^{n-2}}{(1+\J(z))^2}.
\]
Meanwhile, we also have
\[
\sup_{3\le r\le n+1}\sum_{m\ge r}\frac{p_m m!}{r!(m-r)!}\le \sum_m p_m m^{n+1}<\infty.
\]
By \eqref{bdsumh}, we obtain that
\[
\sum_{r=3}^{n+1}\E_\eqref{rsubtrees}(r)\le C (n+1)! \frac{(\log(2+\J(x)))^{n-1}}{(1+\J(x))^2}.
\]
Therefore, by recurrence, we derive Proposition \ref{moments}.
\end{proof}

\begin{proof}[Proof of Lemma \ref{mainHx}]
Let us first prove \eqref{smallpartcvgH2}. Denote the expectation on the left-hand side of \eqref{smallpartcvgH2} by $\E_\eqref{smallpartcvgH2}$. Recalling \eqref{uK}, one has
\begin{align*}
&\,u_K(x)-\E_\eqref{smallpartcvgH2}\\
=&\,\E_{\Q^*_x}\left[\ind{T_K(w)<\infty}\prod_{k=1}^{T_K(w)}\prod_{u\in\L(w_k)}\ind{Z_{\T_u}(K)=0}\times\ind{\sum_{k=1}^{\sigma_\gamma(w)}\sum_{u\in\rt(w_k)} Z_{\T_u}(K)=0}\right]\\
=& \sum_{i=0}^\infty\sum_{j=0}^\infty\E_{\Q^*_x}\left[\begin{array}{r} \ind{T_K(w)=i+j, \sigma_\gamma(w)=i}\prod_{k=1}^{\sigma_\gamma(w)}\prod_{u\in\L(w_k)\cup\rt(w_k)}(1-u_K(S_u))\\
\times \prod_{k=1+\sigma_\gamma(w)}^{T_K(w)}\prod_{u\in\L(w_k)}(1-u_K(S_u))\end{array}\right].
\end{align*}
By the Markov property at time $i$, we get that
\begin{align*}
&u_K(x)-\E_\eqref{smallpartcvgH2}\\
=&\sum_{z\in\partial_\mu\B_\gamma}\sum_{i=0}^\infty\E_{\Q^*_x}\left[\ind{T_K(w)>i, S_i=z}\prod_{k=1}^{i}(1-\sum_eu_K(S_{w_{k-1}}+e)\mu(e))^{L(w_k)+R(w_k)}\right]\\
&\qquad\times\E_{\Q^*_z}\Bigg[\ind{T_K(w)<\infty, S_{w_k}\in\B_\gamma, \forall 1\le k\le T_K(w)}\prod_{k=1}^{T_K(w)}(1-\sum_eu_K(S_{w_{k-1}}+e)\mu(e))^{L(w_k)}\Bigg].
\end{align*}
Lemma 3.1 of \cite{Zhu} shows that for any $z\in\partial_\mu\B_\gamma$, 
\[
\sum_{i=0}^\infty\E_{\Q^*_x}\left[\ind{T_K(w)>i, S_i=z}\prod_{k=1}^{i}(1-\sum_eu_K(S_{w_{k-1}}+e)\mu(e))^{L(w_k)}\right]=(1+o_x(1))g(x,z).
\]
This can be directly extended to the following estimate:
\[
\sum_{i=0}^\infty\E_{\Q^*_x}\left[\ind{T_K(w)>i, S_i=z}\prod_{k=1}^{i}(1-\sum_eu_K(S_{w_{k-1}}+e)\mu(e))^{L(w_k)+R(w_k)}\right]=(1+o_x(1))g(x,z),
\]
where $g(x,z)=(1+o_x(1))\frac{\c_4}{\J(x)^2}$ as $z\in\partial_\mu\B_\gamma$. It implies that
\begin{align*}
&u_K(x)-\E_\eqref{smallpartcvgH2}\\
=&(1+o_x(1))\frac{\c_4}{\J(x)^2}\sum_{z\in\partial_\mu\B_\gamma} \E_{\Q^*_z}\left[\begin{array}{l} \ind{T_K(w)<\infty, S_{w_k}\in\B_\gamma, \forall 1\le k\le T_K(w)}\\
\times \prod_{k=1}^{T_K(w)}(1-\sum_eu_K(S_{w_{k-1}}+e)\mu(e))^{L(w_k)}\end{array}\right].
\end{align*}
Using the same arguments, we also have
\begin{align*}
&u_K(x)\\
=&(1+o_x(1))\frac{\c_4}{\J(x)^2}\sum_{z\in\partial_\mu\B_\gamma} \E_{\Q^*_z}\left[\begin{array}{l} \ind{T_K(w)<\infty, S_{w_k}\in\B_\gamma, \forall 1\le k\le T_K(w)} \\
\times \prod_{k=1}^{T_K(w)}(1-\sum_eu_K(S_{w_{k-1}}+e)\mu(e))^{L(w_k)} \end{array}\right].
\end{align*}
We thus get \eqref{smallpartcvgH2} stating that $\E_\eqref{smallpartcvgH2}=o_x(1)u_K(x)$.

Next, we prove \eqref{smallpartcvgH1}. Denote the expectation on the left-hand side of \eqref{smallpartcvgH1} by $\E_\eqref{smallpartcvgH1}$. Similarly as above, one sees that
\begin{align*}
\E_\eqref{smallpartcvgH1}=&(1+o_x(1))\frac{\c_4}{\J(x)^2}\sum_{z\in\partial_\mu\B_\gamma} \E_{\Q^*_z}\Bigg[\ind{T_K(w)<\infty, S_{w_k}\in\B_\gamma, \forall 1\le k\le T_K(w)}\\
&\times\prod_{k=1}^{T_K(w)}(1-\sum_eu_K(S_{w_{k-1}}+e)\mu(e))^{L(w_k)}\ind{\min_{1\le k\le \sigma_{\J(x)^{a+\delta}}(w)}\J(S_{w_k})\le \J(x)^a}\Bigg].
\end{align*}
By time reversal, we get that
\begin{align*}
\E_\eqref{smallpartcvgH1}=&(1+o_x(1))\frac{\c_4}{\J(x)^2} \sum_{y\in K}\E_y\left[\begin{array}{l} \ind{\tau_\gamma<T^+_K}\prod_{k=0}^{\tau_\gamma-1}(1-\sum_e u_K(S_k+e)\mu(e))^{L_k}\\
\times\ind{\min_{\tau_{\J(x)^{a+\delta}}\le k<\tau_\gamma}\J(S_k)\le \J(x)^a}\end{array}\right]\\
\le& \frac{C}{(1+\J(x))^2}\sum_{y\in K}\P_y\bigg(\min_{\tau_{\J(x)^{a+\delta}}\le k<\tau_\gamma}\J(S_k)\le \J(x)^a\bigg).
\end{align*}
Here we are going to use again Zaitsev's extension \cite{Zaitsev} of the Koml\'os-Major-Tusn\'ady strong invariance principle, now up to time $\gamma^{2+\delta}$, since we know that for any $\delta>0$,
\[
\P_y(\tau_\gamma\ge \gamma^{2+\delta})\le C (\log \J(x))^{-p},
\]
for any $p>1$ and $\|x\|\gg1$. Let us keep the same notation introduced in the proof of \eqref{cvgPsumh}. It then follows that
\begin{align*}
&\P_y\left(\min_{\tau_{\J(x)^{a+\delta}}\le k<\tau_\gamma}\J(S_k)\le \J(x)^a\right)\le C(\log \J(x))^{-p}\\
&+\P_y\left(\max_{1\le k\le \gamma^{2+\delta}}\|\Gamma^{-1/2}S_k-\W_k\|\ge C\log \gamma\right)+\P\left(\inf_{t\ge \tau_{\J(x)^{a+\delta/2}}^{(\rho)}}\rho_t\le \J(x)^{a+\delta/4}\right).
\end{align*}
According to \cite{Zaitsev}, the first probability on the right-hand side is bounded from above by some positive power of $\J(x)^{-1}$. Meanwhile, the second probability is equal to $\J(x)^{-\delta/2}$ by \eqref{infBessel}.
By taking $p>1$, we deduce that
\[
\E_\eqref{smallpartcvgH1}=o_x(1) u_K(x).
\]

It remains to prove \eqref{cvgpositionH}. Denote the expectation on the left-hand side of \eqref{cvgpositionH} by $\E_\eqref{cvgpositionH}$. Again, using similar arguments as above, one has
\begin{align*}
&\E_\eqref{cvgpositionH}\\
=&(1+o_x(1))\frac{\c_4}{\J(x)^2} \sum_{y\in K}\E_y\left[\begin{array}{l} \ind{\tau_\gamma<T^+_K}\prod_{k=0}^{\tau_\gamma-1}(1-\sum_e u_K(S_k+e)\mu(e))^{L_k}\\
\times \prod_{k=\tau_r}^{\tau_\gamma-1}(1-\sum_e u_K(S_k+e)\mu(e))^{R_k} \end{array}\right].
\end{align*}
We write $h(S_k):=\sum_e u_K(S_k+e)\mu(e)$. Using the fact that $1-s=e^{-(1+o(1))s}$ for $s\to 0+$ and \eqref{smallpartcvgH1}, one sees that
\begin{align*}
\E_\eqref{cvgpositionH}=&o_x(1)u_K(x)+
(1+o_x(1))\frac{\c_4}{\J(x)^2} \sum_{y\in K}\E_y\Bigg[\ind{\tau_\gamma<T^+_K}\prod_{k=0}^{\tau_r-1}(1-h(S_k))^{L_k}\\
&\times\exp\bigg(-(1+o_x(1))\sum_{k=\tau_r}^{\tau_\gamma-1}h(S_k)(L_k+R_k)\bigg)\ind{\min_{\tau_r\le k<\tau_\gamma}\J(S_k)\ge r^{1-\delta}}\Bigg].
\end{align*}
Note that $(L_k,R_k)_{k\ge1}$ are i.i.d.~random vectors which are independent of the random walk $(S_k)_{k\ge0}$. So, 
\begin{align*}
&\E\left[\exp\bigg(-(1+o_x(1))\sum_{k=\tau_r}^{\tau_\gamma-1}h(S_k)(L_k+R_k)\bigg)\bigg\vert S_k, k\ge0; L_k, k<\tau_r\right]\\
=&\prod_{k=\tau_r}^{\tau_\gamma-1}\E\big[\exp(-(1+o_x(1))h(S_k) (L_k+R_k))\vert S_k \big].
\end{align*}
When $\J(S_k)\ge r^{1-\delta}$, 
\[
h(S_k)\le \sup_{\J(y) \ge r^{1-\delta}, e\in \supp(\mu)}u_K(y+e)=o_x(1).
\]
Given a positive deterministic $h_x=o_x(1)$, for all $h\in[0,h_x]$ we have uniformly
\[
\E[e^{-h(L_k+R_k)}]=e^{-(1+o_x(1))h\E[L_k+R_k]}=e^{-(1+o_x(1))h \sigma^2}.
\]
Moreover, on the event $\{\min_{\tau_r\le k<\tau_\gamma}\J(S_k)\ge r^{1-\delta}\}$, 
\[
 	h(S_k) = \sum_e u_K(S_k+e)\mu(e) = \frac{1+o_x(1)}{2\sigma^2\J(S_k)^2\log\J(S_k)}
\] 
for $\tau_r\le k\le \tau_{\gamma}-1$.
Putting all these estimates together, we arrive at
\begin{align*}
\E_\eqref{cvgpositionH}=&o_x(1)u_K(x)+
(1+o_x(1))\frac{\c_4}{\J(x)^2} \sum_{y\in K}\E_y\Bigg[\ind{\tau_\gamma<T^+_K}\prod_{k=0}^{\tau_r-1}(1-\sum_e u_K(S_k+e)\mu(e))^{L_k}\\
&\times\exp\bigg(-(1+o_x(1))\sum_{k=\tau_r}^{\tau_\gamma-1}\sigma^2 \frac{1+o_x(1)}{2\sigma^2\J(S_k)^2\log\J(S_k)}\bigg)\ind{\min_{\tau_r\le k<\tau_\gamma}\J(S_k)\ge r^{1-\delta}}\Bigg].
\end{align*}
Using again the fact that 
\[
	\P_y\Big(\min_{\tau_r\le k<\tau_\gamma}\J(S_k)< r^{1-\delta}\Big)=o\Big(\frac{1}{\log \J(x)}\Big),
\]
we get
\begin{align*}
\E_\eqref{cvgpositionH}=&o_x(1)u_K(x)+\\
&(1+o_x(1))\frac{\c_4}{\J(x)^2} \sum_{y\in K}\E_y\left[\begin{array}{l} \ind{\tau_\gamma<T^+_K}\prod_{k=0}^{\tau_r-1}(1-\sum_e u_K(S_k+e)\mu(e))^{L_k}\\
\times\exp\bigg(-(1+o_x(1))\sum_{k=\tau_r}^{\tau_\gamma-1}\sigma^2 \frac{1+o_x(1)}{2\sigma^2\J(S_k)^2\log\J(S_k)}\bigg)\end{array}\right].
\end{align*}
Applying the same arguments for $u_K(x)$, one sees that 
\begin{align*}
u_K(x)=&o_x(1)u_K(x)+\\
&(1+o_x(1))\frac{\c_4}{\J(x)^2} \sum_{y\in K}\E_y\left[\begin{array}{l} \ind{\tau_\gamma<T^+_K}\prod_{k=0}^{\tau_r-1}(1-\sum_e u_K(S_k+e)\mu(e))^{L_k} \\
\times \exp\bigg(-(1+o_x(1))\sum_{k=\tau_r}^{\tau_\gamma-1}\frac{\sigma^2/2}{2\sigma^2\J(S_k)^2\log\J(S_k)}\bigg)\end{array}\right].
\end{align*}
Using \eqref{sumu} for $\E_\eqref{cvgpositionH}$ and approximating 
\[
\exp\Big(-(1+o_x(1))\sum_{k=\tau_r}^{\tau_\gamma-1}\sigma^2 \frac{1+o_x(1)}{2\sigma^2\J(S_k)^2\log\J(S_k)}\Big)
\]
by 
\begin{align*}
\exp\Big(-(1+o_x(1))\frac{\sigma^2/2}{2\sigma^2} 4\log\frac1a-(1+o_x(1))\sum_{k=\tau_r}^{\tau_\gamma-1}\frac{\sigma^2/2}{2\sigma^2\J(S_k)^2\log\J(S_k)}\Big),
\end{align*}
we therefore deduce that 
\begin{align*}
&\E_\eqref{cvgpositionH}=o_x(1)u_K(x)+\\
&(1+o_x(1))\frac{\c_4}{\J(x)^2} \sum_{y\in K}\E_y\Bigg[\begin{array}{l} \ind{\tau_\gamma<T^+_K}\prod_{k=0}^{\tau_r-1}(1-\sum_e u_K(S_k+e)\mu(e))^{L_k}\\
\times a^{1+o_x(1)}\exp\bigg(-(1+o_x(1))\sum_{k=\tau_r}^{\tau_\gamma-1}\frac{\sigma^2/2}{2\sigma^2\J(S_k)^2\log\J(S_k)}\bigg) \end{array}\Bigg]\\
=&(a+o_x(1))u_K(x).
\end{align*}
This ends the proof of \eqref{cvgpositionH}.
\end{proof}

\section*{Acknowledgment}
We are deeply grateful to the anonymous referee for his/her thorough reading of the manuscript and for many valuable comments and suggestions, which have helped to improve the presentation of this work.

X.~C.~is supported by National Key R\&D Program of China (No.~2022YFA1006500). This work was supported by the National Natural Science Foundation of China (Grant No. 12571148) and by the Fundamental Research Funds for the Central Universities (No. 310432104). S.~L.~expresses his gratitude to School of Mathematical Sciences, Beijing Normal University for its hospitality during his research visit.

\bibliographystyle{alpha}

\end{document}